\documentclass[11pt]{article}
\usepackage{geometry}
\usepackage{pst-plot, pstricks}
\usepackage{hyperref}
\usepackage[utf8]{inputenc}
\usepackage{amsmath}
\usepackage{amssymb}
\usepackage{amscd}
\usepackage{textcomp}

\usepackage{amsthm}
\theoremstyle{plain}
\newtheorem{theorem}{Theorem}[section]

\newtheorem{lemma}[theorem]{Lemma}

\newtheorem{corollary}[theorem]{Corollary}
\newtheorem{proposition}[theorem]{Proposition}

\theoremstyle{definition}
\newtheorem{definition}[theorem]{Definition}

\newtheorem{example}[theorem]{Example}
\newtheorem{remark}[theorem]{Remark}

\theoremstyle{remark}

\usepackage{xy}
\xyoption{all}

\newdir^{ (}{{}*!/-3pt/\dir^{(}}    
\newdir^{  }{{}*!/-3pt/\dir^{}}    
\newdir_{ (}{{}*!/-3pt/\dir_{(}}    
\newdir_{  }{{}*!/-3pt/\dir_{}}

\newcounter{zahl}

\def\theenumi{(\alph{enumi})}

\def\p@enumii{\theenumi}

\newcommand{\DS}{\displaystyle}
\newcommand{\TS}{\textstyle}
\newcommand{\SC}{\scriptstyle}
\newcommand{\SSC}{\scriptscriptstyle}

\DeclareMathOperator{\Ad}{Ad}
\DeclareMathOperator{\Aut}{Aut}

\DeclareMathOperator{\End}{End}

\DeclareMathOperator{\Ext}{Ext}
\DeclareMathOperator{\Frob}{Frob}
\DeclareMathOperator{\Gal}{Gal}
\DeclareMathOperator{\GL}{GL}

\DeclareMathOperator{\Koh}{H}

\DeclareMathOperator{\Hom}{Hom}

\DeclareMathOperator{\PGL}{PGL}

\DeclareMathOperator{\QHom}{QHom}
\DeclareMathOperator{\Quot}{Frac}

\DeclareMathOperator{\Rep}{Rep}
\DeclareMathOperator{\Res}{Res}

\DeclareMathOperator{\Spm}{Sp}
\DeclareMathOperator{\Spec}{Spec}
\DeclareMathOperator{\Spf}{Spf}

\DeclareMathOperator{\Sym}{Sym}
\DeclareMathOperator{\Tor}{Tor}

\DeclareMathOperator{\Var}{V}

\newcommand{\ad}{{\rm ad}}
\newcommand{\alg}{{\rm alg}}

\DeclareMathOperator{\coker}{coker}

\newcommand{\cris}{{\rm cris}}

\newcommand{\dR}{{\rm dR}}
\newcommand{\et}{{\rm \acute{e}t}}
\newcommand{\fpqc}{{\it fpqc}}
\newcommand{\fppf}{{\it fppf}}

\DeclareMathOperator{\id}{\,id}
\DeclareMathOperator{\im}{im}

\renewcommand{\mod}{{\rm\;mod\;}}\DeclareMathOperator{\ord}{ord}
\newcommand{\perf}{{\rm perf}}

\newcommand{\rig}{{\rm rig}}
\DeclareMathOperator{\rk}{rk}

\newcommand{\sep}{{\rm sep}}

\newcommand{\tors}{{\rm tors}}

\renewcommand{\phi}{\varphi}
\renewcommand{\theta}{\vartheta}
\renewcommand{\epsilon}{\varepsilon}

\usepackage{amsfonts}
\newcommand{\BOne} {{\mathchoice{\hbox{\rm1\kern-2.7pt l\kern.9pt}}
                              {\hbox{\rm1\kern-2.7pt l\kern.9pt}}
                              {\hbox{\scriptsize\rm1\kern-2.3pt l\kern.4pt}}
                              {\hbox{\scriptsize\rm1\kern-2.4pt l\kern.5pt}}}}

\newcommand{\BC}{{\mathbb{C}}}

\newcommand{\BF}{{\mathbb{F}}}
\newcommand{\BG}{{\mathbb{G}}}
\newcommand{\BH}{{\mathbb{H}}}

\newcommand{\BN}{{\mathbb{N}}}

\newcommand{\BP}{{\mathbb{P}}}
\newcommand{\BQ}{{\mathbb{Q}}}
\newcommand{\BR}{{\mathbb{R}}}

\newcommand{\BZ}{{\mathbb{Z}}}

\newcommand{\bB}{{\mathbf{B}}}

\newcommand{\bD}{{\mathbf{D}}}

\newcommand{\CC}{{\cal{C}}}

\newcommand{\CE}{{\cal{E}}}
\newcommand{\CF}{{\cal{F}}}
\newcommand{\CG}{{\cal{G}}}
\newcommand{\CH}{{\cal{H}}}

\newcommand{\CJ}{{\cal{J}}}

\newcommand{\CN}{{\cal{N}}}
\newcommand{\CO}{{\cal{O}}}
\newcommand{\CP}{{\cal{P}}}
\newcommand{\CQ}{{\cal{Q}}}

\newcommand{\FG}{{\mathfrak{G}}}

\newcommand{\FS}{{\mathfrak{S}}}

\newcommand{\Fa}{{\mathfrak{a}}}

\newcommand{\Fm}{{\mathfrak{m}}}

\newcommand{\Fp}{{\mathfrak{p}}}
\newcommand{\Fq}{{\mathfrak{q}}}

\newbox\dotCOdbox
\newbox\dotCOtbox
\newbox\dotCOsbox
\newbox\dotCOssbox
\setbox\dotCOdbox=\vbox{\hbox{\kern3.5pt\bf.}\vskip-4.5pt\hbox{$\CO$}}
\setbox\dotCOtbox=\vbox{\hbox{\kern3.5pt\bf.}\vskip-4.5pt\hbox{$\CO$}}
\setbox\dotCOsbox=\vbox{\hbox{\kern2.1pt.}
                       \vskip-6.6pt\hbox{$\scriptstyle\CO$}}
\setbox\dotCOssbox=\vbox{\hbox{\kern1.6pt.}
                        \vskip-8.1pt\hbox{$\scriptscriptstyle\CO$}}

\let\setminus\smallsetminus

\newcommand{\es}{\enspace}
\newcommand{\open}{^\circ}

\newcommand{\dual}{^{\SSC\lor}}

\newcommand{\mal}{^{\SSC\times}}
\newcommand{\fdot}{{\,{\SSC\bullet}\,}}
\newcommand{\ul}[1]{{\underline{#1}}}
\newcommand{\ol}[1]{{\overline{#1}}}
\newcommand{\wh}[1]{{\widehat{#1}}}
\newcommand{\wt}[1]{{\widetilde{#1}}}

\usepackage{ifthen}

\newcommand{\invlim}[1][]{\ifthenelse{\equal{#1}{}}
{\DS \lim_{\longleftarrow}}
{\DS \lim_{\underset{#1}{\longleftarrow}}}
}

\newcommand{\dirlim}[1][]{\ifthenelse{\equal{#1}{}}
{\DS \lim_{\longrightarrow}}
{\DS \lim_{\underset{#1}{\longrightarrow}}}
}

\newcommand{\dbl}{{\mathchoice{\mbox{\rm [\hspace{-0.15em}[}}
                              {\mbox{\rm [\hspace{-0.15em}[}}
                              {\mbox{\scriptsize\rm [\hspace{-0.15em}[}}
                              {\mbox{\tiny\rm [\hspace{-0.15em}[}}}}
\newcommand{\dbr}{{\mathchoice{\mbox{\rm ]\hspace{-0.15em}]}}
                              {\mbox{\rm ]\hspace{-0.15em}]}}
                              {\mbox{\scriptsize\rm ]\hspace{-0.15em}]}}
                              {\mbox{\tiny\rm ]\hspace{-0.15em}]}}}}
\newcommand{\dpl}{{\mathchoice{\mbox{\rm (\hspace{-0.15em}(}}
                              {\mbox{\rm (\hspace{-0.15em}(}}
                              {\mbox{\scriptsize\rm (\hspace{-0.15em}(}}
                              {\mbox{\tiny\rm (\hspace{-0.15em}(}}}}
\newcommand{\dpr}{{\mathchoice{\mbox{\rm )\hspace{-0.15em})}}
                              {\mbox{\rm )\hspace{-0.15em})}}
                              {\mbox{\scriptsize\rm )\hspace{-0.15em})}}
                              {\mbox{\tiny\rm )\hspace{-0.15em})}}}}

\usepackage{color}
\newcounter{commentcounter}

\def\?{\ 
{\bf\color{red}???}\ 
\immediate\write16{}
\immediate\write16{Warning: There was still a question mark . . . }
\immediate\write16{}}

\long\def\forget#1{}

\def\longto{\longrightarrow}
\def\into{\hookrightarrow}
\def\onto{\mbox{$\kern2pt\to\kern-8pt\to\kern2pt$}}
\def\isoto{\stackrel{}{\mbox{\hspace{1mm}\raisebox{+1.4mm}{$\SC\sim$}\hspace{-3.5mm}$\longrightarrow$}}}
\def\longinto{\lhook\joinrel\longrightarrow}

\newbox\mybox
\def\arrover#1{\mathrel{
       \setbox\mybox=\hbox spread 1.4em{\hfil$\scriptstyle#1$\hfil}
       \vbox{\offinterlineskip\copy\mybox
             \hbox to\wd\mybox{\rightarrowfill}}}}

\newcommand{\ancon}[1][]{{\mathchoice
           {\TS\langle\frac{z}{\zeta^{#1}},z^{-1}\}}
           {\TS\langle\frac{z}{\zeta^{#1}},z^{-1}\}}
           {\SC\langle\frac{z}{\zeta^{#1}},z^{-1}\}}
           {\SSC\langle\frac{z}{\zeta^{#1}},z^{-1}\}}}}
\newcommand{\tminus}{\ell_{\SSC -}}
\newcommand{\tplus}{\ell_{\SSC +}}
\newcommand{\ttplus}{\tilde{\ell}_{\SSC +}}
\newcommand{\tplusminus}{\ell}

\DeclareMathOperator{\Dr}{Dr}
\DeclareMathOperator{\Def}{D}

\def\olD{{\,\overline{\!D}}}
\def\olK{{\,\overline{\!K}}}
\def\olM{{\,\overline{\!M}}}

\def\ulD{{\underline{D\!}\,}}
\def\ulE{{\underline{E\!}\,}}
\def\ulM{{\underline{M\!}\,}}

\def\ulHM{{\underline{\hat M\!}\,}{}}

\def\ulCC{{\underline{\mathcal{C}\!}\,}{}}
\def\ulCE{{\underline{\mathcal{E}\!}\,}{}}
\def\ulCF{{\underline{\mathcal{F}\!}\,}{}}

\DeclareMathOperator{\Nilp}{\CN\!\it ilp}
\DeclareMathOperator{\Art}{\mathcal{A}\mathit{rt}}

\newcommand{\GlobalField}{F}

\newcommand{\HJLemmaZMinusZeta}{{\cite[Lemma~1.3]{HartlJuschka}}} 
\newcommand{\HJSectHPStruct}{{\cite[\S\,2]{HartlJuschka}}} 
\newcommand{\HJDefHPStruct}{\cite[Definition~2.3]{HartlJuschka}} 
\newcommand{\HJRemHPWts}{\cite[Remark~2.4]{HartlJuschka}} 
\newcommand{\HJDefHPWts}{\cite[Definition~2.7]{HartlJuschka}} 
\newcommand{\HJCohAMot}{{\cite[\S\,3.5]{HartlJuschka}}} 
\newcommand{\HJCohAMotCohAMod}{{\cite[\S\S\,3.5 and 5.7]{HartlJuschka}}} 
\newcommand{\HJCohAMod}{{\cite[\S\,5.7]{HartlJuschka}}} 
\newcommand{\HJSectSigmaBd}{{\cite[\S\,7]{HartlJuschka}}} 
\newcommand{\HJPropDeg}{\cite[Proposition~7.5]{HartlJuschka}} 

\begin{document}


\author{Urs Hartl and Wansu Kim}

\title{Local Shtukas, Hodge-Pink Structures and Galois Representations}
\date{\today}

\maketitle

\begin{abstract}
\noindent
We review the analog of Fontaine's theory of crystalline $p$-adic Galois representations and their classification by weakly admissible filtered isocrystals in the arithmetic of function fields over a finite field. There crystalline Galois representations are replaced by the Tate modules of so-called local shtukas. We prove that the Tate module functor is fully faithful. In addition to this \'etale realization of a local shtuka we discuss also the de Rham and the crystalline cohomology realizations and construct comparison isomorphisms between these realizations. We explain how local shtukas and these cohomology realizations arise from Drinfeld modules and Anderson's $t$-motives. As an application we construct equi-characteristic crystalline deformation rings, establish their rigid-analytic smoothness and compute their dimension.  
\end{abstract}

\bigskip

%
%

\setcounter{tocdepth}{1}
\tableofcontents

\section{Introduction}
\setcounter{equation}{0}

Let $X$ be a smooth proper scheme over a number field $L$. With $X$ one associates various cohomology realizations, for example the \'etale $\ell$-adic realization $\Koh^i_\et(X\times_LL^\alg,\BQ_\ell)$, the de Rham realization $\Koh^i_\dR(X/L)$ and the Betti realization $\Koh^i_{\rm Betti}(X(\BC),\BZ)$. One has comparison isomorphisms between these realizations. This lead Grothendieck to postulate a category of ``motives'' as a universal cohomology theory through which all cohomology realizations factor.

There exists a rich parallel between the above number field situation and the ``Arithmetic of function fields''. In the latter, Anderson~\cite{Anderson86} introduced the concept of ``$t$-motives''. We slightly generalize his definition as follows. Let $\GlobalField$ be a finite field extension of the rational function field $\BF_q(\theta)$ in the variable $\theta$ over a finite field $\BF_q$ with $q$ elements. Equip $\GlobalField$ with a morphism $\gamma\colon\BF_q[t]\to \GlobalField,\,t\mapsto\theta$ where $\BF_q[t]$ is the polynomial ring. Let $\sigma$ be the endomorphism of $\GlobalField[t]$ with $\sigma(t)=t$ and $\sigma(b)=b^q$ for $b\in \GlobalField$. An \emph{$\BF_q[t]$-motive of rank $r$ over $\GlobalField$} is a pair $\ulM=(M,\tau_M)$ consisting of a free $\GlobalField[t]$-module $M$ of rank $r$ and an $\GlobalField[t][\tfrac{1}{t-\theta}]$-isomorphism $\tau_M\colon (\sigma^*M)[\tfrac{1}{t-\theta}]\isoto M[\tfrac{1}{t-\theta}]$, where $\sigma^*M:=M\otimes_{\GlobalField[t],\sigma}\GlobalField[t]$. A \emph{$t$-motive} in the sense of Anderson~\cite{Anderson86} is an $\BF_q[t]$-motive $\ulM$ such that in addition $\tau_M(\sigma^*M)\subset M$ and  $M$ is finitely generated over the non-commutative polynomial ring $\GlobalField\{\tau\}:=\bigl\{\,\sum_{i=0}^n b_i\tau^i\colon n\in\BN, b_i\in \GlobalField\,\bigr\}$ with $\tau b=b^q\tau$ which acts on $m\in M$ by $\tau(m):=\tau_M(\sigma_{\!M}^*m)$, where $\sigma_{\!M}^*m:=m\otimes 1\in\sigma^*M$.

Various ``cohomology'' realizations are attached to an $\BF_q[t]$-motive $\ulM$; see \cite[\S\,2]{Goss94}. Namely, for a maximal ideal $v$ of $\BF_q[t]$ one has a $v$-adic realization given by the \emph{$v$-adic (dual) Tate module}
\begin{equation*}\label{EqTateModAMotive}
\check T_v\ulM:=(M\otimes_{\GlobalField[t]}A_{v,\GlobalField^\sep})^\tau:=\bigl\{m\in M\otimes_{\GlobalField[t]}A_{v,\GlobalField^\sep}\colon \tau_M(\sigma_{\!M}^*m)=m\bigr\},
\end{equation*}
where $\GlobalField^\sep$ is a separable closure of $\GlobalField$ and $A_{v,\GlobalField^\sep}$ is the $v$-adic completion of $\GlobalField^\sep[t]$. For example if $v=(t)$ then $A_{v,\GlobalField^\sep}=\GlobalField^\sep\dbl t\dbr$. $\check T_v\ulM$ is a free module over the $v$-adic completion $A_v$ of $\BF_q[t]$ of rank equal to the rank of $\ulM$ and carries a continuous action of $\Gal(\GlobalField^\sep/\GlobalField)$. It is sometimes also denoted $\Koh^1_v(\ulM,A_v)$. The \emph{de Rham realization} is defined as $\Koh^1_\dR(\ulM,\GlobalField)=\sigma^*M/(t-\theta)\sigma^*M$. If $\ulM$ is ``uniformizable'' there is also a \emph{Betti realization} with a \emph{Hodge structure}, which is described for example in \HJCohAMot. If $\ulM$ is the $\BF_q[t]$-motive associated with a Drinfeld module, resp.\ abelian $t$-module $\ulE$, the de Rham cohomology $\Koh^1_\dR(\ulE,\GlobalField)$ as defined by Deligne, Anderson, Gekeler and Jing Yu~\cite{Gekeler89,Yu90}, resp.\ Brownawell and Papanikolas~\cite{BrownawellPapanikolas02}, is canonically isomorphic to $\Koh^1_\dR(\ulM,\GlobalField)$; see \HJCohAMod.

Assume now that $\ulM$ has good reduction $\ol\ulM$ at a maximal ideal $\Fm$ of the integral closure $\CO_\GlobalField$ of $\BF_q[\theta]$ in $\GlobalField$; see Example~\ref{ExAMotive} for the precise definition. Set $k:=\CO_\GlobalField/\Fm$ and $\epsilon:=\ker(\BF_q[t]\to \CO_\GlobalField/\Fm)$, and let $\GlobalField_\Fm$ be the completion of $\GlobalField$ at the place $\Fm$. If $v\ne\epsilon$, the assumption of good reduction implies that the Galois action on $\check T_v\ulM$ is unramified at $\Fm$ by \cite[Theorem~1.1]{Gardeyn2} and so the decomposition group $\Gal(\GlobalField_\Fm^\sep/\GlobalField_\Fm)$ acts through its quotient $\Gal(k^\sep/k)$. This representation of $\Gal(k^\sep/k)$ is canonically isomorphic to the (dual) Tate module $\check T_v\ol\ulM:=(\olM\otimes_{k[t]}A_{v,k^\sep})^\tau$ of $\ol\ulM$.

If $v=\epsilon$ however, the reduction $\ol\ulM$ does not possess a $v$-adic Tate module. In the parallel between function fields and number fields the $\Gal(\GlobalField_\Fm^\sep/\GlobalField_\Fm)$-representation $\check T_\epsilon\ulM$ corresponds to what is called a $p$-adic Galois representation (see \cite{Berger} or \cite{BrinonConrad} for an introduction), because $\GlobalField_\Fm$ has the same ``residue $\BF_q[t]$-characteristic'' $\epsilon$ as $A_\epsilon$. One aim of this survey article is to explain the function field analog of Fontaine's theory \cite{Fontaine79,Fontaine82,Fontaine90,Fontaine94} of $p$-adic Galois representations. Namely, Fontaine defined the notion of ``crystalline $p$-adic Galois representation'' and conjectured that the Galois representations arising from the $p$-adic \'etale cohomology of a smooth proper variety with good reduction over a finite field extension of $\BQ_p$ is crystalline. After contributions by Grothendieck, Tate, Fontaine, Lafaille, Messing, Kato and many others, Fontaine's conjecture was finally proved independently by Faltings~\cite{Faltings89}, Niziol~\cite{Niziol98} and Tsuji~\cite{Tsuji}. Moreover, Fontaine classified crystalline representations by ``filtered Frobenius-isocrystals'' whose function field analogs we will explain in Section~\ref{SectHPStruct} below.

It turns out that the appropriate function field analogs of crystalline $p$-adic Galois representations are given by \emph{local shtukas}; see Definition~\ref{DefLocSht}. To describe these let $\ulM$ be an $\BF_q[t]$-motive over $\GlobalField_\Fm$ with good reduction, and assume for simplicity that $\epsilon:=\ker(\BF_q[t]\to \CO_{\GlobalField_\Fm}/\Fm)=(t)$. Set $z:=t$ and $\zeta:=\theta$. Then the $z$-adic completion of the (unique) good model of $\ulM$ over $\CO_{\GlobalField_\Fm}$ is a free $\CO_{\GlobalField_\Fm}\dbl z\dbr$-module $\hat M$ equipped with an isomorphism $\tau_{\hat M}\colon\sigma^*\hat M[\tfrac{1}{z-\zeta}]\isoto\hat M[\tfrac{1}{z-\zeta}]$. The pair $\ulHM=(\hat M,\tau_{\hat M})$ is the \emph{local shtuka at $\epsilon$ associated with $\ulM$}; see Example~\ref{ExAMotive} for more details. Via their Tate modules $\check T_\epsilon\ulHM:=\bigl(\hat M\otimes_{\CO_{\GlobalField_\Fm}\dbl z\dbr}\GlobalField_\Fm^\sep\dbl z\dbr\bigr)^\tau$ and $\check V_\epsilon\ulHM:=\check T_\epsilon\ulHM\otimes_{A_\epsilon}\Quot(A_\epsilon)$ local shtukas give rise to actual representations of $\Gal(\GlobalField_\Fm^\sep/\GlobalField_\Fm)$ on finite dimensional vector spaces over the field $Q_\epsilon:=\Quot(A_\epsilon)$ which also equals the completion of $\BF_q(t)$ at the place $\epsilon$. The category of \emph{equal characteristic crystalline representations} of $\Gal(\GlobalField_\Fm^\sep/\GlobalField_\Fm)$ is defined to be the essential image of the functor $\ulHM\mapsto\check V_\epsilon\ulHM$ on local shtukas over $\CO_{\GlobalField_\Fm}$. We explain the reason for this in Remarks~\ref{RemCrystRealiz} and \ref{RemEqCharCrystRep}.

Let us outline the content of this article. For $\epsilon\ne(t)$ the definition of local shtukas and of local shtukas at $\epsilon$ associated with $\BF_q[t]$-motives is given in Section~\ref{SectLocSht}. In Section~\ref{SectDivLocAM} we discuss their relation with the function field analogs of $p$-divisible groups and in Section~\ref{SectGalRep} we explain the Tate module functor and show that it is fully faithful. Moreover, we define the de Rham realization of local shtukas and $\BF_q[t]$-motives and establish a comparison isomorphism $h_{\epsilon,\dR}$ between the $\epsilon$-adic (dual) Tate module and the de Rham realization. In Section~\ref{SectHPStruct} we define the crystalline realization of local shtukas and $\BF_q[t]$-motives and we construct the comparison isomorphisms $h_{\dR,\cris}$ between the de Rham and the crystalline realization and $h_{\epsilon,\cris}$ between the $\epsilon$-adic (dual) Tate module and the crystalline realization. The former provides on the crystalline realization a Hodge structure in the sense of Pink. These are the function field analogs of Fontaine's filtered Frobenius-isocrystals which classify crystalline $p$-adic Galois representations. In Section~\ref{SectWA} we describe which function field isocrystals with Hodge-Pink structure arise from local shtukas using the analog of Fontaine's notion of weak admissibility.

The function field analog of Fontaine's theory has a somewhat surprising application to Galois deformation theory; namely, there exists an equi-characteristic local  analog of flat deformation rings (or rather, crystalline deformation rings, see Theorems~\ref{ThmDeformation} and \ref{ThmGenSm}), which have been intensively studied in the $p$-adic setting for applications to modularity lifting theorems and congruences of modular forms; see for example \cite{KisinPST,KisinModuliFF,KisinFM,GeeKisinBreuilMezard}. The existence of equi-characteristic Galois deformation theory is somewhat surprising because the usual (unrestricted) deformation ring does not exist in the category of complete local noetherian rings since the tangent space cannot be bounded; cf.~Remark~\ref{RemTangentSp}. Nonetheless, the equi-characteristic analog of Fontaine's theory (or rather, its torsion version) provides a means to bound the tangent space of the deformation functor.  We give an expository overview of the equi-characteristic local Galois deformation theory in Section~\ref{SectDefoTh}, which is analogous to,  and was inspired by Kisin's technique to study flat deformation rings \cite{KisinModuliFF}. In order to develop Galois deformation theory, we need some techniques to handle torsion Galois representations -- some analog of Raynaud's theory of finite flat group schemes -- which is explained in Section~\ref{SectTorLocSh}.

\bigskip\noindent
{\it Acknowledgments.} These notes grew out of two lectures given by the authors in the fall of 2009 at the conference ``$t$-motives: Hodge structures, transcendence and other motivic aspects'' at the Banff International Research Station BIRS in Banff, Canada and a lecture course given by the first author at the Centre de Recerca Mathem\`atica CRM in Barcelona, Spain in the spring of 2010. The authors are grateful to BIRS and CRM for their hospitality. The first author acknowledges support of the DFG (German Research Foundation) in form of SFB 878 and Germany's Excellence Strategy EXC 2044--390685587 ``Mathematics M\"unster: Dynamics--Geometry--Structure''. The second author acknowledges support of the Herchel Smith Postdoctoral Research Fellowship, and the EPSRC (Engineering and Physical Sciences Research Council) in the form of EP/L025302/1.

\subsection{Notation}\label{Notation}
Throughout this article we denote by
\begin{tabbing}
$A:=\Gamma(C\setminus\{\infty\},\CO_C)$\quad \=\kill
$\BF_q$\> a finite field with $q$ elements,\\[1mm]
$C$\> a smooth projective geometrically irreducible curve over $\BF_q$,\\[1mm]
$Q:=\BF_q(C)$\> the function field of $C$,\\[1mm]
$\infty$\> a fixed closed point of $C$,\\[1mm]
$A:=\Gamma(C\setminus\{\infty\},\CO_C)$\> \parbox[t]{0.76\textwidth}{the ring of regular functions on $C$ outside $\infty$. This generalizes the rings $A=\BF_q[t]$ and $Q=\BF_q(t)$ from the introduction, which correspond to $C=\BP^1_{\BF_q}$.}\\[1mm]
$K$ \> \parbox[t]{0.76\textwidth}{a field which is complete with respect to a non-trivial, non-archimedean absolute value} \\[1mm]
$|\,.\,|\colon K\to\BR_{\ge0}$\\[1mm]
$|K^{^{\SSC\times}}|=\{|x|\colon x\in K^{^{\SSC\times}}\}$ \>its value group. We do \emph{not} assume that $|K^{^{\SSC\times}}|$ is discrete in $\BR_{>0}$.\\[1mm]
$\olK$ \> the completion of an algebraic closure $K^\alg$ of $K$,\\[1mm]
$R=\{x\in K\colon|x|\le1\}$ \> the valuation ring of $K$, \\[1mm]
$\Fm_R=\{x\in R\colon|x|<1\}$ \> the maximal ideal of $R$,\\[1mm]
$k=R/\Fm_R$ \> the residue field of $R$,\\[1mm]
$\gamma\colon A\into R$ \> an injective ring homomorphism,\\[1mm]
$\epsilon:=\gamma^{-1}(\Fm_R)$ \> \parbox[t]{0.76\textwidth}{the kernel of the induced homomorphism $A\to k$. It is called the \emph{$A$-characteristic of $k$}. We assume that $\epsilon\subset A$ is a maximal ideal.}\\[2mm]
$\BF_\epsilon=A/\epsilon$\> the finite residue field at $\epsilon$,\\[1mm]
$\hat q=\#\BF_\epsilon$ \> the cardinality of $\BF_\epsilon$. It satisfies $\hat q=q^{[\BF_\epsilon:\BF_q]}$.\\[1mm]
$z\in Q$ \> a fixed uniformizing parameter of $C$ at $\epsilon$,\\[1mm]
$\zeta=\gamma(z)$ \> the image of $z$ in $R$,\\[1mm]
$A_\epsilon$\> \parbox[t]{0.76\textwidth}{the completion of the stalk $\CO_{C,\epsilon}$ at  $\epsilon$, which is canonically isomorphic to $\BF_\epsilon\dbl z\dbr$,}\\[1mm]
$Q_\epsilon=\Quot(A_\epsilon)$\> its fraction field,\\[1mm]
$R\dbl z\dbr$ \> the power series ring in the variable $z$ over $R$,\\[1mm]
$\hat\sigma$ \> the endomorphism of $R\dbl z\dbr$ with $\hat\sigma(z)=z$ and $\hat\sigma(b)=b^{\hat q}$ for $b\in R$.
\end{tabbing}

%
%

\section{Local Shtukas}\label{SectLocSht}
\setcounter{equation}{0}

The function field analog of Fontaine's theory of $p$-adic Galois representations rests on the concept of local shtukas. We are interested in two kinds of base schemes $S$ in this article. The first kind is $S=\Spec R$ where $R$ is a valuation ring as in Notation~\ref{Notation}. The second kind are schemes $S$ in the category $\Nilp_{A_\epsilon}$ of $A_\epsilon$-schemes on which a uniformizer $z$ of $A_\epsilon$ is locally nilpotent. We denote by $\zeta$ the image of $z$ in $R$, respectively in the structure sheaf $\CO_S$ of $S$. The relation between the two kinds of base schemes is that $\Spec R/(\zeta^n)\in\Nilp_{A_\epsilon}$ for all positive integers $n$.

To define local shtukas over $S=\Spec R$ we consider finite free modules $\hat M$ over the power series ring $R\dbl z\dbr$ and set $\hat\sigma^*\hat M:=\hat M\otimes_{R\dbl z\dbr,\,\hat\sigma}R\dbl z\dbr$, and $\hat M[\tfrac{1}{z-\zeta}]:=\hat M\otimes_{R\dbl z\dbr}R\dbl z\dbr[\tfrac{1}{z-\zeta}]$, and $\hat M[\tfrac{1}{z}]:=\hat M\otimes_{R\dbl z\dbr}R\dbl z\dbr[\tfrac{1}{z}]$. For more general $S$ we let $\CO_S\dbl z\dbr$ be the sheaf of $\CO_S$-algebras on $S$ for the \fpqc-topology whose ring of sections on an $S$-scheme $Y$ is the ring of power series $\CO_S\dbl z\dbr(Y):=\Gamma(Y,\CO_Y)\dbl z\dbr$.  This is indeed a sheaf being the countable direct product of $\CO_S$. Let $\CO_S\dbl z\dbr[\tfrac{1}{z-\zeta}]$, respectively $\CO_S\dpl z\dpr$, be the \fpqc-sheaf of $\CO_S$-algebras on $S$ associated with the presheaf $Y\mapsto\Gamma(Y,\CO_Y)\dbl z\dbr[\tfrac{1}{z-\zeta}]$, respectively $Y\mapsto\Gamma(Y,\CO_Y)\dbl z\dbr[\tfrac{1}{z}]$. If $Y$ is quasi-compact then $\Gamma\bigl(Y,\CO_S\dbl z\dbr[\tfrac{1}{z-\zeta}]\bigr)=\Gamma(Y,\CO_Y)\dbl z\dbr[\tfrac{1}{z-\zeta}]$ and $\Gamma\bigl(Y,\CO_S\dpl z\dpr\bigr)=\Gamma(Y,\CO_Y)\dbl z\dbr[\tfrac{1}{z}]$ by \cite[Exercise II.1.11]{Hartshorne}. Also $\CO_S\dbl z\dbr[\tfrac{1}{z-\zeta}]=\CO_S\dpl z\dpr$ if $S\in\Nilp_{A_\epsilon}$. Let $\hat\sigma$ be the endomorphism of $\CO_S\dbl z\dbr$ with $\hat\sigma(z)=z$ and $\hat\sigma(b)=b^{\hat q}$ for sections $b$ of $\CO_S$. For a sheaf of $\CO_S\dbl z\dbr$-modules $\hat M$ we set $\hat\sigma^*\hat M:=\hat M\otimes_{\CO_S\dbl z\dbr,\,\hat\sigma}\CO_S\dbl z\dbr$, as well as $\hat M[\tfrac{1}{z-\zeta}]:=\hat M\otimes_{\CO_S\dbl z\dbr}\CO_S\dbl z\dbr[\tfrac{1}{z-\zeta}]$ and $\hat M[\tfrac{1}{z}]:=\hat M\otimes_{\CO_S\dbl z\dbr}\CO_S\dpl z\dpr$. There is a natural $\hat\sigma$-semilinear map $\hat M\to\hat\sigma^*\hat M,\,m\mapsto\hat\sigma_{\!\hat M}^*m:=m\otimes1$. For a morphism of $\CO_S\dbl z\dbr$-modules $f\colon \hat M\to\hat M'$ we set $\hat\sigma^*f:=f\otimes\id\colon\hat\sigma^*\hat M\to\hat\sigma^*\hat M'$. Note that by \cite[Proposition~2.3]{HV1} a sheaf of $\CO_S\dbl z\dbr$-modules which \fpqc-locally on $S$ is isomorphic to $\CO_S\dbl z\dbr^{\oplus r}$ is already Zariski-locally on $S$ isomorphic to $\CO_S\dbl z\dbr^{\oplus r}$. We therefore call such a sheaf simply a \emph{locally free $\CO_S\dbl z\dbr$-module of rank $r$}. If $S=\Spec R$ for a valuation ring $R$ as above, such a sheaf corresponds to a free $R\dbl z\dbr$-module of rank $r$. Note that in this case $R\dbl z\dbr[\tfrac{1}{z-\zeta}]\ne R\dbl z\dbr[\tfrac{1}{z}]$

\begin{definition}\label{DefLocSht}
A \emph{local $\hat\sigma$-shtuka} (or \emph{local shtuka}) \emph{of rank} $r$ over $S$ is a pair $\ulHM=(\hat M,\tau_{\hat M})$ consisting of a locally free $\CO_S\dbl z\dbr$-module $\hat M$ of rank $r$, and an isomorphism $\tau_{\hat M}\colon\hat\sigma^\ast\hat M[\frac{1}{z-\zeta}] \isoto\hat M[\frac{1}{z-\zeta}]$. If $\tau_{\hat M}(\hat\sigma^*\hat M)\subset\hat M$ then $\ulHM$ is called \emph{effective}, and if $\tau_{\hat M}(\hat\sigma^*\hat M)=\hat M$ then $\ulHM$ is called \emph{\'etale}.

A \emph{morphism} of local shtukas $f\colon(\hat M,\tau_{\hat M})\to(\hat M',\tau_{\hat M'})$ over $S$ is a morphism of $\CO_S\dbl z\dbr$-modules $f\colon\hat M\to\hat M'$ which satisfies $\tau_{\hat M'}\circ\hat\sigma^*f = f\circ \tau_{\hat M}$. We denote the set of morphisms from $\ulHM$ to $\ulHM'$ by $\Hom_S(\ulHM,\ulHM')$.

A \emph{quasi-morphism} between local shtukas $f\colon(\hat M,\tau_{\hat M})\to(\hat M',\tau_{\hat M'})$ over $S$ is a morphism of $\CO_S\dpl z\dpr$-modules $f\colon\hat M[\tfrac{1}{z}]\to\hat M'[\tfrac{1}{z}]$ with $\tau_{\hat M'}\circ\hat\sigma^*f=f\circ\tau_{\hat M}$. It is called a \emph{quasi-isogeny} if it is an isomorphism of $\CO_S\dpl z\dpr$-modules. We denote the set of quasi-morphisms from $\ulHM$ to $\ulHM'$ by $\QHom_S(\ulHM,\ulHM')$.
\end{definition}

When $S=\Spec R$ we will show in Corollary~\ref{CorHomFinGen} below that $\Hom_S(\ulHM,\ulHM{}')$ is a finite free $A_\epsilon$-module of rank at most $\rk\ulHM\cdot\rk\ulHM{}'$. Of fundamental importance is the following

\begin{example}\label{ExAMotive}
As in Notation~\ref{Notation} let $R$ be a valuation ring and set $A_R:=A\otimes_{\BF_q}R$, and let $\sigma:=\id_A\otimes\Frob_{q,R}$ be the endomorphism of $A_R$ with $\sigma(a\otimes b)=a\otimes b^{q}$ for $a\in A$ and $b\in R$. By an \emph{effective $A$-motive of rank $r$ over $R$} we mean a pair $\ulM=(M,\tau_M)$ consisting of a locally free $A_R$-module $M$ of rank $r$ and an injective $A_R$-homomorphism $\tau_M\colon\sigma^*M\into M$ whose cokernel is a finite free $R$-module and annihilated by a power of the ideal $\CJ:=(a\otimes1-1\otimes\gamma(a)\colon a\in A)=\ker(\gamma\otimes\id_R\colon A_R\onto R)\subset A_R$.

More generally, an \emph{$A$-motive of rank $r$ over $R$} is a pair $\ulM=(M,\tau_M)$ consisting of a locally free $A_R$-module $M$ of rank $r$ and an isomorphism $\tau_M\colon\sigma^*M|_{\Spec A_R\setminus\Var(\CJ)}\isoto M|_{\Spec A_R\setminus\Var(\CJ)}$ of the associated sheaves outside $\Var(\CJ)\subset\Spec A_R$. Note that if $A=\BF_q[t]$ then $\CJ=(t-\gamma(t))$ and $\Spec A_R\setminus\Var(\CJ)=\Spec R[t][\tfrac{1}{t-\gamma(t)}]$. An $A$-motive over $R$ should be viewed as a model over $R$ with good reduction of the $A$-motive $\ulM\otimes_RK$ over $K$ of generic $A$-characteristic $(0)=\ker(\gamma\colon A\to K)$. Its reduction $\ulM\otimes_Rk$ is an $A$-motive over $k$ of $A$-characteristic $\epsilon=\ker(A\to k)$. Note that we discussed the case where $A=\BF_q[t]$ in the introduction.

Let $\ulM$ be an (effective) $A$-motive. We consider the $\epsilon$-adic completions $A_{\epsilon,R}$ of $A_R$ and $\ulM\otimes_{A_R}A_{\epsilon,R}:=(M\otimes_{A_R}A_{\epsilon,R}\,,\,\tau_M\otimes\id)$ of $\ulM$. If $\BF_\epsilon=\BF_q$, and hence $\hat q=q$ and $\hat\sigma=\sigma$, we have $A_{\epsilon,R}=R\dbl z\dbr$ and $\CJ\cdot A_{\epsilon,R}=(z-\zeta)$ because $R\otimes_{A_R}A_{\epsilon,R}=R$. So $\ulM\otimes_{A_R}A_{\epsilon,R}$ is an (effective) local shtuka over $\Spec R$ which we denote by $\ulHM_\epsilon(\ulM)$ and call the \emph{local $\hat\sigma$-shtuka at $\epsilon$ associated with $\ulM$}. If $f:=[\BF_\epsilon:\BF_q]>1$ the construction is slightly more complicated; compare the discussion in \cite[after Proposition~8.4]{BH1}. Namely, by continuity the map $\gamma$ extends to a ring homomorphism $\gamma\colon A_\epsilon\into R$. We consider the canonical isomorphism $\BF_\epsilon\dbl z\dbr\isoto A_\epsilon$ and the ideals $\Fa_i=(a\otimes1-1\otimes\gamma(a)^{q^i}\colon a\in\BF_\epsilon)\subset A_{\epsilon,R}$ for $i\in\BZ/f\BZ$, which satisfy $\prod_{i\in\BZ/f\BZ}\Fa_i=(0)$, because $\prod_{i\in\BZ/f\BZ}(X-a^{q^i})\in \BF_q[X]$ is a multiple of the minimal polynomial of $a$ over $\BF_q$ and even equal to it when $\BF_\epsilon=\BF_q(a)$. By the Chinese remainder theorem $A_{\epsilon,R}$ decomposes
\begin{equation}\label{EqDecompAEpsilonR}
A_{\epsilon,R}\;=\;\prod_{i\in\BZ/f\BZ}A_{\epsilon,R}/\Fa_i\,.
\end{equation}
Each factor is canonically isomorphic to $R\dbl z\dbr$. The factors are cyclically permuted by $\sigma$ because $\sigma(\Fa_i)=\Fa_{i+1}$. In particular $\sigma^f$ stabilizes each factor. The ideal $\CJ$ decomposes as follows $\CJ\!\cdot\! A_{\epsilon,R}/\Fa_0=(z-\zeta)$ and $\CJ\!\cdot\! A_{\epsilon,R}/\Fa_i=(1)$ for $i\ne0$. We define the \emph{local $\hat\sigma$-shtuka at $\epsilon$ associated with $\ulM$} as $\ulHM_\epsilon(\ulM):=(\hat M,\tau_{\hat M}):=\bigl(M\otimes_{A_R}A_{\epsilon,R}/\Fa_0\,,\,(\tau_M\otimes1)^f\bigr)$, where $\tau_M^f:=\tau_M\circ\sigma^*\tau_M\circ\ldots\circ\sigma^{(f-1)*}\tau_M$. Of course if $f=1$ we get back the definition of $\ulHM_\epsilon(\ulM)$ given above. Also note if $\ulM$ is effective, then $M/\tau_M(\sigma^*M)=\hat M/\tau_{\hat M}(\hat\sigma^*\hat M)$.

The local shtuka $\ulHM_\epsilon(\ulM)$ allows to recover $\ulM\otimes_{A_R}A_{\epsilon,R}$ via the isomorphism
\[
\bigoplus_{i=0}^{f-1}(\tau_M\otimes1)^i\mod\Fa_i\colon\Bigl(\bigoplus_{i=0}^{f-1}\sigma^{i*}(M\otimes_{A_R}A_{\epsilon,R}/\Fa_0),\;(\tau_M\otimes1)^f\oplus\bigoplus_{i\ne0}\id\Bigr)\;\isoto\;\ulM\otimes_{A_R}A_{\epsilon,R}\,,
\]
because for $i\ne0$ the equality $\CJ\!\cdot\! A_{\epsilon,R}/\Fa_i=(1)$ implies that $\tau_M\otimes 1$ is an isomorphism modulo $\Fa_i$; see \cite[Propositions~8.8 and 8.5]{BH1} for more details.

Note that $\ulM\mapsto\ulHM_\epsilon(\ulM)$ is a functor. The philosophy is that $\ulHM_\epsilon(\ulM)$ encodes all the local information of $\ulM$ at $\epsilon$ like the (dual) Tate module $\check T_v\ulM$ encodes all the local information of $\ulM$ at $v\ne\epsilon$. The whole example also works over a base scheme $S\in\Nilp_{A_\epsilon}$ instead of $S=\Spec R$; see \cite[Example~7.2]{HartlIsog} and \cite[\S\,6]{HartlSingh}.
\end{example}

We will need the following lemma whose second part is proved more generally for $S\in\Nilp_{A_\epsilon}$ in \cite[Lemma~2.3]{HartlSingh}.

\begin{lemma}\label{LemmaLocSht}
Let $\ulHM=(\hat M,\tau_{\hat M})$ be a local shtuka of rank $r$ over a valuation ring $R$ as in Notation~\ref{Notation}.
\begin{enumerate}
\item\label{LemmaLocSht_A}
There is an integer $d\in\BZ$ such that $\det\tau_{\hat M}\in (z-\zeta)^d\!\cdot\!R\dbl z\dbr\mal$.
\item\label{LemmaLocSht_B}
If $\ulHM$ is effective, the integer $d$ from \ref{LemmaLocSht_A} satisfies $d\ge0$ and $\hat M/\tau_{\hat M}(\hat\sigma^*\hat M)$ is a free $R$-module of rank $d$ which is annihilated by $(z-\zeta)^d$.
\end{enumerate}
\end{lemma}

\begin{proof}
Compare \cite[Proposition~2.1.3 and Lemma~2.1.2]{HartlPSp}.

\smallskip\noindent
\ref{LemmaLocSht_A} Since $R\dbl z\dbr$ is a local ring, we may choose a basis of $\hat M$ and non-negative integers $s,t$ such that the matrices of $(z-\zeta)^s\!\cdot\!\tau_{\hat M}$ and $(z-\zeta)^t\!\cdot\!\tau_{\hat M}^{-1}$ with respect to this basis lie in $M_r(R\dbl z\dbr)$. Set $f=(z-\zeta)^{rs}\det\tau_{\hat M}$ and $g=(z-\zeta)^{rt}\det\tau_{\hat M}^{-1}$. Then $f,g\in R\dbl z\dbr$ satisfy $fg=(z-\zeta)^{r(s+t)}$.  Since $R$ is an integral domain, $(z-\zeta)R\dbl z\dbr$ is a prime ideal. So $f\in (z-\zeta)^u\!\cdot\!R\dbl z\dbr\mal$ for a non-negative integer $u$, and so $\det\tau_{\hat M}\in(z-\zeta)^{u-rs}\!\cdot\!R\dbl z\dbr\mal$.

\smallskip\noindent
\ref{LemmaLocSht_B} Let $d$ be the integer from \ref{LemmaLocSht_A}. Since $\ulM$ is effective we may take $s=0$, and hence $d\ge0$. By Cramer's rule (e.g.~\cite[III.8.6, Formulas (21) and (22)]{BourbakiAlgebra}) the matrix of $\tau_{\hat M}^{-1}$ lies in $M_r\bigl((z-\zeta)^{-d}R\dbl z\dbr\bigr)$. This implies that $\hat M/\tau_{\hat M}(\hat\sigma^*\hat M)$ is annihilated by $(z-\zeta)^d$. By \cite[Lemma~2.1.2]{HartlPSp} it is a finite free $R$-module. We can compute its rank after reducing modulo $\Fm_R$. Then we are over the principal ideal domain $k\dbl z\dbr$ and the elementary divisor theorem tells us that $\dim_k\bigl(\hat M/\tau_{\hat M}(\hat\sigma^*\hat M)\bigr)\otimes_Rk=\ord_z(\det\tau_{\hat M}\mod\Fm_R)=\ord_z\bigl((z-\zeta)^d\mod\Fm_R\bigr)=d$.
\end{proof}

More precisely \cite[Proposition~2.1.3]{HartlPSp} (and the lemma) say that all (effective) local shtukas over $\Spec R$ are \emph{bounded} (by $(d,0,\ldots,0)$ for $d=\rk_R\hat M/\tau_{\hat M}(\hat\sigma^*\hat M)$) as in the following

\begin{definition}\label{DefBounded}
Let $\mu_1\ge\ldots\ge\mu_r$ be a decreasing sequence of integers. A local shtuka $(\hat M,\tau_{\hat M})$ of rank $r$ over $S$ is \emph{bounded by $(\mu_1,\ldots,\mu_r)$} if
\[
\wedge^i\tau_{\hat M}\bigl(\wedge^i\hat\sigma^\ast\hat M\bigr)\;\subset\;(z-\zeta)^{\mu_{r-i+1}+\ldots+\mu_r}\cdot\wedge^i\hat M\qquad\text{for }1\le i\le r\text{ with equality for } i=r\,. 
\]
\end{definition}

Although over a general base boundedness is not preserved under quasi-isogenies this is true over a valuation ring.

\begin{lemma}\label{LemmaIsogBounded}
Let $\ulHM$ and $\ulHM'$ be two isogenous local shtukas over a valuation ring $R$ as in Notation~\ref{Notation}. If $\ulHM$ is bounded by $(d,0,\ldots,0)$, respectively satisfies $(z-\zeta)^a\hat M\subset \tau_{\hat M}(\hat\sigma^*\hat M)\subset(z-\zeta)^b\hat M$, then the same is true for $\ulHM'$.
\end{lemma}

\begin{proof}
Choosing bases of $M$ and $M'$ we write $\tau_M$ and $\tau_{M'}$ as matrices $T,T'\in\GL_r(R\dbl z\dbr[\tfrac{1}{z-\zeta}])$ and the quasi-isogeny $f\colon\ulM[\tfrac{1}{z}]\isoto\ulM'[\tfrac{1}{z}]$ as a matrix $F\in\GL_r(R\dbl z\dbr[\tfrac{1}{z}])$ satisfying $T'=F\!\cdot\!T\!\cdot\!\hat\sigma^*(F)^{-1}$. Depending on the assumption on $\ulM$, the matrix coefficients of $T'$, $(z-\zeta)^{-d}\det(T')$ and $(z-\zeta)^d\det(T')^{-1}$, respectively $(z-\zeta)^{-b}T'$ and $(z-\zeta)^a(T')^{-1}$ lie in $R\dbl z\dbr[\tfrac{1}{z}]$. Since they also lie in $R\dbl z\dbr[\tfrac{1}{z-\zeta}]$, we can write them as $f/z^s=g/(z-\zeta)^t$. The term $(z-\zeta)^tf=z^s g$ lies in the prime ideal $(z-\zeta)\subset R\dbl z\dbr$, but $z$ does not. Therefore $g$ is divisible by $(z-\zeta)^t$ and all the matrix coefficients $g/(z-\zeta)^t$ lie in $R\dbl z\dbr$ as desired.
\end{proof}

\begin{remark}\label{RemRelationBetweenTheTwoS}
If we are considering a local shtuka $\ulHM$ over $\Spec R$ for a valuation ring $R$ as in Notation~\ref{Notation}, we obtain for all $i\in\BN$ a local shtuka $\ulHM^{(i)}:=\ulHM\otimes_R R/(\zeta^i)$ over $\Spec R/(\zeta^i)$. The $\ulHM^{(i)}$ form a \emph{local shtuka over $\Spf R$} by which we mean a projective system $(\ulHM^{(i)})_{i\in\BN}$ of local shtukas $\ulHM^{(i)}$ over $\Spec R/(\zeta^i)$ together with isomorphisms $\ulHM^{(i+1)}\otimes_{R/(\zeta^{i+1})}R/(\zeta^i)\isoto\ulHM^{(i)}$. By \cite[Proposition~3.16 and \S\,4]{HV1} the functor $\ulHM\mapsto(\ulHM^{(i)})_{i\in\BN}$ is an equivalence between local shtukas \emph{bounded by $(\mu_1,\ldots,\mu_r)$} over $\Spec R$ and over $\Spf R$. In that sense the theory of local shtukas over $\Spec R$ is subsumed under the theory of \emph{bounded} local shtukas over schemes $S\in\Nilp_{A_\epsilon}$.
\end{remark}

\begin{example}\label{ExCarlitz2}
We discuss the case of the Carlitz module \cite{Carlitz35}. We keep the notation from Example~\ref{ExAMotive} and set $A=\BF_q[t]$. Let $\BF_q(\theta)$ be the rational function field in the variable $\theta$ and let $\gamma\colon A\to\BF_q(\theta)$ be given by $\gamma(t)=\theta$. The \emph{Carlitz motive over $\BF_q(\theta)$} is the $A$-motive $\ulM=\bigl(\BF_q(\theta)[t],t-\theta\bigr)$. 

Now let $\epsilon=(z)\subset A$ be a maximal ideal generated by a monic prime element $z=z(t)\in \BF_q[t]$. Then $\BF_\epsilon=A/(z)$ and $A_\epsilon$ is canonically isomorphic to $\BF_\epsilon\dbl z\dbr$. Let $R\supset\BF_\epsilon\dbl\zeta\dbr$ be a valuation ring as in Notation~\ref{Notation} and let $\theta=\gamma(t)\in R$. The Carlitz motive has a model over $R$ with good reduction given by the $A$-motive $\ulM=(R[t],t-\theta)$ over $R$.

If $\deg_t z(t)=1$, that is $z(t)=t-a$ for $a\in\BF_q$, then $\BF_\epsilon=\BF_q$, $\zeta=\theta-a$, and $z-\zeta=t-\theta$. So $\ulHM_\epsilon(\ulM)=(R\dbl z\dbr,z-\zeta)$.

If $\deg_t z(t)=f>1$, then $\ulHM_\epsilon(\ulM)=\bigl(R\dbl z\dbr,(t-\theta)(t-\theta^q)\cdots(t-\theta^{q^{f-1}})\bigr)$. Here the product $(t-\theta)(t-\theta^q)\cdots(t-\theta^{q^{f-1}})=(z-\zeta)u$ for a unit $u\in \BF_\epsilon\dbl\zeta\dbr\dbl z\dbr\mal$, because $\tau_M(\sigma^*M)=(t-\theta)M$ implies that $\ulHM_\epsilon(\ulM)$ is effective and the $d$ from Lemma~\ref{LemmaLocSht} is $1$. In order to get rid of $u$ we denote the image of $t$ in $\BF_\epsilon$ by $\lambda$. Then $\BF_\epsilon=\BF_q(\lambda)$ and $z(t)$ equals the minimal polynomial $(t-\lambda)\cdots(t-\lambda^{q^{f-1}})$ of $\lambda$ over $\BF_q$. Moreover, $t\equiv\lambda\mod zA_\epsilon$ and $\theta\equiv\lambda\mod\zeta\BF_\epsilon\dbl\zeta\dbr$. We compute in $\BF_\epsilon\dbl\zeta\dbr\dbl z\dbr/(\zeta)$
\[
z(t)\;=\;(t-\lambda)\cdots(t-\lambda^{q^{f-1}})\;\equiv\;(t-\theta)\cdots(t-\theta^{q^{f-1}})\;\equiv\;(z-\zeta)u\;\equiv\;zu\;\mod\zeta\,.
\]
Since $z$ is a non-zero-divisor in $\BF_\epsilon\dbl\zeta\dbr\dbl z\dbr/(\zeta)$ it follows that $u\equiv1\mod\zeta\,\BF_\epsilon\dbl\zeta\dbr\dbl z\dbr$. We write $u=1+\zeta u'$ and observe that the product 
\[
w\;:=\;\prod_{n=0}^\infty\hat\sigma^n(u)\;=\;\prod_{n=0}^\infty\hat\sigma^n(1+\zeta u')\;=\;\prod_{n=0}^\infty\bigl(1+\zeta^{\hat q^n}\hat\sigma^n(u')\bigr)
\]
converges in $\BF_\epsilon\dbl\zeta\dbr\dbl z\dbr\mal$ because $\BF_\epsilon\dbl\zeta\dbr\dbl z\dbr$ is $\zeta$-adically complete. It satisfies $w=u\cdot\hat\sigma(w)$ and so multiplication with $w$ defines a canonical isomorphism $(R\dbl z\dbr,z-\zeta)\isoto\ulHM_\epsilon(\ulM)$.

We conclude that $\ulHM_\epsilon(\ulM)=(R\dbl z\dbr,z-\zeta)$, regardless of $\deg_t z(t)$.
\end{example}

%
%

\section{Divisible local Anderson modules}\label{SectDivLocAM}
\setcounter{equation}{0}

Let $S\in\Nilp_{A_\epsilon}$ and let $\ulHM=(\hat M,\tau_{\hat M})$ be an effective local shtuka over $S$. Set $\ulHM_n:=(\hat M_n,\tau_{\hat M_n}):=(\hat M/z^n\hat M,\tau_{\hat M}\mod z^n)$. For $m\in\hat M_n$ set $\hat\sigma_{\!\hat M}^*m:=m\otimes 1\in \hat M_n\otimes_{\CO_S\dbl z\dbr,\,\hat\sigma}\CO_S\dbl z\dbr=:\hat\sigma^*\hat M_n$. Drinfeld~\cite[\S\,2]{Drinfeld87} associates with $\ulHM_n$ a group scheme
\begin{eqnarray}\label{EqDrinfeldGr}
\Dr_{\hat q}(\ulHM_n) & := & \Spec\Bigl(\Sym_{\CO_S}(\hat M_n)\big/ \bigl(m^{\otimes\hat q}-\tau_{\hat M}(\hat\sigma_{\!\hat M}^*m)\colon m\in\hat M_n\bigr)\Bigr)\\[2mm]
& = & \Spec \CO_S[m_1,\ldots,m_{nr}]\big/\bigr(m_i^{\hat q}-\tau_{\hat M}(\hat\sigma_{\!\hat M}^*m_i)\colon 1\le i\le nr\bigr),
\end{eqnarray}
if $\hat M_n=\bigoplus_{i=1}^{nr} \CO_S\cdot m_i$ locally on $S$. It has the following properties
\begin{itemize}
\item $\Dr_{\hat q}(\ulHM_n)\subset\Spec\Sym_{\CO_S}(\hat M_n)$ is a finite locally free subgroup scheme over $S$ of order $\hat q^{nr}$, that is, the $\CO_S$-algebra $\CO_{\Dr_{\hat q}(\ulHM_n)}$ is a finite locally free $\CO_S$-module of rank $\hat q^{nr}$. Note that locally on $S$ we have $\hat M_n=\bigoplus_{i=1}^{nr} \CO_S\cdot m_i$ and $\Spec\Sym_{\CO_S}(\hat M_n)\cong\Spec \CO_S[m_1,\ldots,m_{nr}]=\BG_{a,S}^{nr}$.
\item $\Dr_{\hat q}(\ulHM_n)$ inherits from $\ulHM_n$ an action of $A_\epsilon/(z^n)=\BF_\epsilon[z]/(z^n)$.
\item The Verschiebung map is zero on $\Dr_{\hat q}(\ulHM_n)$.
\item There is a canonical isomorphism between $\coker\tau_{\hat M_n}=\hat M/\bigl(z^n\hat M+\tau_{\hat M}(\hat\sigma^*\hat M)\bigr)$ and the co-Lie module $\omega_{\Dr_{\hat q}(\ulHM_n)}:=e^*\Omega^1_{\Dr_{\hat q}(\ulHM_n)/S}$ where $e\colon S\to\Dr_{\hat q}(\ulHM_n)$ is the zero section. See \cite[Theorem~5.2(f)]{HartlSingh} for a proof.
\item $\Dr_{\hat q}(\ulHM_n)$ is a strict $\BF_\epsilon$-scheme in the sense of Faltings~\cite{Faltings02} and Abrashkin~\cite{Abrashkin}. See \cite[Theorem~2]{Abrashkin} for a proof, or \cite[\S\,5]{HartlSingh}.
\end{itemize}
Conversely we recover $\hat M_n$ as the $A_\epsilon/(z^n)$-module of $\BF_\epsilon$-linear morphisms of $S$-group schemes\linebreak $\Hom_{S\text{-groups},\BF_\epsilon\text{-lin}}\bigl(\Dr_{\hat q}(\ulHM_n)\,,\,\BG_{a,S}\bigr)$ by \cite[Theorem~2]{Abrashkin} or \cite[Theorem~5.2]{HartlSingh}. Moreover the structure as $\CO_S$-module is given via the action of $\CO_S$ on the additive group scheme $\BG_{a,S}$ and $\tau_{\hat M_n}$ corresponds to the map $\hat\tau\colon m\mapsto F_{\hat q,\BG_{a,S}/S}\circ m$, where $F_{\hat q,\BG_{a,S}/S}$ is the relative $\hat q$-Frobenius of $\BG_{a,S}$ over $S$. More precisely, since $\hat\tau(bm)=b^{\hat q}\hat\tau(m)$ for $b\in\CO_S$, the map $\hat\tau$ is $\hat\sigma$-semilinear and satisfies $\tau_{\hat M_n}(\hat\sigma_{\hat M_n}^*m)=\hat\tau(m)=F_{\hat q,\BG_{a,S}/S}\circ m$.

The canonical epimorphisms $\ulHM_{n+1}\onto\ulHM_n$ induce closed immersions $i_n\colon\Dr_{\hat q}(\ulHM_n)\into\Dr_{\hat q}(\ulHM_{n+1})$. The inductive limit $\Dr_{\hat q}(\ulHM):=\dirlim \Dr_{\hat q}(\ulHM_n)$ in the category of sheaves on the big \fppf-site of $S$ is a sheaf of $\BF_\epsilon\dbl z\dbr$-modules that satisfies the following

\begin{definition}\label{DefDivLocAndMod}
A \emph{$z$-divisible local Anderson module over $S$} is a sheaf of $\BF_\epsilon\dbl z\dbr$-modules $G$ on the big \fppf-site of $S$ such that
\begin{enumerate}
\item \label{DefZDivGpAxiom1}
$G$ is \emph{$z$-torsion}, that is $G = \dirlim G[z^n]$, where $G[z^n]:=\ker(z^n\colon G\to G)$,
\item \label{DefZDivGpAxiom2}
$G$ is \emph{$z$-divisible}, that is $z\colon G \to G$ is an epimorphism,
\item \label{DefZDivGpAxiom3}
For every $n$ the $\BF_\epsilon$-module $G[z^n]$ is representable by a finite, locally free, strict $\BF_\epsilon$-module scheme over $S$ in the sense of Faltings~\cite{Faltings02} and Abrashkin~\cite{Abrashkin}, and
\item \label{DefZDivGpAxiom4}
locally on $S$ there exist an integer $d \in \BZ_{\geq 0}$, such that
$(z-\zeta)^d=0$ on $\omega_G$ where $\omega_G := \invlim\omega_{G[z^n]}$ and $\omega_{G[z^n]}=e^*\Omega^1_{G[z^n]/S}$ is the pullback under the zero section $e\colon S\to G[z^n]$.
\end{enumerate}
A \emph{morphism of $z$-divisible local Anderson modules over $S$} is a morphism of \fppf-sheaves of $\BF_\epsilon\dbl z\dbr$-modules. The category of divisible local Anderson modules is $\BF_\epsilon\dbl z\dbr$-linear.
It is shown in \cite[Lemma~8.2 and Theorem~10.8]{HartlSingh} that $\omega_G$ is a finite locally free $\CO_S$-module and we define the \emph{dimension of $G$} as $\rk\omega_G$\,.
\end{definition}

Note that in \cite[Definition~6.2]{HartlAbSh} and W.~Kim \cite[Definition~7.3.1]{Kim} different definitions of $z$-divisible local Anderson modules were given. Unfortunately, the latter definitions are both wrong, because the strictness assumption in \ref{DefZDivGpAxiom3} is missing.

\begin{remark}\label{RemEquiv}
By \cite[Theorem~8.3]{HartlSingh} the functor $\ulHM\mapsto\Dr_{\hat q}(\ulHM)$ is an anti-equivalence between the category of effective local $\hat\sigma$-shtukas over $S$ and the category of $z$-divisible local Anderson modules over $S$. Moreover, it is $A_\epsilon$-linear and exact. Various properties are preserved under this anti-equivalence. More precisely let $\ulHM$ be an effective local $\hat\sigma$-shtuka over $S$ and let $G=\Dr_{\hat q}(\ulHM)$.

Then the $\CO_S\dbl z\dbr$-modules $\hat M/\tau_{\hat M}(\hat\sigma^*\hat M)$ and $\omega_{G}$ are canonically isomorphic. In particular, $\ulHM$ is \'etale, if and only if $\omega_{G}=(0)$, if and only if all $G[z^n]$ are \'etale.

The map $\tau_{\hat M}$ is topologically nilpotent, in the sense that locally on $S$ there is an integer $n$ such that $\im(\tau_{\hat M}^n\colon\hat\sigma^{n*}\hat M\to\hat M)\subset z\hat M$, if and only if $G$ is a formal Lie group.

If $\ulHM$ is bounded by $(d,0,\ldots,0)$ we say that $G$ is \emph{bounded by $d$}. In this case $(z-\zeta)^d\cdot\omega_G=(0)$ globally on $S$ in axiom~\ref{DefZDivGpAxiom4} and $\dim G=d$ as can be seen from the elementary divisor theorem applied to the pullback $s^*\ulHM$ to a closed point $s\colon\Spec\kappa(s)\to S$, where $\kappa(s)$ is the residue field at $s$.
\end{remark}

\begin{remark}\label{RemTopolNilpotent}
If $\ulHM$ is a local shtuka over a valuation ring $R$ as in Notation~\ref{Notation} and we view it as a projective system of local shtukas $\ulHM^{(i)}:=(\hat M^{(i)},\tau_{\hat M^{(i)}}):=\ulHM\otimes_R R/(\zeta^i)$ over $\Spec R/(\zeta^i)$ as in Remark~\ref{RemRelationBetweenTheTwoS}, then the following are equivalent.
\begin{enumerate}
\item \label{RemTopolNilpotent_A}
The map $\tau_{\hat M}$ is \emph{topologically nilpotent} in the sense that $\tau_{\hat M}^n(\sigma^{n*}\hat M)\subset \zeta\hat M+z\hat M$ for $n\gg0$.
\item \label{RemTopolNilpotent_B}
For all $i\in\BN$ the map $\tau_{\hat M^{(i)}}$ is topologically nilpotent, that is $\tau_{\hat M^{(i)}}^n(\hat\sigma^{n*}\hat M^{(i)})\subset z\hat M^{(i)}$ for $n\gg0$.
\item \label{RemTopolNilpotent_C}
There exists an $i\in\BN_{>0}$ for which the map $\tau_{\hat M^{(i)}}$ is topologically nilpotent.
\item \label{RemTopolNilpotent_D}
$\Dr_{\hat q}(\ulHM):=\dirlim[i]\Dr_{\hat q}(\ulHM^{(i)})$ is a formal Lie group over $\Spf R$.
\end{enumerate}
Indeed, \ref{RemTopolNilpotent_A} yields $\tau_{\hat M}^{jn}(\sigma^{jn*}\hat M)\subset \zeta^{1+\hat q^n+\ldots+\hat q^{(j-1)n}}\hat M+z\hat M$ for all $j\in\BN$. This implies \ref{RemTopolNilpotent_B}, from which \ref{RemTopolNilpotent_C} follows trivially. Conversely if \ref{RemTopolNilpotent_C} holds for some $i$ then $\tau_{\hat M}^n(\sigma^{n*}\hat M)\subset \zeta^i\hat M+z\hat M$ for $n\gg0$, whence \ref{RemTopolNilpotent_A}. Finally, by Remark~\ref{RemEquiv} assertion \ref{RemTopolNilpotent_B} implies \ref{RemTopolNilpotent_D}, and \ref{RemTopolNilpotent_D} implies that every $\Dr_{\hat q}(\ulHM^{(i)})$ is a formal Lie group, whence \ref{RemTopolNilpotent_B}.
\end{remark}

\begin{example}\label{ExAndModule}
Let $S=\Spec B\in\Nilp_{A_\epsilon}$ be affine and let $d$ and $r$ be positive integers. An \emph{abelian Anderson $A$-module of rank $r$ and dimension $d$ over $S$} is a pair $\ulE=(E,\phi)$ consisting of a smooth affine group scheme $E$ over $S$ of relative dimension $d$, and a ring homomorphism $\phi\colon A\to\End_{S\text{-groups}}(E),\,a\mapsto\phi_a$ such that
\begin{enumerate}
\item \label{DefAndModule_A}
there is a faithfully flat morphism $S'\to S$ for which $E\times_S S'\cong\BG_{a,S'}^d$ as $\BF_q$-module schemes,
\item \label{DefAndModule_B}
$\bigl(\phi_a-\gamma(a)\bigr)^d=0$ on $\omega_E:=e^*\Omega^1_{E/S}$ for all $a\in A$, where $e\colon S\to E$ is the zero section,
\item \label{DefAndModule_C}
the set $M:=M(\ulE):=\Hom_{S\text{-groups},\BF_q\text{-lin}}(E,\BG_{a,S})$ of $\BF_q$-equivariant homomorphisms of $S$-group schemes is a locally free module over $A_B:=A\otimes_{\BF_q}B$ of rank $r$ under the action given on $m\in M$ by
\[
\begin{array}{rll}
A\ni a\colon & M\longto M, & m\mapsto m\circ \phi_a\\[2mm]
B\ni b\colon & M\longto M, & m\mapsto b\circ m
\end{array}
\]
\end{enumerate}
In addition we consider the map $\tau\colon m\mapsto F_{q,\BG_{a,S}/S}\!\circ\, m$ on $m\in M$, where $F_{q,\BG_{a,S}/S}$ is the relative $q$-Frobenius of $\BG_{a,S}$ over $S$. Since $\tau(bm)=b^q\tau(m)$ the map $\tau$ is $\sigma$-semilinear and induces an $A_B$-linear map $\tau_M\colon\sigma^*M\to M$, which makes $\ulM(\ulE):=\bigl(M(\ulE),\tau_M)$ into an effective $A$-motive over $S$ in the sense of Example~\ref{ExAMotive}. The functor $\ulE\mapsto\ulM(\ulE)$ is fully faithful and its essential image is described in \cite[Theorem~3.5]{HartlIsog} generalizing Anderson's description~\cite[Theorem~1]{Anderson86}.

Now let $\ulE=(E,\phi)$ be an abelian Anderson $A$-module over $S$ and let $\ulHM:=\ulHM_\epsilon(\ulM(\ulE))$ be its associated effective local $\hat\sigma$-shtuka at $\epsilon$; see Example~\ref{ExAMotive}. Let $n\in\BN$ and let $\epsilon^n=(a_1,\ldots,a_s)\subset A$. Then 
\[
\ulE[\epsilon^n]\;:=\;\ker\bigl(\phi_{a_1,\ldots,a_s}:=(\phi_{a_1},\ldots,\phi_{a_s})\colon E \longto E^s\bigr)
\]
is called the \emph{$\epsilon^n$-torsion submodule of $\ulE$}. It is an $A/\epsilon^n$-module via $A/\epsilon^n\to\End_S(\ulE[\epsilon^n]),\,\bar a\mapsto\phi_a$ and independent of the set of generators of $\epsilon^n$; see \cite[Lemma~6.2]{HartlIsog}. Moreover, by \cite[Theorem~7.6]{HartlIsog} it is a finite $S$-group scheme of finite presentation and a strict $\BF_\epsilon$-module scheme and there are canonical $A/\epsilon^n$-equivariant isomorphisms of finite locally free $S$-group schemes
\begin{eqnarray*}
\Dr_{\hat q}(\ulHM/\epsilon^n\ulHM) & \isoto & \ulE[\epsilon^n]\qquad\text{and}\\[2mm]
\ulHM/\epsilon^n\ulHM & \isoto & \Hom_{S\text{-groups},\BF_\epsilon\text{-lin}}\bigl(\ulE[\epsilon^n]\,,\,\BG_{a,S}\bigr)
\end{eqnarray*}
of torsion local shtukas in the sense of Definition~\ref{DefTorLocSht} below. In particular, $\ulE[\epsilon^\infty]:=\dirlim\ulE[\epsilon^n]=\Dr_{\hat q}(\ulHM)$ is a $z$-divisible local Anderson module over $S$.

\end{example}

\begin{example}\label{ExCarlitz3}
We continue with Example~\ref{ExCarlitz2}. Let $A=\BF_q[t]$ and let $\epsilon=(z)\subset A$ be a maximal ideal generated by a monic prime element $z=z(t)\in \BF_q[t]$. Let $R=\BF_\epsilon\dbl\zeta\dbr$ and let $\ulCC:=(\BG_{a,R},\phi)$ with $\phi_t=\theta+F_{q,\BG_{a,R}}$ be the \emph{Carlitz module} over $R$. That is, if $\BG_{a,R}=\Spec R[x]$ then $\phi_t^*(x)=\theta\,x+x^q$. Then $\ulM(\ulCC)=(R[t],t-\theta)$ is the Carlitz $\BF_q[t]$-motive over $R$ from Example~\ref{ExCarlitz2} with associated local shtuka $\ulHM:=\ulHM_\epsilon(\ulM(\ulCC))\cong(R\dbl z\dbr, z-\zeta)$ at $\epsilon$. The $\epsilon^n$-torsion submodule of $\ulCC$ is $\ulCC[\epsilon^n]:=\ker(\phi_{z^n})=\Spec R[x]/(\phi_{z^n}^*(x))$. We can compute $\ulHM/z^n\ulHM=\bigoplus_{i=1}^n R\cdot m_i$ with $m_i=z^{i-1}$ and $\tau_{\hat M}(\hat\sigma_{\!\hat M}^*m_i)=(z-\zeta)m_i=m_{i+1}-\zeta m_i$ for $1\le i<n$ and $\tau_{\hat M}(\hat\sigma_{\!\hat M}^*m_n)=-\zeta m_n$. By Example~\ref{ExAndModule} this implies
\[
\ulCC[\epsilon^n]\;\cong\;\Dr_{\hat q}(\ulHM/\epsilon^n\ulHM)\;\cong\;\Spec R[m_1,\ldots,m_n]/(m_n^{\hat q}+\zeta m_n\,,\,m_i^{\hat q}-m_{i+1}+\zeta m_i\colon 1\le i<n)\,. 
\]
On it $z$ acts via $\phi_z^*(m_i)=m_{i+1}=\zeta m_i+m_i^{\hat q}$ for $1\le i<n$ and $\phi_z^*(m_n)=0$.

Since $\tau_{\hat M}(\sigma^*\hat M)=(z-\zeta)\hat M\subset\zeta\hat M+z\hat M$, Remark~\ref{RemTopolNilpotent} implies that $G:=\ulCC[\epsilon^\infty]\cong\Dr_{\hat q}(\ulHM)$ is a formal Lie group over $\Spf R$. Its dimension is $1$ because $\omega_G\cong\hat M/\tau_{\hat M}(\hat\sigma^*\hat M)=R$. Setting $x:=m_1$ it follows that $G\cong\Spf R\dbl x\dbr$ is the formal additive group scheme with the action of $A_\epsilon=\BF_\epsilon\dbl z\dbr$ given by $\phi_z^*(x)=\zeta\,x+x^{\hat q}$ and $\phi_a^*(x)=a\,x$ for $a\in\BF_\epsilon$. So we can alternatively describe $\ulCC[\epsilon^n]$ as $\Spec R[x]/(\phi_{z^n}^*(x))$ with the latter expression for $\phi_z^*$.
\end{example}

%
%

\section{Tate modules}\label{SectGalRep}
\setcounter{equation}{0}

With a local shtuka one can associate a Galois representation. More precisely, let $\ulHM=(\hat M,\tau_{\hat M})$ be a local shtuka over $\Spec R$ for a valuation ring $R$ as in Notation~\ref{Notation}. Then $\tau_{\hat M}$ induces an isomorphism $\tau_{\hat M}\colon\hat\sigma^*\hat M\otimes_{R\dbl z\dbr}K\dbl z\dbr\isoto\hat M\otimes_{R\dbl z\dbr}K\dbl z\dbr$, because $z-\zeta\in K\dbl z\dbr\mal$. So one could say that $\ulHM\otimes_{R\dbl z\dbr}K\dbl z\dbr$ is an ``\'etale local shtuka over $K$''. 

\begin{definition}\label{DefTateMod}
With $\ulHM$ as above one associates the (\emph{dual}) \emph{Tate module} 
\[
\check T_\epsilon\ulHM\;:=\;(\hat M\otimes_{R\dbl z\dbr}K^\sep\dbl z\dbr)^{\hat\tau}\;:=\;\bigl\{m\in\hat M\otimes_{R\dbl z\dbr}K^\sep\dbl z\dbr\colon \tau_{\hat M}(\hat\sigma_{\!\hat M}^*m)=m\bigr\}
\]
and the \emph{rational} (\emph{dual}) \emph{Tate module} 
\[
\check V_\epsilon\ulHM\;:=\;\bigl\{m\in\hat M\otimes_{R\dbl z\dbr}K^\sep\dpl z\dpr\colon \tau_{\hat M}(\hat\sigma_{\!\hat M}^*m)=m\bigr\}\;=\;\check T_\epsilon\ulHM\otimes_{A_\epsilon}Q_\epsilon\,.
\]
One also sometimes writes $\Koh^1_\epsilon(\ulHM,A_\epsilon)=\check T_\epsilon\ulHM$ and $\Koh^1_\epsilon(\ulHM,Q_\epsilon)=\check V_\epsilon\ulHM$ and calls this the \emph{$\epsilon$-adic realization of $\ulHM$}. By Proposition~\ref{PropTateMod} below, this defines a covariant functor $\check T_\epsilon\colon\ulHM\mapsto \check T_\epsilon\ulHM$ from the category of local shtukas over $R$ to the category $\Rep_{A_\epsilon}\Gal(K^\sep/K)$ of continuous representations of $\Gal(K^\sep/K)$ on finite free $A_\epsilon$-modules and a covariant functor $\check V_\epsilon\colon\ulHM\mapsto \check V_\epsilon\ulHM$ from the category of local shtukas over $R$ with quasi-morphisms to the category $\Rep_{Q_\epsilon}\Gal(K^\sep/K)$ of continuous representations of $\Gal(K^\sep/K)$ on finite dimensional $Q_\epsilon$-vector spaces. For $n\in\BN$ we also define
\[
(\ulHM/{z^n \ulHM})^{\hat\tau}(K^\sep)\;:=\;\bigl\{m\in\hat M\otimes_{R\dbl z\dbr}K^\sep\dbl z\dbr/(z^n)\colon \tau_{\hat M}(\hat\sigma_{\!\hat M}^*m)=m\bigr\}\;=\;\check T_\epsilon\ulHM/z^n\check T_\epsilon\ulHM\,,
\]
which is a free $A_\epsilon/(z^n)$-module of rank equal to the rank of $\ulHM$ with $\check T_\epsilon\ulHM=\invlim(\ulHM/{z^n \ulHM})^{\hat\tau}(K^\sep)$.
\end{definition}

\begin{proposition}\label{PropTateMod}
$\check T_\epsilon\ulHM$ is a free $A_\epsilon$-module of rank equal to $\rk\ulHM$ and $\check V_\epsilon\ulHM$ is a $Q_\epsilon$-vector space of dimension equal to $\rk\ulHM$. Both carry a continuous action of $\Gal(K^\sep/K)$. Moreover the inclusion $\check T_\epsilon\ulHM\subset\hat M\otimes_{R\dbl z\dbr}K^\sep\dbl z\dbr$ defines a canonical isomorphism of $K^\sep\dbl z\dbr$-modules 
\begin{equation}\label{EqTateModIsom}
\check T_\epsilon\ulHM\otimes_{A_\epsilon}K^\sep\dbl z\dbr\isoto\ulHM\otimes_{R\dbl z\dbr}K^\sep\dbl z\dbr
\end{equation}
which is functorial in $\ulHM$ and $\Gal(K^\sep/K)$- and $\hat\tau$-equivariant, where on the left module $\Gal(K^\sep/K)$-acts on both factors and $\hat\tau$ is $\id\otimes\hat\sigma$, and on the right module $\Gal(K^\sep/K)$ acts only on $K^\sep\dbl z\dbr$ and $\hat\tau$ is $(\tau_{\hat M}\circ\hat\sigma_{\!\hat M}^*)\otimes \hat\sigma$. In particular one can recover $\ulHM\otimes_{R\dbl z\dbr}K\dbl z\dbr=\bigl(\check T_\epsilon\ulHM\otimes_{A_\epsilon}K^\sep\dbl z\dbr\bigr)^{\Gal(K^\sep/K)}$ as the Galois invariants.
\end{proposition}

\begin{remark}\label{RemTateMod} 
We actually prove the stronger statement that for every $s>1/\hat q$ the isomorphism \eqref{EqTateModIsom} extends to a $\Gal(K^\sep/K)$- and $\hat\tau$-equivariant isomorphism
\begin{equation}\label{EqTateModIsomConv}
h\colon\;\check T_\epsilon\ulHM\otimes_{A_\epsilon}K^\sep\langle\tfrac{z}{\zeta^s}\rangle\;\isoto\;\ulHM\otimes_{R\dbl z\dbr}K^\sep\langle\tfrac{z}{\zeta^s}\rangle,
\end{equation}
which is functorial in $\ulHM$. Here for a field extension $K$ of $\BF_\epsilon\dpl\zeta\dpr$ and an $s\in\BR_{>0}$ we use the notation
\begin{eqnarray*}
K\langle\tfrac{z}{\zeta^s}\rangle & := & \bigl\{\,\sum_{i=0}^\infty b_iz^i\colon b_i\in K,\,|b_i|\,|\zeta|^{si}\to0\;(i\to+\infty)\,\bigr\}\,.
\end{eqnarray*}
These are subrings of $K\dbl z\dbr$ and the endomorphism $\hat\sigma\colon\sum_i b_iz^i\mapsto\sum_i b_i^{\hat q}z^i$ of $K\dbl z\dbr$ restricts to a homomorphism $\hat\sigma\colon K\langle\tfrac{z}{\zeta^s}\rangle\to K\langle\tfrac{z}{\zeta^{s\hat q}}\rangle$. Note that the $\hat\tau$-equivariance of $h$ means $h\otimes\id_{K\langle\frac{z}{\zeta^{s\hat q}}\rangle}=\tau_{\hat M}\circ\hat\sigma^*h$.
\end{remark}

\begin{proof}[Proof of Proposition~\ref{PropTateMod} and Remark~\ref{RemTateMod}]
That $\check T_\epsilon\ulHM$ is an $A_\epsilon$-module and $\check V_\epsilon\ulHM$ is a $Q_\epsilon$-vector space comes from the fact that the subring of $\hat\sigma$-invariants in $K^\sep\dbl z\dbr$ is $\BF_\epsilon\dbl z\dbr=A_\epsilon$.

We set $r=\rk\ulHM$, choose an $R\dbl z\dbr$-basis of $\hat M$, and write $\tau_{\hat M}$ with respect to this basis as a matrix $T\in\GL_r\bigl(R\dbl z\dbr[\tfrac{1}{z-\zeta}]\bigr)$. Since $z-\zeta$ is a unit in $K\langle\tfrac{z}{\zeta^{s\hat q}}\rangle$ for every $1\ge s>1/\hat q$ with $(z-\zeta)^{-1}=-\sum_{i=0}^\infty\zeta^{-i-1}z^i$, there is an inclusion $R\dbl z\dbr[\tfrac{1}{z-\zeta}]\into K\langle\tfrac{z}{\zeta^{s\hat q}}\rangle$, and we consider $T$ as a matrix in $\GL_r\bigl(K\langle\tfrac{z}{\zeta^{s\hat q}}\rangle\bigr)$. We claim that there is a matrix $U\in\GL_r\bigl(K^\sep\dbl z\dbr\bigr)$ with $\hat\sigma(U)=T^{-1}U$. We write $T^{-1}=\sum_{i=0}^\infty T_iz^i$ and $U=\sum_{n=0}^\infty U_nz^n$. 
We must solve the equations 
\begin{equation}\label{EqU0}
\hat\sigma(U_0)=T_0\cdot U_0\qquad\text{and}\qquad \hat\sigma(U_n)=\sum_{i=0}^{n} T_{n-i}\cdot U_i\,. 
\end{equation}
By Lang's theorem~\cite[Corollary on p.~557]{Lang} there exists a matrix $U_0\in\GL_r(K^\alg)$ satisfying \eqref{EqU0}. Since the morphism $\GL_r(K^\alg)\to\GL_r(K^\alg),\,U_0\mapsto U_0\cdot\hat\sigma(U_0)^{-1}$ is \'etale, we actually have $U_0\in\GL_r(K^\sep)$. Then the second equation takes the form 
\[
\hat\sigma(U_0^{-1}U_n)- U_0^{-1}U_n=\sum_{i=0}^{n-1}\hat\sigma(U_0)^{-1}T_{n-i}U_0\cdot(U_0^{-1}U_i)\,.
\]
This is a system of Artin-Schreier equations for the coefficients of $U_0^{-1}U_n$ which can be solved in $K^\sep$ and this establishes our claim. 

We show that the matrix $U\in\GL_r\bigl(K^\sep\dbl z\dbr\bigr)$ obtained in this way lies in $\GL_r\bigl(K^\sep\langle\tfrac{z}{\zeta^s}\rangle\bigr)$ for every $1\ge s>1/\hat q$. Since $T^{-1}\in\GL_r\bigl(K\langle\tfrac{z}{\zeta^{s\hat q}}\rangle\bigr)\subset\GL_r\bigl(K\langle\tfrac{z}{\zeta^{\hat q}}\rangle\bigr)$ there is a constant $c\ge 1$ with $|T_i \zeta^{\hat qi}|\leq c$ for all $i$, where $|T_i \zeta^{\hat qi}|$ denotes the maximal absolute value of the entries of the matrix $T_i \zeta^{\hat qi}$. We write \eqref{EqU0} as
\[
\hat\sigma\bigl(U_n \zeta^{n}\bigr)\es = \es \sum_{i=0}^n \bigl(T_{n-i} \zeta^{\hat q(n-i)}\bigr) \bigl(U_i \zeta^{i}\bigr) \,\zeta^{i(\hat q-1)}\,.
\]
In view of $|\zeta|<1$ this implies the estimate
\[
|U_n\zeta^{n}|^{\hat q}\es \leq \es c\cdot \max\{\,|U_i\zeta^{i}|:\,0\leq i\leq n\,\}
\]
from which induction yields $|U_n\zeta^{n}|\leq c^{1/(\hat q-1)}$ for all $n\ge0$. 
In particular, if $s>1/\hat q$, then $s\hat q>1$ and $U\in M_r\bigl(K^\sep\langle\frac{z}{\zeta^{s\hat q}}\rangle\bigr)$. But now the equation $\hat\sigma(U)=T^{-1}U$ shows that
$\hat\sigma(U)\in M_r\bigl(K^\sep\langle\frac{z}{\zeta^{s\hat q}}\rangle\bigr)$, hence $U\in M_r\bigl(K^\sep\langle\frac{z}{\zeta^s}\rangle\bigr)$. A similar reasoning with the equation $\hat\sigma(U^{-1})=U^{-1}T$ shows that also $U^{-1}\in M_r\bigl(K^\sep\langle\frac{z}{\zeta^s}\rangle\bigr)$ and $U\in\GL_r\bigl(K^\sep\langle\tfrac{z}{\zeta^s}\rangle\bigr)$ as desired. 

Multiplication with $U$ provides an isomorphism $(K^\sep\langle\tfrac{z}{\zeta^s}\rangle^{\oplus r},\hat\tau=\id)\isoto\ulHM\otimes_{R\dbl z\dbr}K^\sep\langle\tfrac{z}{\zeta^s}\rangle$. Since $(K^\sep\langle\tfrac{z}{\zeta^s}\rangle)^{\hat\sigma}=A_\epsilon$, it follows that $\check T_\epsilon\ulHM=U\!\cdot\! A_\epsilon^{\oplus r}$ is free of rank $r$, and the inclusion $\check T_\epsilon\ulHM\subset U\!\cdot\!K^\sep\langle\tfrac{z}{\zeta^s}\rangle^{\oplus r}$ induces the equivariant isomorphism \eqref{EqTateModIsomConv}. Since $K^\sep\langle\tfrac{z}{\zeta^s}\rangle\subset K^\sep\dbl z\dbr$ this induces the isomorphism \eqref{EqTateModIsom}.

The continuity of the Galois representation $\rho_\ulHM\colon\Gal(K^\sep/K)\to\Aut_{A_\epsilon}(\check T_\epsilon\ulHM)$ means that for all $n\in\BN$ the subgroup $\{\,g\in\Gal(K^\sep/K)\colon \rho_\ulHM(g)\equiv\id\mod z^n\,\}$ is open. This is true because $\rho_\ulHM$ is given by the Galois action on the coefficients of the matrix $U$, and so this subgroup contains the open subgroup $\Gal\bigl(K^\sep/K(U_0,\ldots,U_{n-1})\bigr)$. Finally, the functoriality of $h$ is clear.
\end{proof}

\begin{remark}\label{RemTateModGeneral}
There is a statement similar to Proposition~\ref{PropTateMod} for an \'etale local shtuka $\ulHM$ over a connected scheme $S\in\Nilp_{A_\epsilon}$. For a geometric base point $\bar s\in S$ the (dual) Tate module and the rational (dual) Tate module
\[
\check{T}_\epsilon\ulHM :=(\ulHM\otimes_{\CO_S\dbl z\dbr} \kappa(\bar{s})\dbl z\dbr)^{\hat\tau}\qquad\text{and}\qquad\check{V}_\epsilon\ulHM :=\check{T}_\epsilon\ulHM\otimes_{A_\epsilon}Q_\epsilon
\]
are free of rank $\rk\ulHM$ and carry a continuous action of the \'etale fundamental group $\pi_1^\et(S,\bar s)$. Moreover, the functor $\ulHM\mapsto\check{T}_\epsilon\ulHM$ is an equivalence between the category of \'etale local shtukas over $S$ and the category of representations of $\pi_1^\et(S,\bar s)$ on finite free $A_\epsilon$-modules. Similarly, the functor $\ulHM\mapsto\check{V}_\epsilon\ulHM$ is an equivalence between the category of \'etale local shtukas over $S$ with quasi-morphisms and the category of representations of $\pi_1^\et(S,\bar s)$ on finite $Q_\epsilon$-vector spaces; see \cite[Proposition~3.4]{AH_Local}.
\end{remark}

Back in the situation over a valuation ring $R$ as in Notation~\ref{Notation} there is the following

\begin{corollary}\label{CorHomFinGen}
For local shtukas $\ulHM$ and $\ulHM{}'$ over $R$ the $A_\epsilon$-module $\Hom_R(\ulHM,\ulHM{}')$ is finite free of rank at most $\rk\ulHM\cdot\rk\ulHM{}'$.
\end{corollary}

\begin{proof}
First of all $\Hom_R(\ulHM,\ulHM{}')$ is $A_\epsilon$-torsion free, because $\ulHM{}'$ is. The functor $\ulHM\mapsto\check T_\epsilon\ulHM$ is faithful by \eqref{EqTateModIsom}, because $\hat M\subset\hat M\otimes_{R\dbl z\dbr}K^\sep\dbl z\dbr$. Since $\check T_\epsilon\ulHM$ is a free $A_\epsilon$-module of rank $\rk\ulHM$, it follows that $\Hom_R(\ulHM,\ulHM{}')$ is a finitely generated $A_\epsilon$-module, and hence free of rank at most $\rk\ulHM\cdot\rk\ulHM{}'$.
\end{proof}

\begin{proposition}\label{PropTateModAMotive}
If $\ulHM=\ulHM_\epsilon(\ulM)$ for an $A$-motive $\ulM=(M,\tau_M)$ over $R$ as in Example~\ref{ExAMotive}, the \emph{$\epsilon$-adic (dual) Tate module of $\ulM$} defined by $\check T_\epsilon\ulM:=\{m\in M\otimes_{A_R}A_{\epsilon,K^\sep}\colon \tau_M(\sigma_{\!M}^*m)=m\}$ is canonically isomorphic to $\check T_\epsilon\ulHM$ as representations of $\Gal(K^\sep/K)$. This isomorphism is functorial in $\ulM$.
\end{proposition}

\begin{proof}
Consider the decomposition $M\otimes_{A_R}A_{\epsilon,K^\sep}=\prod\limits_{i\in\BZ/f\BZ}M\otimes_{A_R}K^\sep\dbl z\dbr$ discussed in Example~\ref{ExAMotive}. An element $m=(m_i)_i\in\prod\limits_{i\in\BZ/f\BZ}M\otimes_{A_R}K^\sep\dbl z\dbr$ satisfies $\tau_M(\sigma_{\!M}^*m)=m$ if and only if $m_{i+1}=\tau_M(\sigma_{\!M}^*m_i)$ for all $i$ and $m_0=\tau_M^f(\sigma_{\!M}^{f*}m_0)=\tau_{\hat M}(\hat\sigma_{\!\hat M}^*m_0)$. So the isomorphism $\check T_\epsilon\ulM\isoto\check T_\epsilon\ulHM$ is given by $(m_i)_i\mapsto m_0$. It clearly is functorial.
\end{proof}

\begin{remark}\label{RemTateModAMotive}
The arguments proving Remark~\ref{RemTateMod} and Proposition~\ref{PropTateModAMotive} can also be applied to an $A$-motive $\ulM$ over $K$ as in Example~\ref{ExAMotive} which not necessarily has good reduction. Namely, for every $s>1/\hat q$ the decomposition \eqref{EqDecompAEpsilonR} yields 
\[
A_{\epsilon,R}\otimes_{R\dbl z\dbr}K^\sep\langle\tfrac{z}{\zeta^s}\rangle\;=\;\prod\limits_{i\in\BZ/f\BZ}A_{\epsilon,R}/\Fa_i\otimes_{R\dbl z\dbr}K^\sep\langle\tfrac{z}{\zeta^s}\rangle\;=\;\prod\limits_{i\in\BZ/f\BZ}K^\sep\langle\tfrac{z}{\zeta^s}\rangle\,.
\]
Therefore the decomposition $M\otimes_{A_K}A_{\epsilon,K^\sep}=\prod\limits_{i\in\BZ/f\BZ}M\otimes_{A_K}K^\sep\dbl z\dbr$ extends to a decomposition $M\otimes_{A_K}(A_{\epsilon,R}\otimes_{R\dbl z\dbr}K^\sep\langle\tfrac{z}{\zeta^s}\rangle)=\prod\limits_{i\in\BZ/f\BZ}M\otimes_{A_K}K^\sep\langle\tfrac{z}{\zeta^s}\rangle$. We denote the factor for $i=0$ by $M^{^{\SSC\land}}_0$ and equip it with the Frobenius $\tau_0^{\SSC\land}:=\tau_M^f\mod\Fa_0\colon\hat\sigma^*M^{^{\SSC\land}}_0[\tfrac{1}{z-\zeta}]\isoto M^{^{\SSC\land}}_0[\tfrac{1}{z-\zeta}]$. Since $K^\sep\langle\tfrac{z}{\zeta^s}\rangle$ is a principal ideal domain by \cite[Proposition~2]{Lazard}, we may choose a $K^\sep\langle\tfrac{z}{\zeta^s}\rangle$-basis of $M^{^{\SSC\land}}_0$ and represent $\tau_0^{\SSC\land}$ by a matrix $T\in\GL_r\bigl(K^\sep\langle\tfrac{z}{\zeta^{s\hat q}}\rangle\bigr)$. Now the argument proving Proposition~\ref{PropTateModAMotive} and Remark~\ref{RemTateMod} shows that $\check T_\epsilon\ulM=\{\,m\in M^{^{\SSC\land}}_0\colon\tau_0^{\SSC\land}(\sigma^{f*}_M m)=m\,\}$ and that there is a $\Gal(K^\sep/K)$- and $\hat\tau$-equivariant isomorphism
\begin{equation}\label{EqTateModIsomConvAMotive}
h\colon\;\check T_\epsilon\ulM\otimes_{A_\epsilon}K^\sep\langle\tfrac{z}{\zeta^s}\rangle\;\isoto\;(M^{^{\SSC\land}}_0, \tau_0^{\SSC\land})\;=\;\ulM\otimes_{A_K}(A_{\epsilon,R}/\Fa_0\otimes_{R\dbl z\dbr}K^\sep\langle\tfrac{z}{\zeta^s}\rangle)
\end{equation}
which is functorial in $\ulM$.
\end{remark}

\begin{definition}\label{DefTateModAModule}
Let $R$ be a valuation ring as in Notation~\ref{Notation}. For a $z$-divisible local Anderson module $G$ over $R$ the \emph{Tate module} $T_\epsilon G$ and the \emph{rational Tate module} $V_\epsilon G$ are defined as
\begin{eqnarray*}
T_\epsilon G & := & \Hom_{A_\epsilon}\bigl(Q_\epsilon/A_\epsilon,\,G(K^\sep)\bigr)\qquad\text{and}\\[2mm]
 V_\epsilon G & := & \Hom_{A_\epsilon}\bigl(Q_\epsilon,\,G(K^\sep)\bigr) \es \cong \es T_\epsilon G\otimes_{A_\epsilon}Q_\epsilon\,,
\end{eqnarray*}
where the last isomorphism sends $f\otimes z^{-n}\in T_\epsilon G\otimes_{A_\epsilon}Q_\epsilon$ to the homomorphism $\tilde f\colon Q_\epsilon\to G(K^\sep)$ with $\tilde f(a):=f(az^{-n})$ for $a\in Q_\epsilon$. To see that it is indeed an isomorphism, note that it is clearly injective, because $f$ can be recovered from $\tilde f$ and $n$. Conversely, since every $\tilde f\in V_\epsilon G$ satisfies $\tilde f(1)\in G(K^\sep)=\dirlim G[z^n](K^\sep)$ by \cite[Lemma~5.4]{HV1}, there is an $n$ with $\tilde f(1)\in G[z^n](K^\sep)$, and so $\tilde f(z^nA_\epsilon)=0$. This shows that $\tilde f$ is the image of $f\otimes z^{-n}$ for $f\in T_\epsilon G$ with $f(a):=\tilde f(az^n)$. 

$T_\epsilon G$ is an $A_\epsilon$-module and $V_\epsilon G$ is a $Q_\epsilon$-vector space. Both carry a continuous $\Gal(K^\sep/K)$-action and $G\mapsto T_\epsilon G$ and $G\mapsto V_\epsilon G$ are covariant functors. 

For an abelian Anderson $A$-module $\ulE$ over $R$ as in Example~\ref{ExAndModule} the \emph{$\epsilon$-adic Tate module} $T_\epsilon\ulE$ and the \emph{rational $\epsilon$-adic Tate module} $V_\epsilon\ulE$ are defined as
\[
T_\epsilon\ulE\;:=\;\Hom_A\bigl(Q_\epsilon/A_\epsilon,\,\ulE(K^\sep)\bigr)\qquad\text{and}\qquad V_\epsilon \ulE\::=\;T_\epsilon\ulE\otimes_{A_\epsilon}Q_\epsilon\,.
\]
Since every element of $Q_\epsilon/A_\epsilon$ is annihilated by a power of $\epsilon$, we have $T_\epsilon\ulE=T_\epsilon\,\ulE[\epsilon^\infty]$.
\end{definition}

After choosing a uniformizing parameter $z$ of $A_\epsilon$ there are isomorphisms
\begin{eqnarray*}
T_\epsilon G & \isoto & \bigl\{(P_n)_n\in\prod_{n\in\BN_{0}}G[z^n](K^\sep)\colon z(P_{n+1})=P_n\bigr\} \es =: \es \invlim[n]\bigl(G[z^n](K^\sep),z\bigr)\qquad\text{and}\\[2mm]
V_\epsilon G & \isoto & \bigl\{(P_n)_n\in\prod_{n\in\BZ}G(K^\sep)\colon z(P_{n+1})=P_n\bigr\}\,,
\end{eqnarray*}
which send $f\colon Q_\epsilon\to G(K^\sep)$ to the tuple $P_n:=f(z^{-n})$. These are indeed isomorphisms, because from $(P_n)_n$ we can reconstruct $f$ as follows. Every $a\in Q_\epsilon\mal$ is of the form $a=uz^{-n}$ for an integer $n$ and a unit $u\in A_\epsilon\mal$. Then $f(a)=u(P_n)$. 

\medskip

To describe the relation between $T_\epsilon G$ and $\check T_\epsilon\ulHM$ consider the $A_\epsilon$-module $\Hom_{\BF_\epsilon}(Q_\epsilon/A_\epsilon,\BF_\epsilon)$ which carries the trivial Galois action and is canonically isomorphic to the module of continuous differential forms $\wh\Omega^1_{A_\epsilon/\BF_\epsilon}$ under the map
\begin{equation}\label{EqIsomDifferentials}
\wh\Omega^1_{A_\epsilon/\BF_\epsilon} \; \isoto \; \Hom_{\BF_\epsilon}(Q_\epsilon/A_\epsilon,\BF_\epsilon)\,,\quad\omega\longmapsto\bigl(a\mapsto \Res_\epsilon(a\omega)\bigr)\,.
\end{equation}
After choosing a uniformizing parameter $z$ of $A_\epsilon$ we can identify $\wh\Omega^1_{A_\epsilon/\BF_\epsilon}\cong\BF_\epsilon\dbl z\dbr dz$ and the inverse map is given by $\Hom_{\BF_\epsilon}(Q_\epsilon/A_\epsilon,\BF_\epsilon)\to\BF_\epsilon\dbl z\dbr dz,\;\lambda\mapsto\sum_{i=0}^\infty \lambda(z^{-1-i})z^idz$. In particular $\Hom_{\BF_\epsilon}(Q_\epsilon/A_\epsilon,\BF_\epsilon)$ is non-canonically isomorphic to $A_\epsilon$.

The following proposition generalizes Anderson's result \cite[\S\,4.2]{Anderson93} who treated the case where $G$ is a formal Lie group.

\begin{proposition}\label{PropCompTateMod}
Let $\ulHM$ be an effective local shtuka over $R$, let $G=\Dr_{\hat q}(\ulHM)$ be the associated $z$-divisible local Anderson module from Section~\ref{SectDivLocAM} and view $\ulHM$ as $\invlim\Hom_{R\text{\rm-groups},\BF_\epsilon\text{\rm-lin}}(G[z^n]\,,\,\BG_{a,R})$. Then there is a $\Gal(K^\sep/K)$-equivariant perfect pairing of $A_\epsilon$-modules
\begin{equation}\label{EqPairingTateMod}
\langle\,.\,,\,.\,\rangle\colon \;T_\epsilon G\times \check T_\epsilon\ulHM\;\longto\;\Hom_{\BF_\epsilon}(Q_\epsilon/A_\epsilon,\BF_\epsilon)\;\cong\;\wh\Omega^1_{A_\epsilon/\BF_\epsilon}\,,\quad\langle f, m\rangle:= m\circ f\,.
\end{equation}
which identifies $T_\epsilon G$ with the contragredient $\Gal(K^\sep/K)$-representation $\Hom_{A_\epsilon}(\check T_\epsilon\ulHM,\,\wh\Omega^1_{A_\epsilon/\BF_\epsilon})$ of $\check T_\epsilon\ulHM$. In particular $T_\epsilon G$ is a free $A_\epsilon$-module of rank equal to $\rk\ulHM$ that carries a continuous action of $\Gal(K^\sep/K)$. 
\end{proposition}

\begin{remark}
(a) Note that indeed $m\circ f$ lies in $\Hom_{\BF_\epsilon}(Q_\epsilon/A_\epsilon,\BF_\epsilon)$, because $m=\tau_{\hat M}(\sigma_{\!\hat M}^*m) = F_{\hat q,\BG_{a,S}/S}\circ m$ implies that $m\circ f(a)=F_{\hat q,\BG_{a,S}/S}\bigl(m\circ f(a)\bigr)=\bigl(m\circ f(a)\bigr)^{\hat q}$ in $\BG_a(K^\sep)=K^\sep$, and hence $m\circ f(a)\in\BF_\epsilon$ for all $a\in Q_\epsilon$.

\medskip\noindent
(b) This pairing is functorial in $\ulHM$ in the following sense. If $\alpha\colon\ulHM\to\ulHM'$ is a morphism of local shtukas over $R$ and $\Dr_{\hat q}(\alpha)\colon G':=\Dr_{\hat q}(\ulHM')\to\Dr_{\hat q}(\ulHM)=G$ is the induced morphism of the associated $z$-divisible local Anderson modules then $\check T_\epsilon\alpha\colon\check T_\epsilon\ulHM\to\check T_\epsilon\ulHM'$ and $T_\epsilon\Dr_{\hat q}(\alpha)\colon T_\epsilon G'\to T_\epsilon G$ satisfy $\langle f',\check T_\epsilon\alpha(m)\rangle=m\circ\Dr_{\hat q}(\alpha)\circ f'=\langle T_\epsilon\Dr_{\hat q}(\alpha)(f'),m\rangle$ for $f'\in T_\epsilon G'$ and $m\in\check T_\epsilon\ulHM$.
\end{remark}

\begin{proof}[{Proof of Proposition~\ref{PropCompTateMod}}]
From \cite[Lemma~2.4 and Theorem~8.6]{BoeckleHartl} we have a  perfect pairing of $A_\epsilon/(z^n)$-modules
\begin{eqnarray*}
{G[z^n](K^\sep)}  \times  {(\ulHM/{z^n \ulHM})^{\hat\tau}(K^\sep)} & \longrightarrow & \Hom_{A_\epsilon}(A_\epsilon/(z^n) , \BF_\epsilon)\,,\\[2mm]
(P_n, m) \qquad\qquad\qquad\es & \longmapsto & \bigl(h_n\colon a \mapsto m(a\,P_n)\bigr)\,.
\end{eqnarray*}
Under the $A_\epsilon$-isomorphism $\Hom_{A_\epsilon}\bigl(z^{-n}A_\epsilon/A_\epsilon,\,G(K^\sep)\bigr)\isoto G[z^n](K^\sep)$, $f\mapsto P_n:=f(z^{-n})$ it transforms into a perfect pairing of $A_\epsilon/(z^n)$-modules
\begin{eqnarray*}
\Hom_{A_\epsilon}\bigl(z^{-n}A_\epsilon/A_\epsilon,\,G(K^\sep)\bigr) \times  {(\ulHM/{z^n \ulHM})^{\hat\tau}(K^\sep)} & \longrightarrow & \Hom_{\BF_\epsilon}(z^{-n}A_\epsilon/A_\epsilon, \BF_\epsilon)\,,\\[2mm]
(f, m) \qquad\qquad\qquad\quad & \longmapsto & \bigl(h_n\colon a \mapsto m\circ f(a)\bigr)\,.
\end{eqnarray*}
Again the identification $\Hom_{\BF_\epsilon}(z^{-n}A_\epsilon/A_\epsilon, \BF_\epsilon)\,=\,A_\epsilon/(z^n)\cdot dz$ induced from \eqref{EqIsomDifferentials}, shows that this is a free $A_\epsilon/(z^n)$-module of rank one. It follows that 
\[
\Hom_{A_\epsilon}\bigl(z^{-n}A_\epsilon/A_\epsilon,\,G(K^\sep)\bigr)\;\cong\; \Hom_{A_\epsilon/(z^n)}\bigl((\ulHM/{z^n \ulHM})^{\hat\tau}(K^\sep),\,A_\epsilon/(z^n)\cdot dz\bigr)
\]
is a free $A_\epsilon/(z^n)$-module of rank $\rk\ulHM$, because this holds for $(\ulHM/{z^n \ulHM})^{\hat\tau}(K^\sep)$. This implies that for varying $n$ the inclusions $z^{-n}A_\epsilon\subset z^{-(n+1)}A_\epsilon$ induce vertical maps in the following diagram
\begin{equation}\label{EqSurjDiag}
\xymatrix @R-0.5pc {
**{!L(0) =<8.5pc,1.5pc>} \objectbox{\Hom_{A_\epsilon}\bigl(z^{-n-1}A_\epsilon/A_\epsilon,\,G(K^\sep)\bigr) } \ar[d] & { \times \es (\ulHM/{z^{n+1} \ulHM})^{\hat\tau}(K^\sep)} \ar[r] \ar@{->>}[d] & A_\epsilon/(z^{n+1})\cdot dz \ar@{->>}[d] \\
**{!L(0) =<7pc,1.5pc>} \objectbox{\Hom_{A_\epsilon}\bigl(z^{-n}A_\epsilon/A_\epsilon,\,G(K^\sep)\bigr) } & {\times \quad (\ulHM/{z^n \ulHM})^{\hat\tau}(K^\sep)\es} \ar[r] & A_\epsilon/(z^n)\cdot dz\,.
}
\end{equation}
By evaluating these pairings on a fixed choosen $A_\epsilon$-basis of $\check T_\epsilon\ulHM=\invlim(\ulHM/{z^n \ulHM})^{\hat\tau}(K^\sep)$, we obtain isomorphisms 
\[
\Hom_{A_\epsilon}\bigl(z^{-n}A_\epsilon/A_\epsilon,\,G(K^\sep)\bigr)\;\isoto\;\bigl(A_\epsilon/(z^n)\cdot dz\bigr)^{\rk\ulHM},
\]
which are compatible for all $n$. Since the middle and the right vertical maps in diagram \eqref{EqSurjDiag} are surjective, this shows that also the left vertical map is surjective. So the projective limit of this diagram is the pairing \eqref{EqPairingTateMod}, and this yields an isomorphism
\[
T_\epsilon G\;=\;\Hom_{A_\epsilon}\bigl(Q_\epsilon/A_\epsilon,\,G(K^\sep)\bigr)\;\isoto\;\bigl(A_\epsilon dz\bigr)^{\rk\ulHM}\;\cong\;\Hom_{A_\epsilon}\bigl(\check T_\epsilon\ulHM,\,\wh\Omega^1_{A_\epsilon/\BF_\epsilon}\bigr)\,.
\]
This shows that \eqref{EqPairingTateMod} is a perfect pairing of free $A_\epsilon$-modules. If $g \in \Gal(K^\sep/K)$ then $\langle g(f),g(m)\rangle\,=\,g(m)\circ g(f)\,=\,g(m\circ f)\,=\,m\circ f\,=\,\langle f,m\rangle$, because $g$ acts trivially on $m\circ f$.
Therefore, the pairing \eqref{EqPairingTateMod} is $\Gal(K^\sep/K)$-equivariant.
\end{proof}

\begin{example}\label{ExCarlitz4A}
We describe the $\epsilon$-adic (dual) Tate module $\check T_\epsilon\ulM=\check T_\epsilon\ulHM_\epsilon(\ulM)$ of the Carlitz motive $\ulM=(R[t],t-\theta)$ from Example~\ref{ExCarlitz2} by using the local shtuka $\ulHM:=\ulHM_\epsilon(\ulM)=(R\dbl z\dbr,z-\zeta)$ computed there. For all $i\in\BN_0$ let $\tplusminus_i\in K^\sep$ be solutions of the equations $\tplusminus_0^{\hat q-1}=-\zeta$ and $\tplusminus_i^{\hat q}+\zeta \tplusminus_i=\tplusminus_{i-1}$. This implies $|\tplusminus_i|=|\zeta|^{\hat q^{-i}/(\hat q-1)}<1$. Define the power series $\tplus=\sum_{i=0}^\infty \tplusminus_iz^i\in\CO_{K^\sep}\dbl z\dbr$. It satisfies $\hat\sigma(\tplus)=(z-\zeta)\!\cdot\!\tplus$, but depends on the choice of the $\tplusminus_i$. A different choice yields a different power series $\ttplus$ which satisfies $\ttplus=u\tplus$ for a unit $u\in(K^\sep\dbl z\dbr\mal)^{\hat\sigma=\id}=A_\epsilon^{^{\SSC\times}}$, because $\hat\sigma(u)=\frac{\hat\sigma(\ttplus)}{\hat\sigma(\tplus)}=\frac{\ttplus}{\tplus}=u$. The field extension $\BF_\epsilon\dpl\zeta\dpr(\tplusminus_i\colon i\in\BN_0)$ of $\BF_\epsilon\dpl\zeta\dpr$ is the function field analog of the cyclotomic tower $\BQ_p(\sqrt[p^i]{1}\colon i\in\BN_0)$; see \cite[\S\,1.3 and \S\,3.4]{HartlDict}. There is an isomorphism of topological groups called the \emph{$\epsilon$-adic cyclotomic character}
\[
\chi_\epsilon\colon\Gal\bigl(\BF_\epsilon\dpl\zeta\dpr(\tplusminus_i\colon i\in\BN_0)\big/\BF_\epsilon\dpl\zeta\dpr\bigr)\;\isoto\; A_\epsilon^{^{\SSC\times}},
\]
which satisfies $g(\tplus):=\sum_{i=0}^\infty g(\tplusminus_i)z^i=\chi_\epsilon(g)\cdot \tplus$ in $K^\sep\dbl z\dbr$ for $g$ in the Galois group. It is independent of the choice of the $\tplusminus_i$. The $\epsilon$-adic (dual) Tate module $\check T_\epsilon\ulHM$ of $\ulHM$ and $\ulM$ is generated by $\tplus^{-1}$ on which the Galois group acts by the inverse of the cyclotomic character. According to Proposition~\ref{PropCompTateMod}, $\Gal(K^\sep/K)$ acts on the $\epsilon$-adic Tate module $T_\epsilon\,\ulCC$ of the Carlitz module $\ulCC$ from Example~\ref{ExCarlitz3} by $\chi_\epsilon$. So using the notation of Tate twists we may write $\check T_\epsilon\ulM=A_\epsilon(-1)$ and $\check V_\epsilon\ulM=Q_\epsilon(-1)$, as well as $T_\epsilon\,\ulCC=A_\epsilon(1)$ and $V_\epsilon\,\ulCC=Q_\epsilon(1)$. The isomorphisms \eqref{EqTateModIsom} and \eqref{EqTateModIsomConv} for $s>1/\hat q$ are given by sending the generator $\tplus^{-1}$ of $\check T_\epsilon\ulHM$ to $\tplus^{-1}\in K^\sep\langle\tfrac{z}{\zeta^s}\rangle\mal$.

We compute $T_\epsilon\,\ulCC$ explicitly. By Example~\ref{ExCarlitz3} the group scheme $\ulCC[\epsilon^n]$ equals the kernel of the endomorphism $\phi_z^n$ of $\BG_{a,R}=\Spec R[x]$ where $\phi_z^*(x)=\zeta x+x^{\hat q}$. By definition $\phi_z(\tplusminus_0)=0$ and $\phi_z(\tplusminus_i)=\tplusminus_{i-1}$ for all $i$. This means that $\ulCC[\epsilon^n](K^\sep)=A/\epsilon^n\cdot\tplusminus_{n-1}$ and $T_\epsilon\,\ulCC=A_\epsilon\cdot(\tplusminus_{n-1})_n$ with the action of $g\in\Gal(K^\sep/K)$ given by $g(\tplusminus_{n-1})_n=\chi_\epsilon(g)\cdot(\tplusminus_{n-1})_n$. In this respect the power series $\tplus=\sum_{n=0}^\infty\tplusminus_n z^n$ is the \emph{Anderson generating function} from \cite[\S\,4.2]{Anderson93} of the $z$-division tower $(\tplusminus_{n-1})_n$. The pairing \eqref{EqPairingTateMod} sends $(\tplusminus_{n-1})_n\times\tplus^{-1}$ to $1\!\cdot\! dz$. Indeed, if we write $\tplus^{-1}=\sum_{k=0}^\infty\tplusminus'_k z^k$, then $\tplusminus'_0\tplusminus_0=1$ and $\sum_{k=0}^n\tplusminus'_k\tplusminus_{n-k}=0$ for $n\ge1$. The element $(\tplusminus_{n-1})_n\in T_\epsilon\,\ulCC$ corresponds to the element $f\in\Hom_{A_\epsilon}\bigl(Q_\epsilon/A_\epsilon,\,\ulCC(K^\sep)\bigr)$ with $f(z^{-n-1})=\tplusminus_n$. We compute $\tplus^{-1}\circ f(z^{-n-1})=\sum_{k=0}^\infty\tplusminus'_k \phi_z^k(\tplusminus_n)=\sum_{k=0}^n\tplusminus'_k\tplusminus_{n-k}=\delta_{n,0}=\Res_\epsilon(z^{-n-1}dz)$ for all $n$. This proves the claim.
\end{example}

\begin{definition}\label{DefdR}
Let $\ulHM$ be a local shtuka over a valuation ring $R$ as in Notation~\ref{Notation}. We denote by $K\dbl z-\zeta\dbr$ the power series ring over $K$ in the ``variable'' $z-\zeta$ and by $K\dpl z-\zeta\dpr$ its fraction field. We consider the ring homomorphism $R\dbl z\dbr\into K\dbl z-\zeta\dbr,\,z\mapsto z=\zeta+(z-\zeta)$ and define the \emph{de Rham realization of $\ulHM$} as 
\begin{eqnarray*}\label{EqdR}
\Koh^1_\dR\bigl(\ulHM,K\dbl z-\zeta\dbr\bigr) & := & \hat\sigma^*\hat M\otimes_{R\dbl z\dbr}K\dbl z-\zeta\dbr\,,\\[2mm]
\Koh^1_\dR\bigl(\ulHM,K\dpl z-\zeta\dpr\bigr) & := & \hat\sigma^*\hat M\otimes_{R\dbl z\dbr}K\dpl z-\zeta\dpr\qquad\text{and}\\[2mm]
\Koh^1_\dR(\ulHM,K) & := & \hat\sigma^*\hat M\otimes_{R\dbl z\dbr,\,z\mapsto \zeta}K \\[1mm]
& = & \Koh^1_\dR\bigl(\ulHM,K\dbl z-\zeta\dbr\bigr)\otimes_{K\dbl z-\zeta\dbr}K\dbl z-\zeta\dbr/(z-\zeta)\,.\nonumber
\end{eqnarray*}
The de Rham realization $\Koh^1_\dR\bigl(\ulHM,K\dpl z-\zeta\dpr\bigr)$ contains a full $K\dbl z-\zeta\dbr$-lattice 
\[
\Fq\;:=\;\tau_{\hat M}^{-1}(\hat M\otimes_{R\dbl z\dbr}K\dbl z-\zeta\dbr), 
\]
which is called the \emph{Hodge-Pink lattice of $\ulHM$}. The de Rham realization $\Koh^1_\dR(\ulHM,K)$ carries a descending separated and exhausting filtration $F^\bullet$ by $K$-subspaces called the \emph{Hodge-Pink filtration of $\ulHM$}. It is defined via $\Fp:=\Koh^1_\dR(\ulHM,K\dbl z-\zeta\dbr)$ and (for $i\in\BZ$)
\[
F^i\Koh^1_\dR(\ulHM,K)\;:=\;\bigl(\Fp\cap(z-\zeta)^i\Fq\bigr)\big/\bigl((z-\zeta)\Fp\cap(z-\zeta)^i\Fq\bigr)\;\subset\;\Koh^1_\dR(\ulHM,K)\,.
\]
If we equip $\Koh^1_\dR\bigl(\ulHM,K\dpl z-\zeta\dpr\bigr)$ with the descending filtration $F^i\Koh^1_\dR\bigl(\ulHM,K\dpl z-\zeta\dpr\bigr):=(z-\zeta)^i\Fq$ by $K\dbl z-\zeta\dbr$-submodules, then $F^i\Koh^1_\dR(\ulHM,K)$ is the image of $\Koh^1_\dR\bigl(\ulHM,K\dbl z-\zeta\dbr\bigr)\,\cap\,F^i\Koh^1_\dR\bigl(\ulHM,K\dpl z-\zeta\dpr\bigr)$ in $\Koh^1_\dR(\ulHM,K)$. Since $z=\zeta+(z-\zeta)$ is invertible in $K\dbl z-\zeta\dbr$ the de Rham realization with Hodge-Pink lattice and filtration is a functor on the category of local shtukas over $R$ with quasi-morphisms.
\end{definition}

\begin{definition}\label{DefdRAMotive}
If $\ulM=(M,\tau_M)$ is an $A$-motive over $R$ as in Example~\ref{ExAMotive} we use $A_K:=A\otimes_{\BF_q}K$ and $A_K/\CJ=K$, as well as the identification $\invlim A_K/\CJ^n=K\dbl z-\zeta\dbr$ from \HJLemmaZMinusZeta. Then the \emph{de Rham realization of $\ulM$} is defined as
\begin{eqnarray*}\label{EqdRAMotive}
\Koh^1_\dR\bigl(\ulM,K\dbl z-\zeta\dbr\bigr) & := & \sigma^*M\otimes_{A_R}\invlim A_K/\CJ^n\,,\\[2mm]
\Koh^1_\dR\bigl(\ulM,K\dpl z-\zeta\dpr\bigr) & := & \Koh^1_\dR\bigl(\ulM,K\dbl z-\zeta\dbr\bigr)\otimes_{K\dbl z-\zeta\dbr}K\dpl z-\zeta\dpr\qquad\text{and}\\[2mm]
\Koh^1_\dR(\ulM,K) & := & \sigma^*M\otimes_{A_R}A_K/\CJ.
\end{eqnarray*}
(See \HJCohAMotCohAMod{} for a justification of this definition and the relation with the de Rham cohomology of a Drinfeld module, resp.\ abelian $t$-module, studied by Deligne, Anderson, Gekeler and Jing Yu~\cite{Gekeler89,Yu90}, resp.\ Brownawell and Papanikolas~\cite{BrownawellPapanikolas02}.) The \emph{Hodge-Pink lattice of $\ulM$} is defined as $\Fq:=\tau_M^{-1}(M\otimes_{A_R}\invlim A_K/\CJ^n)\subset\Koh^1_\dR\bigl(\ulM,K\dpl z-\zeta\dpr\bigr)$, and the \emph{Hodge-Pink filtration of $\ulM$} is defined via $\Fp:=\Koh^1_\dR(\ulM,K\dbl z-\zeta\dbr)$ and
\begin{eqnarray*}
F^i\Koh^1_\dR(\ulM,K) \es := & \bigl(\Fp\cap(z-\zeta)^i\Fq\bigr)\big/\bigl((z-\zeta)\Fp\cap(z-\zeta)^i\Fq\bigr)\\[2mm]
= & \text{image of }\bigl(\sigma^*M\cap\tau_M^{-1}(\CJ^i M)\bigr)\otimes_RK & \subset \es \Koh^1_\dR(\ulM,K)\,;
\end{eqnarray*}
see \cite[\S\,2.6]{Goss94}. Note that the de Rham realization with Hodge-Pink lattice and filtration only depends on the generic fiber $\ulM\otimes_RK$ of $\ulM$. It is a functor on the category of $A$-motives over $K$ with quasi-morphisms.
\end{definition}

\begin{remark}\label{RemdRAMotive}
If $\ulHM:=\ulHM_\epsilon(\ulM)$ is the associated local shtuka of $\ulM$ and $f=[\BF_\epsilon:\BF_q]$ as in Example~\ref{ExAMotive}, the map 
\[
\sigma^*\tau_M^{f-1}=\sigma^*\tau_M\circ\sigma^{2*}\tau_M\circ\cdots\circ\sigma^{(f-1)*}\tau_M\colon(\sigma^{f*}M)\otimes_{A_R}A_{\epsilon,R}/\Fa_0\isoto(\sigma^*M)\otimes_{A_R}A_{\epsilon,R}/\Fa_0
\]
is an isomorphism, because $\tau_M$ is an isomorphism over $A_{\epsilon,R}/\Fa_i$ for all $i\ne0$. Therefore it defines canonical functorial isomorphisms $\sigma^*\tau_M^{f-1}\colon\Koh^1_\dR\bigl(\ulHM,K\dbl z-\zeta\dbr\bigr)\isoto\Koh^1_\dR\bigl(\ulM,K\dbl z-\zeta\dbr\bigr)$ and $\sigma^*\tau_M^{f-1}\colon\Koh^1_\dR(\ulHM,K)\isoto\Koh^1_\dR(\ulM,K)$, which are compatible with Hodge-Pink lattice and filtration because $\sigma^*\tau_M^{f-1}\circ\tau_{\hat M}^{-1}=\tau_M^{-1}$.
\end{remark}

For the next theorem note that there is a ring homomorphism $Q_\epsilon=\BF_\epsilon\dpl z\dpr\into K\dbl z-\zeta\dbr$ sending $z$ to $z=\zeta+(z-\zeta)$, because $\zeta+(z-\zeta)$ is invertible in $K\dbl z-\zeta\dbr$.

\begin{theorem}\label{ThmdR}
Let $\olK$ be the completion of an algebraic closure $K^\alg$ of $K$. There is a canonical functorial comparison isomorphism 
\[
h_{\epsilon,\dR}\colon\;\Koh^1_\epsilon(\ulHM,Q_\epsilon)\otimes_{Q_\epsilon}\olK\dpl z-\zeta\dpr\;\isoto\;\Koh^1_\dR\bigl(\ulHM,K\dpl z-\zeta\dpr\bigr)\otimes_{K\dpl z-\zeta\dpr}\olK\dpl z-\zeta\dpr\,,
\]
which satisfies $h_{\epsilon,\dR}\bigl(\Koh^1_\epsilon(\ulHM,Q_\epsilon)\otimes_{Q_\epsilon}\olK\dbl z-\zeta\dbr\bigr)=\Fq\otimes_{K\dbl z-\zeta\dbr}\olK\dbl z-\zeta\dbr$ and which is equivariant for the action of $\Gal(K^\sep/K)$, where on the source of $h_{\epsilon,\dR}$ this group acts on both factors of the tensor product and on the target of $h_{\epsilon,\dR}$ it acts only on $\olK$. However, if $K$ is not perfect, $h_{\epsilon,\dR}$ does not allow to recover $\Koh^1_\dR\bigl(\ulHM,K\dpl z-\zeta\dpr\bigr)$ or $\Koh^1_\dR(\ulHM,K)$ from $\Koh^1_\epsilon(\ulHM,Q_\epsilon)$ because the field of Galois invariants $\olK{}^{\Gal(K^\sep/K)}$ equals the completion $\wh{K^\perf}$ of the perfect closure of $K$ by the Ax-Sen-Tate Theorem~\cite[p.\ 417]{Ax70} and $\olK\dpl z-\zeta\dpr^{\Gal(K^\sep/K)}$ $=\wh{K^\perf}\dpl z-\zeta\dpr$.
\end{theorem}

\noindent
{\itshape Remark.} Regardless of the field $K$ the comparison isomorphism does not allow to recover $\Koh^1_\epsilon(\ulHM,Q_\epsilon)$ from $\Koh^1_\dR\bigl(\ulHM,K\dbl z-\zeta\dbr\bigr)$ and $\Fq$.

\begin{proof}[Proof of Theorem~\ref{ThmdR}]
Note that the map $z\mapsto\zeta+(z-\zeta)$ induces ring homomorphisms $R\dbl z\dbr[\tfrac{1}{z-\zeta}]\into\olK\dpl z-\zeta\dpr$ and $K^\sep\langle\tfrac{z}{\zeta}\rangle\into \olK\dbl z-\zeta\dbr$, $\sum\limits_{n=0}^\infty b_nz^n\mapsto\sum\limits_{n=0}^\infty b_n\bigl(\sum\limits_{i=0}^n{n\choose i}\zeta^{n-i}(z-\zeta)^i\bigr)=\sum\limits_{i=0}^\infty\zeta^{-i}\bigl(\sum\limits_{n=i}^\infty{n\choose i}b_n\zeta^n\bigr)(z-\zeta)^i$. The series $\sum\limits_{n=i}^\infty{n\choose i}b_n\zeta^n$ converges in $\olK$ because $|{n\choose i}b_n\zeta^n|\le|b_n\zeta^n|\to0$ for $n\to\infty$. Thus we can take the functorial isomorphism $h$ from \eqref{EqTateModIsomConv} in Remark~\ref{RemTateMod} and define
\[
h_{\epsilon,\dR}:=(\tau_{\hat M}^{-1}\circ h)\otimes\id_{\olK\dpl z-\zeta\dpr}\colon\;\check T_\epsilon\ulHM\otimes_{A_\epsilon}\olK\dpl z-\zeta\dpr\;\isoto\;\hat\sigma^*\hat M\otimes_{R\dbl z\dbr}\olK\dpl z-\zeta\dpr\,.
\]
Clearly $\Fq\otimes_{K\dbl z-\zeta\dbr}\olK\dbl z-\zeta\dbr=\tau_{\hat M}^{-1}\bigl(\hat M\otimes_{R\dbl z\dbr}\olK\dbl z-\zeta\dbr\bigr)=\tau_{\hat M}^{-1}\circ h\bigl(\check T_\epsilon\ulHM\otimes_{A_\epsilon}\olK\dbl z-\zeta\dbr\bigr)$.
\end{proof}

\begin{remark}\label{RemdRIsomAMotive}
If $\ulM=(M,\tau_M)$ is an $A$-motive over $K$ as in Example~\ref{ExAMotive} which not necessarily has good reduction, the functorial isomorphism $h$ from \eqref{EqTateModIsomConvAMotive} in Remark~\ref{RemTateModAMotive} defines a canonical functorial comparison isomorphism 
\[
h_{\epsilon,\dR}:=(\tau_M^{-1}\circ h)\otimes\id_{\olK\dpl z-\zeta\dpr}\colon\;\check T_\epsilon\ulM\otimes_{A_\epsilon}\olK\dpl z-\zeta\dpr\;\isoto\;\sigma^*M\otimes_{A_K}\olK\dpl z-\zeta\dpr\,,
\]
that is, a canonical functorial comparison isomorphism 
\[
h_{\epsilon,\dR}\colon\;\Koh^1_\epsilon(\ulM,Q_\epsilon)\otimes_{Q_\epsilon}\olK\dpl z-\zeta\dpr\;\isoto\;\Koh^1_\dR\bigl(\ulM,K\dpl z-\zeta\dpr\bigr)\otimes_{K\dpl z-\zeta\dpr}\olK\dpl z-\zeta\dpr\,,
\]
which satisfies $h_{\epsilon,\dR}\bigl(\Koh^1_\epsilon(\ulM,Q_\epsilon)\otimes_{Q_\epsilon}\olK\dbl z-\zeta\dbr\bigr)=\Fq\otimes_{K\dbl z-\zeta\dbr}\olK\dbl z-\zeta\dbr$ and which is equivariant for the action of $\Gal(K^\sep/K)$, where on the source of $h_{\epsilon,\dR}$ this group acts on both factors of the tensor product and on the target of $h_{\epsilon,\dR}$ it acts only on $\olK$. When $\ulM$ has good reduction, this comparison isomorphism is compatible with the comparison isomorphism from Theorem~\ref{ThmdR} and the isomorphisms $\sigma^*\tau_M^{f-1}\colon\Koh^1_\dR\bigl(\ulHM(\ulM),K\dbl z-\zeta\dbr\bigr)\isoto\Koh^1_\dR\bigl(\ulM,K\dbl z-\zeta\dbr\bigr)$ from Remark~\ref{RemdRAMotive} and $\Koh^1_\epsilon(\ulM,Q_\epsilon)\isoto\Koh^1_\epsilon(\ulHM(\ulM),Q_\epsilon)$ from Proposition~\ref{PropTateModAMotive}.
\end{remark}

\begin{remark}\label{RemdR}
The comparison isomorphisms $h_{\epsilon,\dR}$ from Theorem~\ref{ThmdR} and Remark~\ref{RemdRIsomAMotive} are the function field analog for the comparison isomorphism $\Koh_\et^i(X\times_LL^\alg,\BQ_p)\otimes_{\BQ_p}\bB_\dR\isoto\Koh_\dR^i(X/L)\otimes_L\bB_\dR$ for a smooth proper scheme $X$ over a complete discretely valued field $L$ of characteristic $0$ with perfect residue field of characteristic $p$. The existence of the latter comparison isomorphism was conjectured by Fontaine~\cite[A.6]{Fontaine82} and proved by Faltings~\cite{Faltings89}. It is equivariant for the action of $\Gal(L^\sep/L)$ and allows to compute $\Koh_\dR^i(X/L)=\bigl(\Koh_\et^i(X\times_LL^\alg,\BQ_p)\otimes_{\BQ_p}\bB_\dR\bigr)^{\Gal(L^\sep/L)}$, because $\bB_\dR^{\Gal(L^\sep/L)}=L$. Note that it does not allow to reconstruct $\Koh_\et^i(X\times_LL^\alg,\BQ_p)$ from $\Koh_\dR^i(X/L)$. The existence of this comparison isomorphism is also phrased by saying that the $p$-adic Galois representation $\Koh_\et^i(X\times_LL^\alg,\BQ_p)$ is \emph{de Rham}. 

In our comparison isomorphism, $\olK\dpl z-\zeta\dpr$ is the analog of $\bB_\dR$; see \cite[\S\,2.9]{HartlDict}, and we may thus say that for every $A$-motive $\ulM$ over $K$, which not necessarily has good reduction, the $Q_\epsilon$-representation $\Koh^1_\epsilon(\ulM,Q_\epsilon)$ of $\Gal(K^\sep/K)$ is \emph{de Rham}.
\end{remark}

\begin{remark}
The entries of a matrix representing the comparison isomorphism with respect to some bases are called the \emph{periods of $\ulHM$}, respectively of $X$. The transcendence degree of the periods of $\ulHM$ was related by Mishiba~\cite{Mishiba12} to the dimension of the Tannakian Galois group of $\ulHM$ following the approach of Papanikolas~\cite{Papanikolas}.
\end{remark}

\begin{example}\label{ExCarlitz4B}
For the Carlitz motive from Example~\ref{ExCarlitz4A} we have $\Koh^1_\epsilon(\ulM,Q_\epsilon)=Q_\epsilon\cdot\tplus^{-1}\cong Q_\epsilon$ and \mbox{$\Koh^1_\dR(\ulM,K\dbl z-\zeta\dbr)=K\dbl z-\zeta\dbr=:\Fp$}. The Hodge-Pink lattice is $\Fq=(z-\zeta)^{-1}\Fp$ and the Hodge filtration satisfies $F^1=\Koh^1_\dR(\ulM,K)\supset F^2=(0)$. With respect to the bases $\tplus^{-1}$ of $\Koh^1_\epsilon(\ulM,Q_\epsilon)$ and $1$ of \mbox{$\Koh^1_\dR(\ulM,K\dbl z-\zeta\dbr)$} the comparison isomorphism $h_{\epsilon,\dR}$ from Theorem~\ref{ThmdR} is given by the \emph{$\epsilon$-adic Carlitz period} \mbox{$(z-\zeta)^{-1}\tplus^{-1}=\hat\sigma(\tplus)^{-1}$}. It has a pole of order one at $z=\zeta$ because $\tplus\in K^\sep\langle\tfrac{z}{\zeta}\rangle\mal\subset \olK\dbl z-\zeta\dbr\mal$. So $h_{\epsilon,\dR}\bigl(\Koh^1_\epsilon(\ulM,Q_\epsilon)\otimes_{Q_\epsilon}\olK\dbl z-\zeta\dbr\bigr)=(z-\zeta)^{-1}\olK\dbl z-\zeta\dbr=\Fq\otimes_{K\dbl z-\zeta\dbr}\olK\dbl z-\zeta\dbr$. Mishiba~\cite[Example~6.6]{Mishiba12} shows that $\hat\sigma(\tplus)^{-1}$ is transcendental over $\BF_q(\zeta)\dpl z\dpr$.
\end{example}

We next study properties of the (dual) Tate module functors.

\begin{theorem}\label{ThmAndersonKim}
Assume that $R$ is discretely valued. Then the functor $\check T_\epsilon\colon\ulHM\mapsto \check T_\epsilon\ulHM$ from the category of local shtukas over $R$ to the category $\Rep_{A_\epsilon}\Gal(K^\sep/K)$ of representations of $\Gal(K^\sep/K)$ on finite free $A_\epsilon$-modules and the functor $\check V_\epsilon\colon\ulHM\mapsto \check V_\epsilon\ulHM$ from the category of local shtukas over $R$ with quasi-morphisms to the category $\Rep_{Q_\epsilon}\Gal(K^\sep/K)$ of representations of $\Gal(K^\sep/K)$ on finite dimensional $Q_\epsilon$-vector spaces are fully faithful.
\end{theorem}

We will give a proof at the end of the section.\\

If we additionally require that $\tau_{\hat M}$ is topologically nilpotent in the sense of Remark~\ref{RemTopolNilpotent}, then the above theorem was previously obtained by Anderson \cite[\S4.5, Theorem~1]{Anderson93} by a different method from ours. Note that if $\tau_{\hat M}$ is ``topologically nilpotent'', then the action of the (not necessarily commutative) polynomial ring $R\dbl z\dbr\{\hat\tau\}$ on $\ulHM$ extends to the action of the formal power series ring \mbox{$R\dbl z\dbr \{\!\{\hat\tau\}\!\}$}, which is also a (not necessarily commutative) local ring.\\

The theorem allows to make the following

\begin{definition}\label{DefCrystRep}
Let $R$ be discretely valued. The full subcategory of $\Rep_{Q_\epsilon}\Gal(K^\sep/K)$ which is the essential image of the functor $\check V_\epsilon$ is called the \emph{category of equal characteristic crystalline representations}.
\end{definition}

We explain the motivation for this definition in Remarks~\ref{RemCrystRealiz} and \ref{RemEqCharCrystRep} below.

\begin{proposition}\label{PropLattices}
Let $R$ be discretely valued, let $\ulHM$ be a local shtuka over $R$, and set $\check V:=\check V_\epsilon\ulHM$. Then the map
\[
\left\{\,\begin{array}{l}
\ulHM'\subset\ulHM[\tfrac{1}{z}]\colon\text{ local shtukas}\\\text{over }R\text{ of full rank}
\end{array}
\,\right\}  
\xrightarrow{\es\check T_\epsilon\;}
\left\{\,\begin{array}{l}\check T'\subset\check V\colon\Gal(K^\sep/K)\text{-stable}\\\text{full }A_\epsilon\text{-lattices}
\end{array}
\,\right\}
\]
is a bijection.
\end{proposition}

Let us now prove Theorem~\ref{ThmAndersonKim} and Proposition~\ref{PropLattices}. We begin with a few lemmas.

\begin{lemma}\label{LemmaSaturation}
Assume that $R$ is discretely valued.
Let $\hat M$ be  a finitely generated torsion-free $R\dbl z\dbr $-module (not necessarily free). We set $\hat M':= \hat M[\frac{1}{z}]\cap (\hat M \otimes_{R\dbl z\dbr }K\dbl z\dbr ) $, where the intersection is taken inside $\hat M\otimes_{R\dbl z\dbr }K\dpl z\dpr$. Then $\hat M'$ is free over $R\dbl z\dbr $ and we have $\hat M[\frac{1}{z-\zeta}] = \hat M'[\frac{1}{z-\zeta}]$. In particular, if $\hat M$ is equipped with an isomorphism $\tau_{\hat M}:\hat\sigma^*\hat M[\frac{1}{z-\zeta}] \isoto \hat M[\frac{1}{z-\zeta}]$, then $\ulHM':=(\hat M', \tau_{\hat M'} ) $ is a local shtuka, where $\tau_{\hat M'}=\tau_{\hat M} $. Furthermore, we have $\hat M' = \hat M$ if $\hat M$ is already free.
\end{lemma}
\begin{proof}
Note that $R\dbl z\dbr[\frac{1}{z}]$ is a principal ideal domain, being a $1$-dimensional factorial ring, so the torsion-free module $\hat M'[\frac{1}{z}]$ is free over $R\dbl z\dbr[\frac{1}{z}]$. Likewise,  $\hat M\otimes_{R\dbl z\dbr}K\dbl z\dbr$ is free over $K\dbl z\dbr$.

Clearly $\hat M\subset\hat M'$ and $\hat M[\frac{1}{z}]=\hat M'[\frac{1}{z}]$. Choose isomorphisms $\alpha\colon\hat M[\frac{1}{z}]\isoto R\dbl z\dbr[\frac{1}{z}]^{\oplus r}$ and $\beta\colon\hat M\otimes_{R\dbl z\dbr}K\dbl z\dbr\isoto K\dbl z\dbr^{\oplus r}$. After multiplying $\alpha$ with a high enough power of $z$ we may assume that all entries of the matrix $A:=\alpha\beta^{-1}\in\GL_r\bigl(K\dpl z\dpr\bigr)$ lie in $K\dbl z\dbr$. Then every $m'\in\hat M'$ satisfies $\alpha(m')=A\cdot\beta(m')\in(R\dbl z\dbr[\frac{1}{z}]\cap K\dbl z\dbr)^{\oplus r}=R\dbl z\dbr^{\oplus r}$. Let $\hat M'':=\alpha^{-1}(R\dbl z\dbr^{\oplus r})$ so that $\hat M'\subset\hat M''$. Since $R\dbl z\dbr$ is noetherian, $\hat M'$ is finitely generated over $R\dbl z\dbr$. Together with $\hat M[\frac{1}{z}]=\hat M'[\frac{1}{z}]$ this implies that there is a power of $z$ which annihilates $\hat M'/\hat M$. From $\hat M'\otimes_{R\dbl z\dbr}K\dbl z\dbr\subset\hat M''\otimes_{R\dbl z\dbr}K\dbl z\dbr\cong K\dbl z\dbr^{\oplus r}$ it 
follows that $\hat M'\otimes_{R\dbl z\dbr}K\dbl z\dbr$ has no $z$-torsion, whence $\hat M'\otimes_{R\dbl z\dbr}K\dbl z\dbr\subset\hat M'\otimes_{R\dbl z\dbr}K\dpl z\dpr=\hat M\otimes_{R\dbl z\dbr}K\dpl z\dpr$. This implies $\hat M'\otimes_{R\dbl z\dbr}K\dbl z\dbr=\hat M\otimes_{R\dbl z\dbr}K\dbl z\dbr$ and $(\hat M'/\hat M)\otimes_R K=(\hat M'/\hat M)\otimes_{R\dbl z\dbr}K\dbl z\dbr=(0)$. So $\hat M'/\hat M$ is annihilated by a power of $\zeta$ and thus by some power of $z-\zeta $, which shows that $\hat M[\frac{1}{z-\zeta}] = \hat M'[\frac{1}{z-\zeta}]$.

So it remains to show that $\hat M'$ is free over $R\dbl z\dbr $. Note that by construction $\hat M'$ is a \emph{reflexive} $R\dbl z\dbr$-module; i.e., $\hat M'$ is naturally isomorphic to the $R$-linear double dual $(\hat M')\dual{}\dual$; indeed, this follows from the fact that $\hat M'[\tfrac{1}{z}]$ and $\hat M'\otimes_{R\dbl z\dbr}K\dbl z\dbr$ are free over  $R\dbl z\dbr[\tfrac{1}{z}]$ and $K\dbl z\dbr$, respectively, and the equality $R\dbl z\dbr=R\dbl z\dbr[\tfrac{1}{z}]\cap K\dbl z\dbr$ with the intersection taking place in $K\dpl z\dpr$. Now, it is well known that a reflexive module over a regular $2$-dimensional local ring is necessarily free; \emph{cf.} \cite[\S6, Lemme~6]{Serre58}. The last assertion follows from the equation $R\dbl z\dbr=R\dbl z\dbr[\tfrac{1}{z}]\cap K\dbl z\dbr$ in $K\dpl z\dpr$.
\end{proof}

\begin{lemma}\label{LemmaFaithfulBM}
Assume that  $\bigcap_n\hat\sigma^n(\Fm_R) = (0)$. Then the base change functor $\ulHM\mapsto \ulHM\otimes_R k$ from the category of local shtukas over $R$ to the categories of local shtukas over $k$ is faithful.
\end{lemma}
\noindent
{\it Remark.} The assumption  $\bigcap_n\hat\sigma^n(\Fm_R) = (0)$ is satisfied if $R$ is discretely valued.
\begin{proof}[Proof of Lemma~\ref{LemmaFaithfulBM}]
For local shtukas $\ulHM$ and $\ulHM'$ over $R$, let $f\colon\ulHM\rightarrow\ulHM'$ be a morphism which becomes zero over $k$; i.e., $f(\ulHM)=(0)$ in $\ulHM'\otimes_R R/\Fm_R$. Since we have $\tau_{\hat M'}\circ \hat\sigma^*f = f\circ \tau_{\hat M}$, it follows inductively that $f\bigl(\ulHM[\tfrac{1}{z}]\bigr)=(0)$ in $\ulHM'[\tfrac{1}{z-\zeta}]\otimes_R R/\hat\sigma^n(\Fm_R)=\ulHM'[\tfrac{1}{z}]\otimes_R R/\hat\sigma^n(\Fm_R)$ for any $n\geqslant1$. Now the assumption $\bigcap_n\hat\sigma^n(\Fm_R) = (0)$ forces $f=0$.
\end{proof}

\begin{lemma}\label{LemmaTateIsog}
Assume that $R$ is discretely valued, and let $\ulHM$ and $\ulHM'$ be local shtukas over $R$. Let $f\colon\ulHM\rightarrow\ulHM'$ be a morphism such that $\check T_\epsilon f \colon \check T_\epsilon\ulHM \rightarrow \check T_\epsilon\ulHM'$ is an isomorphism. Then $f$ is an isomorphism.
\end{lemma}

\begin{proof}
To show that $f$ is an isomorphism, it suffices to show that the determinant of $f$ is an isomorphism. Therefore we may assume that $\ulHM$ and $\ulHM'$ are of rank~$1$ by replacing $\ulHM$ and $\ulHM'$ with their respective top exterior powers.

Let us first show that $f$ is an isogeny; i.e., $f[\frac{1}{z}]$ is an isomorphism. By Nakayama's lemma, it suffices to show the surjectivity of the following map induced by $f[\frac{1}{z}]$:
\[
\hat M \otimes_{R\dbl z\dbr } k\dpl z \dpr \rightarrow \hat M' \otimes_{R\dbl z\dbr } k\dpl z \dpr .
\]
Since both the source and the target  are $1$-dimensional vector spaces over $k\dpl z\dpr$, the above map is surjective as long as it is non-zero. The latter follows from Lemma~\ref{LemmaFaithfulBM} and the assumption that $\check{T}_\epsilon f$ is an isomorphism, which implies that $f\ne0$.

Let us now show that $f$ is an isomorphism. By Proposition~\ref{PropTateMod} it follows that $f$ induces an isomorphism
\[
f\otimes1\colon \hat M\otimes_{R\dbl z\dbr }K\dbl z\dbr  \isoto \hat M'\otimes_{R\dbl z\dbr }K\dbl z\dbr .
\]
Since $f[\frac{1}{z}]$ is also an isomorphism, it follows from Lemma~\ref{LemmaSaturation} that $f$ is an isomorphism.
\end{proof}

\begin{proof}[Proof of Theorem~\ref{ThmAndersonKim}]
Assume that $R$ is discretely valued, and let $\ulHM$ and $\ulHM'$ be local shtukas over $R$. To prove Theorem~\ref{ThmAndersonKim}, 
it suffices to show that for any $\Gal(K^\sep/K)$-equivariant morphism $g\colon\check T_\epsilon\ulHM \rightarrow \check T_\epsilon\ulHM'$ there exists a unique morphism $f\colon\ulHM\rightarrow\ulHM'$ with $\check T_\epsilon f = g$. 

Since we have $\hat M\otimes_{R\dbl z\dbr }K\dbl z\dbr  = (\check T_\epsilon\ulHM\otimes_{A_\epsilon}K^\sep\dbl z\dbr )^{\Gal(K^\sep/K)}$ which matches $\tau_{\hat M}\otimes1$ and $1\otimes\hat\sigma $ by Proposition~\ref{PropTateMod}, it follows that $g$ induces a uniquely determined morphism
\[
f_K\colon \hat M\otimes_{R\dbl z\dbr }K\dbl z\dbr  \rightarrow  \hat M'\otimes_{R\dbl z\dbr }K\dbl z\dbr 
\]
satisfying $(\tau_{\hat M'}\otimes1)\circ \hat\sigma^*f_K = f_K\circ (\tau_{\hat M}\otimes1)$.
Furthermore,  any morphism $f\colon\ulHM\rightarrow\ulHM'$ with $\check T_\epsilon f = g$ has to satisfy $f = f_K|_{\hat M}$. Therefore, it suffices to show that $f_K(\hat M)\subset \hat M'$ for any $g$. 

We define  $\hat M''_1, \hat M''_2 \subset \hat M'\otimes_{R\dbl z\dbr }K\dbl z\dbr $ as follows:
\begin{align*}
\hat M''_1 & := f_K(\hat M)[\tfrac{1}{z}]\cap \left( f_K(\hat M)\otimes_{R\dbl z\dbr }K\dbl z\dbr   \right) \\
\hat M''_2 & := ( f_K(\hat M)\cap \hat M')[\tfrac{1}{z}]\cap \left( ( f_K(\hat M)\cap \hat M')\otimes_{R\dbl z\dbr }K\dbl z\dbr \right)   . 
\end{align*}
By Lemma~\ref{LemmaSaturation} applied to the torsion free modules $ f_K(\hat M)$ and $( f_K(\hat M)\cap \hat M')$, we obtain local shtukas $\ulHM_1''$ and $\ulHM''_2$ with underlying $R\dbl z\dbr $-modules $\hat M''_1$ and $\hat M_2''$, respectively, and the natural maps $\ulHM \rightarrow \ulHM''_1$, $\ulHM''_2\hookrightarrow \ulHM''_1$, and $\ulHM_2''\hookrightarrow \ulHM'$ are morphisms of local shtukas. Clearly, we have $\hat M''_1\otimes_{R\dbl z\dbr }K\dbl z\dbr  =  \hat M''_2\otimes_{R\dbl z\dbr }K\dbl z\dbr $ since both are equal to the image of $f_K$, so the natural inclusion $\ulHM''_2\hookrightarrow \ulHM_1''$ induces an isomorphism $\check T_\epsilon\ulHM_2''\isoto\check T_\epsilon\ulHM''_1$. We can now apply Lemma~\ref{LemmaTateIsog} to show that $\ulHM''_1 = \ulHM_2''$. Therefore, we obtain a map $f\colon\ulHM \rightarrow\ulHM'$ by the following 
composition
\[
\ulHM \rightarrow \ulHM''_1 = \ulHM''_2 \hookrightarrow \ulHM',
\]
which clearly extends to $f_K$.
\end{proof}

\begin{proof}[Proof of Proposition~\ref{PropLattices}]
Let $\ulHM$ be a local shtuka over $R$. We want to show that $\check T_\epsilon$ induces a bijection from the set of local shtukas $\ulHM'\subset \ulHM[\frac{1}{z}]$ to the set of Galois-stable $A_\epsilon$-lattices $\check T'\subset \check V_\epsilon\ulHM$. The injectivity of $\check T_\epsilon$ is clear from Theorem~\ref{ThmAndersonKim}, so it suffices to show the surjectivity.

Let $\check T'\subset \check V_\epsilon\ulHM$ be a Galois-stable $A_\epsilon$-lattice. We want to show that there exists a local shtuka $\ulHM'\subset\ulHM[\frac{1}{z}]$ with $\check T_\epsilon\ulHM'=\check T'$. We set $N:= (\check T' \otimes_{A_\epsilon}K^\sep\dbl z\dbr )^{\Gal(K^\sep/K)}$, which can be viewed as a $K\dbl z\dbr$-lattice in $\hat M[\frac{1}{z}]\otimes_{R\dbl z\dbr } K\dbl z\dbr $ by Proposition~\ref{PropTateMod}, and set $\hat M':=  \hat M[\tfrac{1}{z}] \cap N$. By construction, $\hat M'$ is finitely generated over $R\dbl z\dbr $, and we have $\hat M' = \hat M'[\frac{1}{z}]\cap (\hat M'\otimes_{R\dbl z\dbr}K\dbl z\dbr)$. So by Lemma~\ref{LemmaSaturation}, $\ulHM':=(\hat M',\tau_{\hat M'})$ is a local shtuka where $\tau_{\hat M'}$ is the restriction of $\tau_{\hat M}$.

It remains to show that $\check T_\epsilon\ulHM' = \check T'$. By Proposition~\ref{PropTateMod}, it suffices to show that $\hat M'\otimes_{R\dbl z\dbr}K\dbl z\dbr = N$, which follows from the left exactness of the following sequence:
\begin{equation}\label{EqLatticesProof}
\xymatrix @C+1pc {
0 \ar[r] & \hat M'\otimes_{R\dbl z\dbr}K\dbl z\dbr \ar[r] &  N[\tfrac{1}{z}]  \ar[r] &  N[\tfrac{1}{z}]/N.} 
\end{equation}
Indeed, this sequence is the flat scalar extension of the left exact sequence $0\rightarrow \hat M'\rightarrow \hat M[\frac{1}{z}] \rightarrow N[\frac{1}{z}]/N$; note that we have $N[\frac{1}{z}] = \hat M[\frac{1}{z}]\otimes_{R\dbl z\dbr } K\dbl z\dbr $ and 
\[
(N[\tfrac{1}{z}]/N) \otimes_{R\dbl z\dbr } K\dbl z\dbr \cong \varinjlim_r \left(z^{-r}N/N\otimes_{R\dbl z\dbr/(z^r) } K\dbl z\dbr/(z^r)\right) \cong \varinjlim_r (z^{-r}N/N)[\tfrac{1}{\zeta}] \cong N[\tfrac{1}{z}]/N.
\]
This concludes the proof.
\end{proof}

%
%

\section{Hodge-Pink structures}\label{SectHPStruct}
\setcounter{equation}{0}

From now on we work over the base scheme $S=\Spec R$ where $R$ is a valuation ring as in Notation~\ref{Notation} with fraction field $K$ and residue field $k$. In particular $\zeta=0$ in $k$. We assume that there is a section $k\into R$ and we fix one. As mentioned in the introduction we want to describe the analog of Fontaine's classification of crystalline $p$-adic Galois representations by filtered isocrystals. In this function field analog we make the following

\begin{definition}\label{DefZIsoc}
A \emph{$z$-isocrystal} over $k$ is a pair $(D,\tau_D)$ consisting of a finite dimensional $k\dpl z\dpr$-vector space together with a $k\dpl z\dpr$-isomorphism $\tau_D\colon\hat\sigma^*D\isoto D$. A \emph{morphism} $(D,\tau_D)\to(D',\tau_{D'})$ is a $k\dpl z\dpr$-homomorphism $f\colon D\to D'$ satisfying $\tau_{D'}\circ\hat\sigma^*f=f\circ\tau_D$.
\end{definition}

Here again $\hat\sigma^*D:=D\otimes_{k\dpl z\dpr,\,\hat\sigma}k\dpl z\dpr$ and we set $\hat\sigma_{\!D}^*x:=x\otimes 1\in\hat\sigma^*D$ for $x\in D$. By the assumed existence of the fixed section $k\into R$ there is a ring homomorphism 
\begin{equation}\label{EqSection}
\TS k\dpl z\dpr\longinto K\dbl z-\zeta\dbr\,,\quad z\longmapsto \zeta+(z-\zeta)\,,\quad\sum\limits_i b_i z^i\longmapsto\sum\limits_{j=0}^\infty(z-\zeta)^j\cdot\sum\limits_i{i\choose j}b_i\zeta^{i-j}\,.
\end{equation}
We always consider $K\dbl z-\zeta\dbr$ and its fraction field $K\dpl z-\zeta\dpr$ as $k\dpl z\dpr$-vector spaces via \eqref{EqSection}.

\begin{definition}\label{DefHPStruct}
A \emph{$z$-isocrystal with Hodge-Pink structure over $R$} is a triple $\ulD=(D,\tau_D,\Fq_D)$ consisting of a $z$-isocrystal $(D,\tau_D)$ over $k$ and a $K\dbl z-\zeta\dbr$-lattice $\Fq_D$ in $D\otimes_{k\dpl z\dpr}K\dpl z-\zeta\dpr$ of full rank, which is called the \emph{Hodge-Pink lattice of $\ulD$}. The dimension of $D$ is called the \emph{rank of $\ulD$} and is denoted $\rk\ulD$. 

A \emph{morphism} $(D,\tau_D,\Fq_D)\to(D',\tau_{D'},\Fq_{D'})$ is a $k\dpl z\dpr$-homomorphism $f\colon D\to D'$ satisfying $\tau_{D'}\circ\hat\sigma^*f=f\circ\tau_D$ and $(f\otimes\id)(\Fq_D)\subset\Fq_{D'}$.

A \emph{strict subobject $\ulD'\subset\ulD$} is a $z$-isocrystal with Hodge-Pink structure of the form \linebreak\mbox{$\ulD'=\bigl(D',\,\tau_D|_{\hat\sigma^*\!D'},\,\Fq_D\cap D'\otimes_{k\dpl z\dpr}K\dpl z-\zeta\dpr\bigr)$} where $D'\subset D$ is a $k\dpl z\dpr$-subspace with $\tau_D(\hat\sigma^*D')=D'$.
\end{definition}

\begin{remark}\label{RemDeviation}
This definition slightly deviates from \cite[\S\,3 and \S\,7]{GL11} and \cite[\S\,2.3]{HartlPSp} where the Hodge-Pink lattice $\Fq_D$ was defined to be a $K\dbl z-\zeta\dbr$-lattice in $\hat\sigma^*D\otimes_{k\dpl z\dpr}K\dpl z-\zeta\dpr$. The reason for our definition here is explained in Example~\ref{ExHPOfLocSht} below.
\end{remark}

\begin{definition}\label{DefHPWts}
On a $z$-isocrystal with Hodge-Pink structure $\ulD$ there always is the tautological lattice $\Fp_D:=D\otimes_{k\dpl z\dpr}K\dbl z-\zeta\dbr$. Since $K\dbl z-\zeta\dbr$ is a principal ideal domain the elementary divisor theorem provides basis vectors $v_i\in\Fp_D$ such that $\Fp_D=\bigoplus_{i=1}^rK\dbl z-\zeta\dbr\cdot v_i$ and $\Fq_D=\bigoplus_{i=1}^rK\dbl z-\zeta\dbr\cdot(z-\zeta)^{\mu_i}\cdot v_i$ for integers $\mu_1\ge\ldots\ge\mu_r$. We call $(\mu_1,\ldots,\mu_r)$ the \emph{Hodge-Pink weights of $\ulD$}. Alternatively if $e$ is large enough such that $\Fq_D\subset(z-\zeta)^{-e}\Fp_D$ or $(z-\zeta)^e\Fp_D\subset \Fq_D$ then the Hodge-Pink weights are characterized by
\begin{eqnarray*}
(z-\zeta)^{-e}\Fp_D/\Fq_D&\cong&\bigoplus_{i=1}^n K\dbl z-\zeta\dbr/(z-\zeta)^{e+\mu_i}\,,\\
\text{or}\quad\Fq_D/(z-\zeta)^e\Fp_D&\cong&\bigoplus_{i=1}^n K\dbl z-\zeta\dbr/(z-\zeta)^{e-\mu_i}
\end{eqnarray*}
The category of $z$-isocrystals with Hodge-Pink structure possesses \emph{tensor products} and \emph{duals}
\begin{eqnarray*}
\ul D\otimes\ul D' & = & \bigl(D\otimes_{k\dpl z\dpr} D', \tau_D\otimes \tau_{D'}, \Fq_D\otimes_{K\dbl z-\zeta\dbr}\Fq_{D'}\bigr)\,,\\[2mm]
\ul D\dual & = & \bigl(D\dual:=\Hom_{k\dpl z\dpr}(D,k\dpl z\dpr),(\tau_D^{-1})\dual, \Hom_{K\dbl z-\zeta\dbr}(\Fq_D,K\dbl z-\zeta\dbr)\bigr)\,,
\end{eqnarray*}
\emph{internal Hom's}, and the \emph{unit object} $\bigl(k\dpl z\dpr,\tau_D=1,\Fq_D=\Fp_D\bigr)$. The endomorphism ring of the unit object is $\BF_\epsilon\dpl z\dpr$. For every $n\in\BZ$ we consider the \emph{Tate object} 
\[
\ul\BOne(n)\es:=\es\Bigl(D=k\dpl z\dpr\,,\,\tau_D=z^{-n}\,,\,\Fq_D=(z-\zeta)^n\Fp_D\Bigr)\,. 
\]
The Hodge-Pink lattice $\Fq_D$ induces a descending filtration of $D_K:=D\otimes_{k\dpl z\dpr,\,z\mapsto\zeta}K$ by $K$-subspaces as follows. Consider the natural projection
\[
\Fp_D\;\onto\;\Fp_D/(z -\zeta)\Fp_D\;=\;D_K\,.
\]
The \emph{Hodge-Pink filtration} $F^\bullet D_K=(F^i D_K)_{i\in\BZ}$ is defined by letting $F^i D_K$ be the image in $D_K$ of $\Fp_D\cap(z-\zeta)^i\Fq_D$ for all $i\in\BZ$. This means, $F^i D_K=\bigl(\Fp_D\cap(z-\zeta)^i\Fq_D\bigr)\big/\bigl((z-\zeta)\Fp_D\cap(z-\zeta)^i\Fq_D\bigr)$.
\end{definition}

\begin{example}\label{ExHPFilt}
The Hodge-Pink lattice contains finer information than the Hodge-Pink filtration. For example let $D=k\dpl z\dpr^{\oplus r}$ and $\Fq_D=(z-\zeta)^2\Fp_D+K\dbl z-\zeta\dbr\cdot(v_0+(z-\zeta)v_1)\subset\Fp_D$ for vectors $v_0,v_1\in K^{\oplus r}$. Then $F^{-2}D_K=D_K\supset F^{-1}D_K=K\cdot v_0=F^0D_K\supset F^1D_K=(0)$. So the information about $v_1$ is not contained in the Hodge-Pink filtration.
\end{example}

\begin{remark}\label{RemCompWithHJ_A}
In comparison with the Hodge-Pink structures at $\infty$ from \HJSectHPStruct{} our Frobenius $\tau_D$ replaces the weight filtration $W_\bullet H$ from \HJDefHPStruct{} and the Hodge-Pink weights from \HJRemHPWts{} are the negatives of the Hodge-Pink weights we defined in Definition~\ref{DefHPWts}. Observe that the Frobenius $\tau_D$ induces on $D$ an increasing filtration given by Newton slopes similarly to \cite{Zink84}.
\end{remark}

At the heart of the theory is the following

\begin{example}\label{ExHPOfLocSht}
Let $\ulHM=(\hat M,\tau_{\hat M})$ be a local shtuka over $\Spec R$. We set $(\wt D,\tau_{\wt D}):=(\hat M,\tau_{\hat M})\otimes_{R\dbl z\dbr}k\dpl z\dpr$ and take $(D,\tau_D):=\hat\sigma^*(\wt D,\tau_{\wt D})$ as the underlying $z$-isocrystal. This deviates from \cite[\S\,3 and \S\,7]{GL11} and \cite[\S\,2.3]{HartlPSp} where instead $(\wt D,\tau_{\wt D})$ was taken as the underlying $z$-isocrystal. Note however, that there is a canonical isomorphism $\tau_{\wt D}\colon(D,\tau_D)\isoto(\wt D,\tau_{\wt D}$). The reason for our definition here is that we also define an integral structure $\Koh^1_\cris(\ulHM,k\dbl z\dbr)$ in Definition~\ref{DefCrystRealiz} below. In order to construct $\Fq_D$ we will need a ``comparison isomorphism''. This does not exist in general. However, if $R$ is discretely valued it is constructed as follows. Consider the  $R$-algebra
\begin{eqnarray}\label{EqBrig}
R\dbl z,z^{-1}\}&:=&\bigl\{\,\sum_{i=-\infty}^\infty b_iz^i\colon\; b_i\in R\,,\,|b_i|\,|\zeta|^{ri}\to0\;(i\to-\infty) \text{ for all }r>0\,\bigr\}\,.
\end{eqnarray}
It is a subring of $K\dbl z-\zeta\dbr$ via the expansion $\sum\limits_{i=-\infty}^\infty b_iz^i=\sum\limits_{j=0}^\infty\zeta^{-j}\Bigl(\sum\limits_{i=-\infty}^\infty{i\choose j}b_i\zeta^i\Bigr)(z-\zeta)^j$. The homomorphism \eqref{EqSection} factors through $R\dbl z,z^{-1}\}$. We view the elements of $R\dbl z,z^{-1}\}$ as functions that converge on the punctured open unit disc $\{0<|z|<1\}$. An example of such a function is
\begin{equation}\label{EqTminus}
\tminus:=\prod_{i\in\BN_0}(1-\tfrac{\zeta^{\hat q^i}}{z})\;\in\;\BF_\epsilon\dbl\zeta\dbr\dbl z,z^{-1}\}\;\subset\;R\dbl z,z^{-1}\}\,,
\end{equation}
because the coefficient of $z^{-n}$ is $\sum\limits_{0\le i_1<\ldots<i_n}(-\zeta^{\hat q^{i_1}})\cdots(-\zeta^{\hat q^{i_n}})$ which has absolute value $|\zeta|^{1+\hat q+\ldots+\hat q^{n-1}}=|\zeta|^{(\hat q^n-1)/(\hat q-1)}$. Note that $\tminus=(1-\tfrac{\zeta}{z})\!\cdot\!\hat\sigma(\tminus)$. If we evaluate $\hat\sigma(\tminus)$ at $z=\zeta$ we obtain $\hat\sigma(\tminus)|_{z=\zeta}=\prod_{i\in\BN_{>0}}(1-\zeta^{\hat q^i-1})$. If we were working at the place $\infty$ on the projective line $\BP^1$ instead of at $\epsilon\subset A$, and hence took $\zeta=\theta^{-1}$ and $\hat q=q$, the number $\sqrt[q-1]{-\zeta^q}\cdot\hat\sigma(\tminus)|_{z=\zeta}$ would be the inverse of the period of the Carlitz module which is considered to be the function field analog of $2\pi i$. Moreover, $\sqrt[q-1]{-\zeta^q}\cdot\hat\sigma(\tminus)$ would equal Papanikolas's function $\Omega\in\BC_\infty\dbl\tfrac{1}{z}\dbr$; see \cite[3.3.4]{Papanikolas}.

\smallskip
Further note that in the reduction we have $R\dbl z,z^{-1}\}\otimes_R k=k\dpl z\dpr$ and $\tminus\equiv 1\mod\Fm_R$.
\end{example}

The following lemma of A.~Genestier and V.~Lafforgue \cite[Lemma~7.4]{GL11} is used for the construction of the Hodge-Pink lattice $\Fq_D$ attached to the local shtuka $\ulHM$.

\begin{lemma}\label{LemmaGL}
If $R$ is discretely valued then there is a unique functorial isomorphism
\[
\delta_{\hat M}\colon\es\hat M\otimes_{R\dbl z\dbr}R\dbl z,z^{-1}\}[\tminus^{-1}]\es\isoto\es\wt D\otimes_{k\dpl z\dpr}R\dbl z,z^{-1}\}[\tminus^{-1}]
\]
which satisfies $\delta_{\hat M}\circ \tau_{\hat M}=\tau_{\wt D}\circ\hat\sigma^\ast\delta_{\hat M}$ and $\delta_{\hat M}\equiv\id_{\wt D}\mod\Fm_R$. 
\end{lemma}

\begin{proof}
We refer to \cite[Lemma~7.4]{GL11} and \cite[Lemma~2.3.1]{HartlPSp} for a proof and just present the idea.

We write $\tau_{\hat M}$ as a matrix $T\in\GL_r\bigl(R\dbl z\dbr[\tfrac{1}{z-\zeta}]\bigr)$ and set $S:=T\mod\Fm_R\in\GL_r\bigl(k\dpl z\dpr\bigr)$. Under the fixed section $k\into R$ we view $S$ as an element of $\GL_r\bigl(R\dbl z\dbr[z^{-1}]\bigr)$. We set
\begin{equation}\label{EqCm}
C_m\;:=\;S\cdot\hat\sigma(S)\cdots\hat\sigma^m(S)\cdot\hat\sigma^m(T^{-1})\cdots T^{-1}\,.
\end{equation}
Then $C_m\cdot T=S\cdot \hat\sigma^*C_{m-1}$ and $T^{-1}-S^{-1}\equiv 0\mod\Fm_R$ implies that 
\[
C_m-C_{m-1}\;=\; S\cdots \hat\sigma^m(S)\cdot\hat\sigma^m(T^{-1}-S^{-1})\cdot \hat\sigma^{m-1}(T^{-1})\cdots T^{-1} \;\equiv\;0\mod\Fm_R^{\hat q^m}\,.
\]
From this one derives that $\delta_{\hat M}:=\lim\limits_{m\to\infty}C_m$ converges in $\GL_r\bigl(R\dbl z,z^{-1}\}[\tminus^{-1}]\bigr)$ and satisfies $\delta_{\hat M}\cdot T=S\cdot \hat\sigma^*\delta_{\hat M}$. For example if $T=(z-\zeta)^d$ and hence $S=z^d$,  one obtains $\delta_{\hat M}=\lim\limits_{m\to\infty}\;\prod\limits_{i=0}^m(1-\tfrac{\zeta^{\hat q^i}}{z})^{-d}=\tminus^{-d}$.
\end{proof}

If $R$ is not discretely valued the isomorphism $\delta_{\hat M}$ might not exist, or if it exists it might not be unique; see the discussion on \cite[pp.~1293\,f after Lemma~2.3.1]{HartlPSp}. Therefore we include it into the data.

\begin{definition}\label{DefRigLocSht}
Let $R$ be an arbitrary valuation ring as in Notation~\ref{Notation}. A \emph{rigidified local $\hat\sigma$-shtuka over $R$} is a triple $(\hat M,\tau_{\hat M},\delta_{\hat M})$ where $(\hat M,\tau_{\hat M})$ is a local $\hat\sigma$-shtuka and $\delta_{\hat M}$ is an isomorphism as in Lemma~\ref{LemmaGL}. A morphism of rigidified local shtukas over $R$ is a morphism $f\colon(\hat M,\tau_{\hat M})\to(\hat M',\tau_{\hat M'})$ of local shtukas with $(f\otimes 1_{k\dpl z\dpr})\circ\delta_{\hat M}=\delta_{\hat M'}\circ f$.
\end{definition}

\begin{remark}\label{RemGL}
The proof of Lemma~\ref{LemmaGL} works not only if $R$ is discretely valued, but more generally if $T-S\equiv0\mod (\pi)$ and $T^{-1}-S^{-1}\equiv0\mod(\pi)$ for an element $\pi\in\Fm_R$. In particular if $\ulHM=\ulHM_\epsilon(\ulM)$ for an $A$-motive $\ulM=(M,\tau_M)$ over $R$, then this is satisfied. Namely, after choosing bases of the locally free $A_R$-modules $M$ and $\sigma^*M$ on a neighborhood $\Spec A_R[f^{-1}]$ with $f\in A_R$ of the point $(\Fm_R,\CJ)\in\Spec A_R$, the matrix $T$ of $\tau_M$ with respect to these bases contains only finitely many elements of $R$ and therefore the differences $T-S$ and $T^{-1}-S^{-1}$ likewise contain only finitely many elements of $\Fm_R$. All these elements lie in $(\pi)$ for a suitable $\pi\in\Fm_R$. This implies that there exists a canonical rigidification on $\ulHM=\ulHM_\epsilon(\ulM)$ even if $R$ is an arbitrary valuation ring as in Notation~\ref{Notation}.
\end{remark}

After these preparations we can complete Example~\ref{ExHPOfLocSht} by defining the Hodge-Pink lattice $\Fq_D$ associated with a rigidified local shtuka $(\hat M,\tau_{\hat M},\delta_{\hat M})$ as the image 
\begin{equation}\label{EqHPLattice}
\Fq_D:=\hat\sigma^*\delta_{\hat M}\circ\tau_{\hat M}^{-1}\bigl(\hat M\otimes_{R\dbl z\dbr}K\dbl z-\zeta\dbr\bigr)\,,
\end{equation}
under $\hat\sigma^*\delta_{\hat M}\circ\tau_{\hat M}^{-1}=\tau_{\wt D}^{-1}\circ\delta_{\hat M}\colon\hat M\otimes_{R\dbl z\dbr}K\dpl z-\zeta\dpr\isoto D\otimes_{k\dpl z\dpr}K\dpl z-\zeta\dpr=\Fp_D[\tfrac{1}{z-\zeta}]$. 

\begin{definition}\label{DefMystFunct}
The functor
\begin{eqnarray}\label{EqMystFunct}
\ulHM=(\hat M,\tau_{\hat M},\delta_{\hat M}) & \longmapsto & \BH(\ulHM):=\ulD:=(D,\tau_D,\Fq_D)\quad\text{with}\\
& & (D,\tau_D)=\hat\sigma^*(\hat M,\tau_{\hat M})\otimes_{R\dbl z\dbr}k\dpl z\dpr\text{ and $\Fq_D$ from \eqref{EqHPLattice}} \nonumber
\end{eqnarray}
from the category of rigidified local shtukas with quasi-morphisms over $R$ to the category of $z$-isocrystals with Hodge-Pink structure over $R$ is called the \emph{mysterious functor}. Clearly $\rk\BH(\ulHM)=\rk\ulHM$.
\end{definition}

We quote the following result from \cite[Proposition~2.4.10]{HartlPSp}.

\begin{proposition}\label{PropBHfullyFaithful}
Let $R$ be an arbitrary valuation ring as in Notation~\ref{Notation}. The functor $\BH$ from \eqref{EqMystFunct} is fully faithful.
\end{proposition}

\begin{remark}\label{RemCrystRealiz}
Let $R$ be discretely valued. In view of Theorem~\ref{ThmAndersonKim} we like to view the essential image of $\check V_\epsilon$ as the (analog of the) category of \emph{crystalline Galois representations} and the mysterious functor $\BH$ as the analog of Fontaine's functor from crystalline Galois representations to filtered isocrystals:
\begin{equation}\label{DiagMystFunct}
\xymatrix { \bigl\{\text{crystalline $\Gal(K^\sep/K)$-representations}\bigr\} \ar@{^{ (}->}[r] & \Rep_{Q_\epsilon}\Gal(K^\sep/K)\\
\bigl\{\text{local shtukas over $R$ with quasi-morphisms}\bigr\} \ar[u]_{\TS\check V_\epsilon}^{\TS\cong\,} \ar[r]^{\;\TS\BH} & \bigl\{\text{$z$-isocrystals with Hodge-Pink structure}\bigr\}.
}
\end{equation}
The name ``mysterious functor'' derives from the fact that for an abelian variety $X$ over a $p$-adic field $L$ with good reduction $X_0$, Tate~\cite{Tate66} and Grothendieck~\cite{GrothendieckCristaux} had constructed a functor relating the $p$-adic \'etale cohomology $\Koh_\et^i(X\times_LL^\alg,\BQ_p)$ to the crystalline cohomology $\Koh^i_\cris(X_0/L_0)$ with its Hodge filtration. Here $L_0$ is the fraction field of the ring of $p$-typical Witt vectors with coefficients in the (perfect) residue field of $L$ and $\Koh^i_\cris(X_0/L_0)$ is a ``filtered isocrystal''; see Remark~\ref{RemCompIsom} below. Grothendieck then posed the problem to extend this functor (the ``mysterious functor'') to general proper smooth schemes $X$ over $L$ with good reduction. For those $X$ the problem was solved by Fontaine~\cite{Fontaine79,Fontaine82,Fontaine90,Fontaine94}, who defined the notion of ``crystalline $p$-adic Galois representations'' and constructed a functor from crystalline $p$-adic Galois representations to filtered isocrystals. Fontaine conjectured that $\Koh_\et^i(X\times_LL^\alg,\BQ_p)$ is crystalline. After contributions by Grothendieck, Tate, Fontaine, Lafaille, Messing, Hyodo, Kato and many others, Fontaine's conjecture was proved independently by Faltings~\cite{Faltings89}, Niziol~\cite{Niziol98} and Tsuji~\cite{Tsuji}. Tate's original construction \cite{Tate66} proceeds similarly to our diagram~\eqref{DiagMystFunct}, as it starts with the $p$-divisible group $X[p^\infty]$ of the abelian variety $X$ and associates with it on the one hand the $p$-adic Tate module of $X$, which is the dual of $\Koh_\et^1(X\times_LL^\alg,\BQ_p)$, and on the other hand the filtered isocrystal which is the dual of $\Koh^1_\cris(X_0/L_0)$. In our diagram the category of local shtukas over $R$ with quasi-morphisms takes the place of the $p$-divisible group in Tate's construction. This is a perfect analog, because its full subcategory of effective local shtukas is anti-equivalent to the category of $z$-divisible local Anderson modules (see Remark~\ref{RemEquiv}). The latter arise from Drinfeld modules and abelian Anderson $A$-modules in the same way as $p$-divisible groups arise from abelian varieties (see Example~\ref{ExAndModule}). The reader should also note that in Fontaine's theory the information given by the Hodge filtration is fine enough to determine the Galois representation. In the function field analog this is not the case and one needs the finer information contained in the Hodge-Pink lattice; compare Example~\ref{ExHPFilt}. These observations are the reason for Definition~\ref{DefCrystRep}; see also Remark~\ref{RemEqCharCrystRep}. They suggest the following
\end{remark}

\begin{definition}\label{DefCrystRealiz}
Let $R$ be an arbitrary valuation ring as in Notation~\ref{Notation} and let $\ulHM$ be a rigidified local shtuka over $R$. Then $\Koh^1_\cris\bigl(\ulHM,k\dpl z\dpr\bigr):=\hat\sigma^*(\hat M,\tau_{\hat M})\otimes_{R\dbl z\dbr}k\dpl z\dpr$ is called the \emph{crystalline realization} of $\ulHM$. It is a functor from the category of rigidified local shtukas over $R$ with quasi-morphism to the category of $z$-isocrystals. This functor is faithful by Lemma~\ref{LemmaFaithfulBM} if $\bigcap_n\hat\sigma^n(\Fm_R) = (0)$.

If $\ulM$ is an $A$-motive over $R$ as in Example~\ref{ExAMotive} and $\ulHM_\epsilon(\ulM)$ is its associated rigidified local shtuka (see Remark~\ref{RemGL}), then $\Koh^1_\cris\bigl(\ulM,k\dpl z\dpr\bigr):=\Koh^1_\cris\bigl(\ulHM_\epsilon(\ulM),k\dpl z\dpr\bigr)$ is called the \emph{crystalline realization} of $\ulM$. 

We also define the integral structure $\Koh^1_\cris\bigl(\ulHM,k\dbl z\dbr\bigr):=\hat\sigma^*(\hat M,\tau_{\hat M})\otimes_{R\dbl z\dbr}k\dbl z\dbr\subset\Koh^1_\cris\bigl(\ulHM,k\dpl z\dpr\bigr)$ and $\Koh^1_\cris\bigl(\ulM,k\dbl z\dbr\bigr):=\Koh^1_\cris\bigl(\ulHM_\epsilon(\ulM),k\dbl z\dbr\bigr)$. It is a functor from the category of rigidified local shtukas over $R$ with the usual morphism to the category of $z$-isocrystals. The crystalline realization only depends on the special fiber $\ulHM\otimes_Rk$ of $\ulHM$. This functor is faithful by Lemma~\ref{LemmaFaithfulBM} if $\bigcap_n\hat\sigma^n(\Fm_R) = (0)$.
\end{definition}

\begin{example}\label{ExCarlitz5A}
We continue with the discussion of the Carlitz motive $\ulM=(R[t],t-\theta)$ from Example~\ref{ExCarlitz2} and its local shtuka $\ulHM:=\ulHM_\epsilon(\ulM)=(R\dbl z\dbr,z-\zeta)$. The associated $z$-isocrystal is $(k\dpl z\dpr,z)$. The isomorphism $\delta_{\hat M}$ from Lemma~\ref{LemmaGL} is given by multiplication with the element $\tminus^{-1}$ from \eqref{EqTminus} because $\tminus^{-1}\cdot(z-\zeta)=z\cdot\hat\sigma(\tminus^{-1})$. Since $\hat\sigma(\tminus)$ converges at $z=\zeta$ to the non-zero element $\hat\sigma(\tminus)|_{z=\zeta}=\prod_{i\in\BN_{>0}}(1-\zeta^{\hat q^i-1})\in\BF_\epsilon\dbl\zeta\dbr\mal$, we have $\hat\sigma(\tminus)\in\BF_\epsilon\dpl\zeta\dpr\dbl z-\zeta\dbr\mal$ and $\Fq_D=\hat\sigma(\tminus^{-1})(z-\zeta)^{-1}\cdot K\dbl z-\zeta\dbr=(z-\zeta)^{-1}\Fp_D$. So the Hodge-Pink weight is $-1$ and 
\[
\BH(\ulHM)\;=\;\bigl(k\dpl z\dpr,\,z,\,(z-\zeta)^{-1}\Fp_D\bigr)\;=\;\ul\BOne(-1)\,.
\]
\end{example}

\medskip

In order to describe comparison isomorphisms which relate the crystalline realization $\Koh^1_\cris\bigl(\ulHM,k\dpl z\dpr\bigr)$ of a rigidified local shtuka with the other realizations we let $\olK$ be the completion of an algebraic closure of $K$, and we recall the definition of $\CO_\olK\dbl z,z^{-1}\}$ from \eqref{EqBrig} and the elements $\tminus\in\BF_\epsilon\dbl\zeta\dbr\dbl z,z^{-1}\}$ from \eqref{EqTminus} and $\tplus\in\CO_\olK\dbl z\dbr$ from Example~\ref{ExCarlitz4A}, which also satisfies $\tplus\in\olK\langle\tfrac{z}{\zeta}\rangle\mal$ and $\hat\sigma(\tplus)=(z-\zeta)\!\cdot\!\tplus$. We set
\begin{equation}\label{EqT}
\tplusminus\;:=\;\tplus\tminus\;\in\;\CO_\olK\dbl z,z^{-1}\}\,.
\end{equation}
Then $\hat\sigma(\tplusminus)=z\!\cdot\!\tplusminus$ because $\tminus=(1-\tfrac{\zeta}{z})\!\cdot\!\hat\sigma(\tminus)$, and $g(\tplusminus)=\chi_\epsilon(g)\!\cdot\!\tplusminus$ for $g\in\Gal(K^\sep/K)$ where $\chi_\epsilon$ is the cyclotomic character from Example~\ref{ExCarlitz4A}. With these properties $\tplusminus$ is the function field analog of Fontaine's period $\log[\ul e]\in\wt\bB_\rig$ of the multiplicative group, where $\ul e=(e_i\colon i\in\BN_0)$ is a compatible system of primitive $p^i$-th roots of unity $e_i$ and $[\ul e]$ is its Teichm\"uller lift; see \cite[\S\,2.7]{HartlDict}.

\medskip

The following lemma is the function field analog of the \emph{fundamental exact sequence} \cite[Proposition~1.3(v)]{CF} of $\BQ_p$-vector spaces
\begin{equation}\label{EqFundSeqFontaine}
\xymatrix @C+1pc {
0 \ar[r] & \BQ_p \ar[r] & \wt\bB_\rig^{\phi=\id} \oplus \bB_\dR^+ \ar[r] & \bB_\dR \ar[r] & 0
}
\end{equation}
in which $\wt\bB_\rig^{\phi=\id}=\bB_\cris^{\phi=\id}$; see for example \cite[\S\,II.3.4]{Berger}.

\begin{lemma}\label{LemmaFundSequence}
The following sequence of $Q_\epsilon$-vector spaces is exact
\[
\xymatrix @R=0.2pc {
0 \ar[r] & Q_\epsilon \ar[r] & \CO_\olK\dbl z,z^{-1}\}[\tplusminus^{-1}]^{^{\SC\hat\sigma=\id}} \oplus \olK\dbl z-\zeta\dbr \ar[r] & \olK\dpl z-\zeta\dpr \ar[r] & 0\\
& \quad a \quad \ar@{|->}[r] & \quad\qquad (a,a)\qquad,\qquad (f,g) \qquad\quad \ar@{|->}[r] & \quad\; f-g\,. \quad
}
\]
\end{lemma}

\noindent
{\it Remark.} Note that in \cite[\S\,2.9, last line on p.~1745]{HartlDict} the exactness of a similar sequence was stated. That sequence contains an error, as in the middle term ``$[\tplusminus^{-1}]$'' is missing. 

\bigskip

We prove the lemma simultaneously with the following

\begin{corollary}\label{CorF11Surj}
For every $c\in\olK$ there is an $f=\sum_{i\in\BZ}b_i z^i\in\CO_\olK\dbl z,z^{-1}\}$ with $f=z^{-1}\hat\sigma(f)$ and $c=f(\zeta):=\sum_{i\in\BZ}b_i \zeta^i$.
\end{corollary}

\begin{proof}[Proof of Lemma~\ref{LemmaFundSequence} and Corollary~\ref{CorF11Surj}]
In Lemma~\ref{LemmaFundSequence} the exactness on the left is clear because the maps into each of the summands are ring homomorphisms.

To prove exactness in the middle let $g=\tfrac{b}{\tplusminus^n}\in\CO_\olK\dbl z,z^{-1}\}[\tplusminus^{-1}]^{^{\SC\hat\sigma=\id}}\cap\olK\dbl z-\zeta\dbr$ with $b\in\CO_\olK\dbl z,z^{-1}\}$. Since $\hat\sigma(\tplusminus^n)=z^n\cdot\tplusminus^n$ we must have $b=z^{-n}\hat\sigma(b)$. It follows from \cite[Proposition~1.4.4]{HartlPSp} that $b\in Q_\epsilon\!\cdot\!\tplusminus^n$. This can also be seen directly as follows. For $s\ge r>0$ consider the ring
\[
\olK\langle\tfrac{z}{\zeta^r},\tfrac{\zeta^s}{z}\rangle\;:=\;\bigl\{\;{\TS\sum\limits_{i=-\infty}^\infty} b_i z^i:\es |b_i|\,|\zeta|^{ri} \to 0 \es(i\to\infty),\es |b_i|\,|\zeta|^{si}\to0\es(i\to-\infty)\;\bigr\}\,,
\]
which is a principal ideal domain by \cite[Proposition~4]{Lazard}. Fix a real number $r$ with $1/\hat q<r<1$. Then $1<r\hat q<\hat q$. Therefore the elements $\tplus$ and $\hat\sigma(\tminus)$ are units in $\olK\langle\tfrac{z}{\zeta^r},\tfrac{\zeta^{r\hat q}}{z}\rangle$ and the element $(1-\tfrac{\zeta}{z})^ng=\tplus^{-n}\hat\sigma(\tminus)^{-n}b\in\olK\langle\tfrac{z}{\zeta^r},\tfrac{\zeta^{r\hat q}}{z}\rangle$ has a zero of order $\ge n$ at $z=\zeta$ because $g\in\olK\dbl z-\zeta\dbr$. By \cite[Lemme~2]{Lazard} this shows that it is divisible by $(1-\tfrac{\zeta}{z})^n$ in $\olK\langle\tfrac{z}{\zeta^r},\tfrac{\zeta^{r\hat q}}{z}\rangle$. So $g=\sum\limits_{i=-\infty}^\infty b_iz^i\in\olK\langle\tfrac{z}{\zeta^r},\tfrac{\zeta^{r\hat q}}{z}\rangle$ satisfies $g=\hat\sigma(g)=\sum\limits_{i=-\infty}^\infty b_i^{\hat q}z^i$ in $\olK\langle\tfrac{z}{\zeta^{r\hat q}},\tfrac{\zeta^{r\hat q}}{z}\rangle$. It follows that $b_i=b_i^{\hat q}$, whence $b_i\in\BF_\epsilon$. Now the convergence condition $|b_i|\,|\zeta|^{r\hat qi}\to0$ for $i\to-\infty$ implies that $g\in\BF_\epsilon\dpl z\dpr=Q_\epsilon$ as desired. Thus the sequence is exact in the middle.

\medskip

We next prove Corollary~\ref{CorF11Surj}. The condition $f=z^{-1}\hat\sigma(f)=\sum_{i\in\BZ}b_i^{\hat q} z^{i-1}$ implies $b_i=b_{i-1}^{\hat q^{-1}}=b_0^{{\hat q}^{-i}}$, whence $f=\sum_{i\in\BZ}b_0^{{\hat q}^{-i}}z^i$. The convergence condition on $f$ for $i\to-\infty$ further implies that $|b_0|<1$. Conversely this guarantees that $f\in\CO_\olK\dbl z,z^{-1}\}$. Thus we must find an element $b_0\in\olK$ with $|b_0|<1$ and
\begin{equation}\label{EqLemmaF11Surj}
\sum_{i=-\infty}^\infty b_0^{{\hat q}^{-i}}\zeta^i\;=\;c\,.
\end{equation}
We consider the valuation $v$ on $\olK$ with $v(x):=\log|x|/\log|\zeta|$, which satisfies $|x|=|\zeta|^{v(x)}$ and $v(\zeta)=1$. If $f\in\CO_\olK\dbl z,z^{-1}\}$ satisfies $f=z^{-1}\hat\sigma(f)$ and $f(\zeta)=c$ then $z^nf\in\CO_\olK\dbl z,z^{-1}\}$ satisfies $z^nf=z^{-1}\hat\sigma(z^nf)$ and $(z^nf)(\zeta)=\zeta^nc$ for all $n\in\BZ$. Therefore we may multiply $c$ with an integral power of $\zeta$ and assume that $\tfrac{1}{{\hat q}-1}<v(c)\le 1+\tfrac{1}{{\hat q}-1}$, that is $|\zeta|>|c|^{{\hat q}-1}\ge|\zeta|^{\hat q}$. To solve equation~\eqref{EqLemmaF11Surj} in this case we iteratively try to find $u_n\in\olK$ such that
\begin{equation}\label{EqLemmaF11Surj1}
\sum_{i=-n}^nu_n^{{\hat q}^{-i}}\zeta^i\;=\;c
\end{equation}
for each $n\in\BN$ with $n\ge2$. Writing $u_n=\sum_{i=2}^nx_n$ and $i=n-j$, multiplying \eqref{EqLemmaF11Surj1} with $\zeta^n$ and raising it to the ${\hat q}^n$-th power we have to solve the equations
\begin{equation}\label{EqLemmaF11Surj2}
\sum_{j=0}^{2n}\zeta^{(2n-j){\hat q}^{n}}(x_n)^{{\hat q}^j}\;=\;(x_n)^{{\hat q}^{2n}}+\ldots+\zeta^{(2n-1){\hat q}^n}(x_n)^{{\hat q}}+\zeta^{2n{\hat q}^n}x_n\;=\; c_n
\end{equation}
for $x_n\in\olK$ where $c_2=\zeta^{2{\hat q}^2}c^{{\hat q}^2}$ and $c_n=-(u_{n-1})^{{\hat q}^{2n}}-\zeta^{2n{\hat q}^n}u_{n-1}$ for $n\ge 3$. We claim that there are solutions $x_n\in\olK$ with $|x_2|=|c|$ and $|x_n|^{\hat q}=|c|\,|\zeta|^{{\hat q}^n}<|c|\,|\zeta|^{\hat q}\le|c|^{\hat q}$ for $n\ge3$. From this claim it follows that $|u_n|=|c|$ and that the limit $b_0:=\lim_{n\to\infty}u_n$ exists in $\olK$ with $|b_0|=|c|<1$. In the expression
\[
\sum_{|i|> n} b_0^{{\hat q}^{-i}}\zeta^i\,+\,\sum_{i=-n}^{n} (b_0-u_{n})^{{\hat q}^{-i}}\zeta^i\;=\;\sum_{i=-\infty}^\infty b_0^{{\hat q}^{-i}}\zeta^i\,-\,\sum_{i=-n}^{n} u_{n}^{{\hat q}^{-i}}\zeta^i\;=\;\sum_{i=-\infty}^\infty b_0^{{\hat q}^{-i}}\zeta^i\,-\,c
\]
the first sum goes to zero for $n\to\infty$ because $f\in\CO_\olK\dbl z,z^{-1}\}$. Since $v(b_0-u_n)=v(x_{n+1})={\hat q}^{-1}v(c)+{\hat q}^n$ the $i$-th summand in the second sum has valuation $v\bigl((b_0-u_n)^{{\hat q}^{-i}}\zeta^i\bigr)={\hat q}^{-i-1}v(c)+{\hat q}^{n-i}+i>({\hat q}-1)(n-i)+1+i\ge n+1$ by Bernoulli's inequality. So the second sum likewise goes to zero and $b_0$ solves equation~\eqref{EqLemmaF11Surj}.

To prove the claim we use the Newton polygon of \eqref{EqLemmaF11Surj2}, which is defined as the lower convex hull in the plane $\BR^2$ of the points
\[
\bigl(0, v(c_n)\bigr)\,,\es(1, 2n{\hat q}^n)\,,\es({\hat q},(2n-1){\hat q}^n)\,,\es\ldots\,,\es({\hat q}^j,(2n-j){\hat q}^n)\,,\es\ldots\,,\es({\hat q}^{2n},0)\,.
\]
The piecewise linear function passing through these points has slope $-\tfrac{{\hat q}^{n-j}}{{\hat q}-1}$ on the interval $[{\hat q}^j,{\hat q}^{j+1}]$ for all $j=0,\ldots,2n-1$. In particular, these slopes are strictly increasing. If $n=2$ and $v(c_2)={\hat q}^2v(c)+2{\hat q}^2$ the Newton polygon starts with the line segment of slope $-v(c)$ on the interval $[0,{\hat q}^2]$ because $v(c)\le1+\tfrac{1}{{\hat q}-1}\le2$ implies $v(c_2)-1\cdot v(c)=({\hat q}^2-1)v(c)+2{\hat q}^2<4{\hat q}^2$ and $v(c_2)-{\hat q}\cdot v(c)=({\hat q}^2-{\hat q})v(c)+2{\hat q}^2\le3{\hat q}^2$, and because on the next interval $[{\hat q}^2,{\hat q}^3]$ the slope is $-\tfrac{1}{{\hat q}-1}>-v(c)$. So by the theory of the Newton polygon, \eqref{EqLemmaF11Surj2} has a solution $x_2\in\olK$ with $v(x_2)=v(c)$.

If $n\ge 3$ then $n<3(n-2)+1\le({\hat q}^2-1)(n-2)+({\hat q}-1)$ implies $\sum_{j=n+1}^{2n}({\hat q}^n-{\hat q}^{2n-j})\le n{\hat q}^n<({\hat q}^{n+2}-{\hat q}^n)(n-2)+({\hat q}^{n+1}-{\hat q}^n)\le\sum_{j=1}^n({\hat q}^{2n-j}-{\hat q}^n)$, whence $2n{\hat q}^n<\sum_{j=1}^{2n}{\hat q}^{2n-j}=\tfrac{{\hat q}^{2n}-1}{{\hat q}-1}$. Therefore the induction hypothesis $v(u_{n-1})=v(c)>\tfrac{1}{{\hat q}-1}$ yields 
\[
v(u_{n-1}^{{\hat q}^{2n}})\;=\;{\hat q}^{2n}v(c)\;>\;v(c)+\tfrac{{\hat q}^{2n}-1}{{\hat q}-1}\;>\;v(c)+2n{\hat q}^n\;=\;v(\zeta^{2n{\hat q}^n}u_{n-1})
\]
and hence $v(c_n)=v(c)+2n{\hat q}^n$. It follows that the Newton polygon starts with the line segment of slope $-\tfrac{v(c)}{{\hat q}}-{\hat q}^{n-1}$ on the interval $[0,{\hat q}]$ because $v(c)+2n{\hat q}^n+1\cdot(-\tfrac{v(c)}{{\hat q}}-{\hat q}^{n-1})=\tfrac{({\hat q}-1)v(c)}{{\hat q}}-{\hat q}^{n-1}+2n{\hat q}^n<2n{\hat q}^n$, and because on the next interval $[{\hat q},{\hat q}^2]$ the slope is $-\tfrac{{\hat q}^{n-1}}{{\hat q}-1}\ge-{\hat q}^{n-1}>-\tfrac{v(c)}{{\hat q}}-{\hat q}^{n-1}$. Again by the theory of the Newton polygon, \eqref{EqLemmaF11Surj2} has a solution $x_n\in\olK$ with $v(x_n)=\tfrac{v(c)}{{\hat q}}+{\hat q}^{n-1}$ for $n\ge3$. This proves our claim and hence also the corollary.

\medskip

Finally, to prove exactness on the right in Lemma~\ref{LemmaFundSequence} let $c\,(z-\zeta)^{-n}\in\olK\dpl z-\zeta\dpr$ for $c\in\olK$ and fix an $n$-th root $\sqrt[n]{c}\in\olK$. By the corollary there is an $f\in\CO_\olK\dbl z,z^{-1}\}$ with $f=z^{-1}\hat\sigma(f)$ and $f(\zeta)=\sqrt[n]{c}\cdot(\tfrac{\ell}{z-\zeta})\big|_{z=\zeta}=\sqrt[n]{c}\cdot\bigl(z^{-1}\tplus\hat\sigma(\tminus)\bigr)\big|_{z=\zeta}$. Therefore $f^n/\ell^n\in\CO_\olK\dbl z,z^{-1}\}[\tplusminus^{-1}]^{^{\SC\hat\sigma=\id}}$ and $f/\ell\equiv \sqrt[n]{c}\,(z-\zeta)^{-1}\mod\olK\dbl z-\zeta\dbr$ implies $c\,(z-\zeta)^{-n}-f^n/\ell^n\in (z-\zeta)^{1-n}\olK\dbl z-\zeta\dbr$. In this way we can successively get rid of all denominators in $\olK\dpl z-\zeta\dpr$ and this proves exactness on the right.
\end{proof}

\begin{theorem}\label{ThmCompdRCris}
Let $\ulHM$ be a rigidified local shtuka over $R$. There are canonical functorial comparison isomorphisms between the de Rham and crystalline realizations
\[
\begin{array}{rccl}
h_{\dR,\cris}\colon & \Koh^1_\dR(\ulHM,K\dbl z-\zeta\dbr) & \isoto & \Koh^1_\cris\bigl(\ulHM,k\dpl z\dpr\bigr)\otimes_{k\dpl z\dpr}K\dbl z-\zeta\dbr\quad\text{and}\\[2mm]
h_{\dR,\cris}\colon & \Koh^1_\dR(\ulHM,K) & \isoto & \Koh^1_\cris\bigl(\ulHM,k\dpl z\dpr\bigr)\otimes_{k\dpl z\dpr,\,z\mapsto\zeta}K\,,
\end{array}
\]
which are compatible with the Hodge-Pink lattices and Hodge-Pink filtrations. They usually are \emph{not} compatible with the integral structures $\Koh^1_\dR(\ulHM,R):=\hat\sigma^*\hat M/(z-\zeta)\hat\sigma^*\hat M\subset\Koh^1_\dR(\ulHM,K)$ and $\Koh^1_\cris(\ulHM,k\dbl z\dbr)$.
\end{theorem}

\begin{proof}
Taking the functorial isomorphism $\hat\sigma^*\delta_{\hat M}$ from Lemma~\ref{LemmaGL}, which is an isomorphism over $\hat\sigma^*R\dbl z,z^{-1}\}[\tminus^{-1}]\subset R\dbl z,z^{-1}\}[\hat\sigma(\tminus)^{-1}]$, and using the inclusion $R\dbl z,z^{-1}\}[\hat\sigma(\tminus)^{-1}]\subset K\dbl z-\zeta\dbr$, see Example~\ref{ExCarlitz5A}, we define 
\[
h_{\dR,\cris}:=\hat\sigma^*\delta_{\hat M}\otimes\id_{K\dbl z-\zeta\dbr}\colon\es\hat\sigma^*\hat M\otimes_{R\dbl z\dbr}K\dbl z-\zeta\dbr\es\isoto\es D\otimes_{k\dpl z\dpr}K\dbl z-\zeta\dbr\,.
\]
We will see in Example~\ref{ExOrdDriMod} below that $h_{\dR,\cris}$ does not need to be compatible with the integral structures.
\end{proof}

\begin{example}\label{ExOrdDriMod}
To give an example in which $h_{\dR,\cris}$ is not compatible with the integral structures, we let $A_\epsilon=\BF_q\dbl z\dbr$ and assume that $\hat q=q\ge3$. As our base ring we take $R=k\dbl\pi\dbr$ with $\zeta=\pi^{q+1}$. So the ramification index of $\gamma\colon A_\epsilon\into R$ is $q+1>q$. Over $R$ we consider the local shtuka $\ulHM=\bigl(R\dbl z\dbr^{\oplus2},\tau_{\hat M}=\left(\begin{smallmatrix} 0 & z-\zeta \\[1mm] 1 & \pi-1 \end{smallmatrix}\right)\bigr)$. It is the local shtuka at $\epsilon=(t)\subset\BF_q[t]$ associated with the Drinfeld $\BF_q[t]$-module $\ulE=(\BG_{a,R},\phi)$ with $\phi\colon\BF_q[t]\to R\{\tau\},\,\phi_t=\zeta+(1-\pi)\tau+\tau^2$ which has good ordinary reduction. We compute its rigidification $\delta_M$, or rather $\delta_M^{-1}$ by following the proof of Lemma~\ref{LemmaGL}. We have $T=\left(\begin{smallmatrix} 0 & z-\zeta \\[1mm] 1 & \pi-1 \end{smallmatrix}\right)$ and therefore $S=\left(\begin{smallmatrix} 0 & z \\[1mm] 1 & -1 \end{smallmatrix}\right)$ and $S^{-1}=z^{-1}\left(\begin{smallmatrix} 1 & z \\[1mm] 1 & 0\end{smallmatrix}\right)$. We obtain $C_0^{-1}=T\,S^{-1}=\left(\begin{smallmatrix} 1-\zeta/z & \;\;0 \\[1mm] \pi/z & \;\;1 \end{smallmatrix}\right)$ and $T-S=\left(\begin{smallmatrix} 0 & -\zeta \\[1mm] 0 & \pi \end{smallmatrix}\right)$. The $C_m$ from \eqref{EqCm} satisfy the recursion formula 
\[
C_m^{-1}-C_{m-1}^{-1}\;=\;T\cdots \hat\sigma^{m-1}(T)\cdot\hat\sigma^m(T-S)\cdot \hat\sigma^m(S^{-1})\cdots S^{-1} \;\in\;M_2\bigl(\tfrac{\pi^{q^m}}{z^{m+1}}R\dbl z\dbr\bigr)\,.
\]
We have $\delta_M^{-1}=\lim_{m\to\infty}C_m^{-1}$ and we want to evaluate $\hat\sigma^*\delta_M^{-1}$ at $z=\zeta$. We observe $\hat\sigma(C_0^{-1})|_{z=\zeta}=\left(\begin{smallmatrix} 1-\zeta^q/\zeta & \;\;0 \\[1mm] \pi^q/\zeta & \;\;1\end{smallmatrix}\right)=\left(\begin{smallmatrix} 1-\zeta^{q-1} & \;\;0 \\[1mm] \pi^{-1} & \;\;1\end{smallmatrix}\right)\notin M_2(R)$ and $\hat\sigma(C_m^{-1}-C_{m-1}^{-1})|_{z=\zeta}\in M_2\bigl(\tfrac{\pi^{q^{m+1}}}{z^{m+1}}R\dbl z\dbr\bigr)|_{z=\zeta}\subset M_2(R)$ for $m\ge 1$, because $\tfrac{\pi^{q^{m+1}}}{\zeta^{m+1}}=\pi^{q^{m+1}-(q+1)(m+1)}\in R$ as $q^{m+1}-(q+1)(m+1)\ge0$ for $q\ge3$ and $m\ge1$. This shows that $\hat\sigma(\delta_M^{-1})|_{z=\zeta}\notin\GL_2(R)$, that is $h_{\dR,\cris}$ is not compatible with the integral structures.

Note that this example is the function field analog of the example \cite[Remark~2.10]{BerthelotOgus83} which shows that also for a proper smooth scheme $X$ over $\CO_L$ the comparison isomorphism $\Koh^i_\dR(X/\CO_L)\otimes_{\CO_L}L\isoto\Koh^i_\cris(X_0/W)\otimes_WL$ does not need to be compatible with the integral structures; see Remark~\ref{RemCompIsom}.
\end{example}

\begin{theorem}\label{ThmCompEpsCris}
Let $\ulHM$ be a rigidified local shtuka over $R$. There is a canonical functorial comparison isomorphism between the $\epsilon$-adic and crystalline realizations
\[
h_{\epsilon,\cris}\colon\;\Koh^1_\epsilon(\ulHM,Q_\epsilon)\otimes_{Q_\epsilon}\CO_\olK\dbl z,z^{-1}\}[\tplusminus^{-1}]\;\isoto\;\Koh^1_\cris\bigl(\ulHM,k\dpl z\dpr\bigr)\otimes_{k\dpl z\dpr}\CO_\olK\dbl z,z^{-1}\}[\tplusminus^{-1}]\,.
\]
The isomorphism $h_{\epsilon,\cris}$ is $\Gal(K^\sep/K)$- and $\hat\tau$-equivariant, where on the left module $\Gal(K^\sep/K)$ acts on both factors and $\hat\tau$ is $\id\otimes\hat\sigma$, and on the right module $\Gal(K^\sep/K)$ acts only on $\CO_\olK\dbl z,z^{-1}\}[\tplusminus^{-1}]$ and $\hat\tau$ is $(\tau_D\circ\hat\sigma_{\!D}^*)\otimes \hat\sigma$. In other words $h_{\epsilon,\cris}=\tau_D\circ\hat\sigma^*h_{\epsilon,\cris}$. Moreover, $h_{\epsilon,\cris}$ satisfies $h_{\epsilon,\dR}=(h_{\dR,\cris}^{-1}\otimes\id_{\olK\dpl z-\zeta\dpr})\circ(h_{\epsilon,\cris}\otimes\id_{\olK\dpl z-\zeta\dpr})$. It allows to recover $\Koh^1_\epsilon(\ulHM,Q_\epsilon)$ from $\Koh^1_\cris\bigl(\ulHM,k\dpl z\dpr\bigr)$ as the intersection inside $\Koh^1_\cris\bigl(\ulHM,k\dpl z\dpr\bigr)\otimes_{k\dpl z\dpr}\olK\dpl z-\zeta\dpr$
\[
h_{\epsilon,\cris}\bigl(\Koh^1_\epsilon(\ulHM,Q_\epsilon)\bigr)\;=\;\bigl(\Koh^1_\cris\bigl(\ulHM,k\dpl z\dpr\bigr)\otimes_{k\dpl z\dpr}\CO_\olK\dbl z,z^{-1}\}[\tplusminus^{-1}]\bigr)^{\hat\tau=\id}\,\cap\,\Fq_D\otimes_{K\dbl z-\zeta\dbr}\olK\dbl z-\zeta\dbr\,,
\]
where $\Fq_D\subset\Koh^1_\cris\bigl(\ulHM,k\dpl z\dpr\bigr)\otimes_{k\dpl z\dpr}K\dpl z-\zeta\dpr$ is the Hodge-Pink lattice of $\ulHM$. Since $k\dpl z\dpr\subsetneq R\dbl z,z^{-1}\}\subset\bigl(\CO_\olK\dbl z,z^{-1}\}[\tplusminus^{-1}]\bigr)^{\Gal(K^\sep/K)}$, it does not allow to recover $\Koh^1_\cris\bigl(\ulHM,k\dpl z\dpr\bigr)$ from $\Koh^1_\epsilon(\ulHM,Q_\epsilon)$.
\end{theorem}

\begin{remark}\label{RemCompIsom}
The comparison isomorphisms $h_{\dR,\cris}$ are the function field analogs for the comparison isomorphism $\Koh^i_\dR(X/\CO_L)\otimes_{\CO_L}L\isoto\Koh^i_\cris(X_0/W)\otimes_WL$ between de Rham and crystalline cohomology of a proper smooth scheme $X$ over a complete discrete valuation ring $\CO_L$ with perfect residue field $\kappa$ of characteristic $p$ and fraction field $L$ of characteristic $0$, where $W$ is the ring of $p$-typical Witt vectors with coefficients in $\kappa$ and $X_0=X\otimes_{\CO_L}\kappa$; see \cite[Corollary~2.5]{BerthelotOgus83}. If $L_0$ is the fraction field of $W$ then $\Koh^i_\cris(X_0/L_0)=\Koh^i_\cris(X_0/W)\otimes_WL_0$ is an \emph{$F$-isocrystal}, that is a finite dimensional $L_0$-vector space equipped with a Frobenius linear automorphism, and the Hodge filtration on $\Koh^i_\dR(X/\CO_L)\otimes_{\CO_L}L$ makes it into a \emph{filtered isocrystal} via the comparison isomorphism.

Less straightforward is the comparison isomorphism $h_{\epsilon,\cris}$ which is the function field analog of Fontaine's comparison isomorphism $\Koh^i_\et(X\times_LL^\alg,\BQ_p)\otimes_{\BQ_p}\wt\bB_\rig\isoto\Koh^i_\cris(X_0/L_0)\otimes_{L_0}\wt\bB_\rig$. Namely, with a $\Gal(L^\alg/L)$-representation $V$ such as $\Koh^i_\et(X\times_LL^\alg,\BQ_p)$ Fontaine associates a filtered isocrystal $\bD_\cris(V):=(V\otimes_{\BQ_p}\bB_\cris)^{\Gal(L^\alg/L)}$ and he calls $V$ ``crystalline'' if $\dim_{L_0}\bD_\cris(V)=\dim_{\BQ_p}V$. In this case there is a comparison isomorphism $V\otimes_{\BQ_p}\bB_\cris\isoto\bD_\cris(V)\otimes_{L_0}\bB_\cris$, which is already defined over the subring $\wt\bB_\rig\subset\bB_\cris$; see \cite[II.3.5]{Berger} and \cite[p.~228]{Berger02}. It was then conjectured by Fontaine~\cite[A.11]{Fontaine82} and proved by Faltings~\cite{Faltings89} that $\Koh^i_\et(X\times_LL^\alg,\BQ_p)$ is indeed crystalline and $\bD_\cris\bigl(\Koh^i_\et(X\times_LL^\alg,\BQ_p)\bigr)=\Koh^i_\cris(X_0/L_0)$. The fundamental exact sequence \eqref{EqFundSeqFontaine} allows to also recover
\[
\Koh^i_\et(X\times_LL^\alg,\BQ_p) \; \cong \;  F^0\bigl(\Koh^i_\cris(X_0/L_0)\otimes_{L_0}\wt\bB_\rig\bigr)^{\Frob\,=\id}
\]
from $\Koh^i_\cris(X_0/L_0)$.

In our comparison isomorphism, $\tplusminus$ is the analog of $\log[\ul e]\in\wt\bB_\rig$ where $\ul e=(e_i\colon i\in\BN_0)$ is a compatible system of primitive $p^i$-th roots of unity $e_i$ and $[\ul e]$ is its Teichm\"uller lift; see Example~\ref{ExCarlitz5A}. Our ring $\CO_\olK\dbl z,z^{-1}\}[\tplusminus^{-1}]$ is the analog of $\wt\bB_\rig$; see \cite[\S\,2.5 and \S\,2.7]{HartlDict}. In that sense we could write $F^0\Koh^1_\dR\bigl(\ulHM,K\dpl z-\zeta\dpr\bigr):=\Fq$ and 
\begin{eqnarray*}
& & F^0\Koh^1_\cris\bigl(\ulHM,k\dpl z\dpr\bigr)\otimes_{k\dpl z\dpr}\CO_\olK\dbl z,z^{-1}\}[\tplusminus^{-1}] \es :=  \\[1mm]
& & \qquad\qquad\qquad\qquad \Koh^1_\cris\bigl(\ulHM,k\dpl z\dpr\bigr)\otimes_{k\dpl z\dpr}\CO_\olK\dbl z,z^{-1}\}[\tplusminus^{-1}]\,\cap\,\Fq_D\otimes_{K\dbl z-\zeta\dbr}\olK\dbl z-\zeta\dbr\,.
\end{eqnarray*}
\end{remark}

\begin{proof}[Proof of Theorem~\ref{ThmCompEpsCris}]
To construct $h_{\epsilon,\cris}$ we use for $s\in\BR_{>0}$ the notation
\begin{eqnarray}\label{Kancon}
K\ancon[s] & := & \bigl\{\,{\TS\sum\limits_{i=-\infty}^\infty} b_iz^i\colon b_i\in K\,,\,|b_i|\,|\zeta|^{ri}\to0\;(i\to\pm\infty) \text{ for all }r\ge s\,\bigr\}
\end{eqnarray}
for the ring of rigid analytic functions with coefficients in $K$ on the punctured closed disc $\{0<|z|\le|\zeta|^s\}$ of radius $|\zeta|^s$. The ring $K\ancon$ contains both the rings $K\langle\tfrac{z}{\zeta}\rangle$ and $R\dbl z,z^{-1}\}$, and so $\tplus\in\olK\ancon\mal$. In addition to $\delta_{\hat M}$ we take the isomorphism $h$ from Remark~\ref{RemTateMod} and define $h_{\epsilon,\cris}:=(\hat\sigma^*\delta_{\hat M}\circ\tau_{\hat M}^{-1}\circ h)\otimes\id_{\olK\ancon{}[\tplusminus^{-1}]}$
\[
h_{\epsilon,\cris}\colon\;\check T_\epsilon\ulHM\otimes_{A_\epsilon}\olK\ancon{}[\tplusminus^{-1}]\;\isoto\;\Koh^1_\cris\bigl(\ulHM,k\dpl z\dpr\bigr)\otimes_{k\dpl z\dpr}\olK\ancon{}[\tplusminus^{-1}]\,.
\]
It satisfies $h_{\epsilon,\cris}=\tau_D\circ\hat\sigma^*h_{\epsilon,\cris}$ and is functorial in $\ulHM$. If we choose an $A_\epsilon$-basis of $\check T_\epsilon\ulHM$ and a $k\dpl z\dpr$-basis of $\Koh^1_\cris\bigl(\ulHM,k\dpl z\dpr\bigr)=D:=\hat\sigma^*\hat M\otimes_{R\dbl z\dbr}k\dpl z\dpr$, we may write $h_{\epsilon,\cris}$ and its inverse as matrices $H,H^{-1}\in\GL_r\bigl(\olK\ancon{}[\tplusminus^{-1}]\bigr)$. Let also $T\in\GL_r\bigl(k\dpl z\dpr\bigr)\subset\GL_r\bigl(\CO_\olK\dbl z,z^{-1}\}\bigr)$ be the matrix of $\tau_D$ with respect to the second basis. Then the equivariance of $h_{\epsilon,\cris}$ with respect to $\hat\tau$ implies that $T\cdot\hat\sigma(H)=H$. After multiplying $H$ and $H^{-1}$ each with a high enough power of $\tplusminus$ we have $\tplusminus^nH,\tplusminus^m H^{-1}\in M_r\bigl(\olK\ancon\bigr)$ with $\hat\sigma(\tplusminus^nH)=z^nT^{-1}\tplusminus^nH$ and $\hat\sigma(\tplusminus^mH^{-1})=z^m\tplusminus^mH^{-1}T$. From Lemma~\ref{LemmaBerger} below it follows that $\tplusminus^nH,\tplusminus^m H^{-1}\in M_r\bigl(\CO_\olK\dbl z,z^{-1}\}\bigr)$ and $H\in\GL_r\bigl(\CO_\olK\dbl z,z^{-1}\}[\tplusminus^{-1}]\bigr)$. So $h_{\epsilon,\cris}$ comes from an isomorphism
\[
h_{\epsilon,\cris}\colon\;\Koh^1_\epsilon(\ulHM,Q_\epsilon)\otimes_{Q_\epsilon}\CO_\olK\dbl z,z^{-1}\}[\tplusminus^{-1}]\;\isoto\;\Koh^1_\cris\bigl(\ulHM,k\dpl z\dpr\bigr)\otimes_{k\dpl z\dpr}\CO_\olK\dbl z,z^{-1}\}[\tplusminus^{-1}]\,.
\]
By definition of $h_{\epsilon,\dR}:=(\tau_{\hat M}^{-1}\circ h)\otimes\id_{\olK\dpl z-\zeta\dpr}$ in Theorem~\ref{ThmdR} and $h_{\dR,\cris}:=\hat\sigma^*\delta_{\hat M}\otimes\id_{K\dbl z-\zeta\dbr}$ in Theorem~\ref{ThmCompdRCris} we obtain $h_{\epsilon,\dR}=(h_{\dR,\cris}^{-1}\otimes\id_{\olK\dpl z-\zeta\dpr})\circ(h_{\epsilon,\cris}\otimes\id_{\olK\dpl z-\zeta\dpr})$.

To prove the remaining statement we use the exact sequence of $Q_\epsilon$-vector spaces from Lemma~\ref{LemmaFundSequence} and tensor it with $\Koh^1_\epsilon(\ulHM,Q_\epsilon)$ to obtain the left column in the diagram
\[
\xymatrix @C+1pc {
0 \ar[d] & 0 \ar[d] \\
\Koh^1_\epsilon(\ulHM,Q_\epsilon) \ar[d] \ar[r]_{\TS h_{\epsilon,\cris}\qquad}^{\TS\cong\qquad} & h_{\epsilon,\cris}\bigl(\Koh^1_\epsilon(\ulHM,Q_\epsilon)\bigr) \ar[d] \\
**{!U(0.2) =<15pc,1.2pc>} \objectbox{\parbox{6cm}{\center $\Koh^1_\epsilon(\ulHM,Q_\epsilon)\otimes_{Q_\epsilon}\CO_\olK\dbl z,z^{-1}\}[\tplusminus^{-1}]^{^{\SC\hat\sigma=\id}}$\\[1.5ex]$\bigoplus$}} \ar[r]_{\TS h_{\epsilon,\cris}\qquad}^{\TS\cong\qquad} & 
**{!U(0.2) =<18pc,1.2pc>} \objectbox{\parbox{8cm}{\center $\bigl(\Koh^1_\cris\bigl(\ulHM,k\dpl z\dpr\bigr)\otimes_{k\dpl z\dpr}\CO_\olK\dbl z,z^{-1}\}[\tplusminus^{-1}]\bigr)^{\hat\tau=\id}$\\[1.5ex]$\bigoplus$}}\\
\Koh^1_\epsilon(\ulHM,Q_\epsilon)\otimes_{Q_\epsilon}\olK\dbl z-\zeta\dbr \ar[d] \ar[r]_{\TS h_{\epsilon,\cris}\qquad}^{\TS\cong\qquad} & 
\Fq_D\otimes_{K\dbl z-\zeta\dbr}\olK\dbl z-\zeta\dbr \ar[d] \\
\Koh^1_\epsilon(\ulHM,Q_\epsilon)\otimes_{Q_\epsilon}\olK\dpl z-\zeta\dpr \ar[r]_{\TS h_{\epsilon,\cris}\qquad}^{\TS\cong\qquad} &
\Koh^1_\cris\bigl(\ulHM,k\dpl z\dpr\bigr)\otimes_{k\dpl z\dpr}\olK\dpl z-\zeta\dpr\,.\!
}
\]
The second horizontal map is an isomorphism because $h_{\epsilon,\cris}$ is $\hat\tau$-equivariant. The third horizontal map is an isomorphism because $h_{\epsilon,\dR}\bigl(\Koh^1_\epsilon(\ulHM,Q_\epsilon)\otimes_{Q_\epsilon}\olK\dbl z-\zeta\dbr\bigr)=\Fq\otimes_{K\dbl z-\zeta\dbr}\olK\dbl z-\zeta\dbr$ by Theorem~\ref{ThmdR} and because $h_{\dR,\cris}$ is compatible with the Hodge-Pink lattices by Theorem~\ref{ThmCompdRCris}. From the right column it follows that
\[
h_{\epsilon,\cris}\bigl(\Koh^1_\epsilon(\ulHM,Q_\epsilon)\bigr)\;=\;\bigl(\Koh^1_\cris\bigl(\ulHM,k\dpl z\dpr\bigr)\otimes_{k\dpl z\dpr}\CO_\olK\dbl z,z^{-1}\}[\tplusminus^{-1}]\bigr)^{\hat\tau=\id}\,\cap\,\Fq_D\otimes_{K\dbl z-\zeta\dbr}\olK\dbl z-\zeta\dbr\,.
\]
\end{proof}

The following lemma is the function field analog of \cite[Proposition~I.4.1 and Corollary~I.4.2]{Berger04a}.

\begin{lemma}\label{LemmaBerger}
Let $U\in M_r(K\ancon[s])$ for some $s>0$, and let $V,W\in M_r(R\dbl z,z^{-1}\})$ such that $\hat\sigma(U)=VUW$. Then $U\in M_r(R\dbl z,z^{-1}\})$.
\end{lemma}

\begin{remark}\label{RemAdmAlg}
The lemma and its proof are valid more generally if $R$ is replaced by an admissible formal $R$-algebra $B\open$ in the sense of Raynaud, if $K$ is replaced by $B:=B\open\otimes_RK$ and $|\,.\,|$ is a $K$-Banach norm on $B$ with $B\open=\{b\in B\colon|b|\le1\}$; see \cite{FRG1}.
\end{remark}

\begin{proof}[Proof of Lemma~\ref{LemmaBerger}]
For the elements $\sum\limits_{i=-\infty}^\infty b_iz^i\in K\ancon[s]$ we consider the norm $\bigl\|\sum_ib_iz^i\bigr\|_s:=\max_{i\in\BZ}|b_i|\,|\zeta|^{si}$ and we extend this to a norm on the matrix $U$ by taking the maximum of the norms of the entries of $U$. It satisfies $\|VU\|_s\le\|V\|_s\cdot\|U\|_s$. Since $\|z\|_s=|\zeta|^s<1$ there is an integer $c$ such that $\|z^cU\|_s\le1$. Likewise there is an integer $k$ such that $\|z^kV\|_s,\|z^kW\|_s\le1$. The entries of $V$ and $W$ are of the form $\sum_ib_iz^i$ with $|b_i|\le1$. For those we have $|b_i|^{\hat q}\le|b_i|$ and this implies $\|\hat\sigma^n(\sum_ib_iz^i)\|_s=\|\sum_ib_i^{\hat q^n}z^i\|_s\le\|\sum_ib_iz^i\|_s$, and hence $\|\hat\sigma^n(z^kV)\|_s,\|\hat\sigma^n(z^kW)\|_s\le1$ for all $n\in\BN_0$. We obtain
\[
\bigl\|\hat\sigma^{n+1}(z^{c+2k(n+1)}U)\bigr\|_s\;=\;\bigl\|\hat\sigma^n(z^kV)\hat\sigma^n(z^{c+2kn}U)\hat\sigma^n(z^kW)\bigr\|_s\;\le\;\bigl\|\hat\sigma^n(z^{c+2kn}U)\bigr\|_s\;\le\;1
\]
for all $n\in\BN_0$ by induction. Let $\sum_ib_iz^i$ be an entry of $U$. Then $\sum_ib_i^{\hat q^n}z^{i+c+2kn}$ is the corresponding entry of $\hat\sigma^n(z^{c+2kn}U)$. The estimate $1\ge\bigl\|\sum_ib_i^{\hat q^n}z^{i+c+2kn}\bigr\|_s=\max_i|b_i|^{\hat q^n}|\zeta|^{s(i+c+2kn)}$ implies $|b_i|\le|\zeta|^{-s(i+c+2kn)\hat q^{-n}}$ for all $i$. For fixed $i$ and $n\to+\infty$ the exponent goes to $0$ and so we find $|b_i|\le1$, that is $b_i\in R$ for all $i\in\BZ$. In order that $\sum_ib_iz^i\in R\dbl z,z^{-1}\}$ we have to verify that for all $r>0$ the condition $|b_i|\,|\zeta|^{ri}\to0$ for $i\to-\infty$ holds. From $\sum_ib_iz^i\in K\ancon[s]$ we already know that this condition holds for all $r\ge s$. If $r<s$ and $i<0$ then $|b_i|\,|\zeta|^{ri}<|b_i|\,|\zeta|^{si}$ and so the condition also holds for the remaining $r<s$. This proves that $U\in M_r(R\dbl z,z^{-1}\})$ as desired.
\end{proof}

\begin{example}\label{ExCarlitz5B}
For the Carlitz motive from Examples~\ref{ExCarlitz4B} and \ref{ExCarlitz5A} the comparison isomorphisms from Theorems~\ref{ThmCompdRCris} and \ref{ThmCompEpsCris} between $\Koh^1_\dR(\ulM,K\dbl z-\zeta\dbr)=K\dbl z-\zeta\dbr$ with basis $1$ and $\Koh^1_\epsilon(\ulM,Q_\epsilon)=Q_\epsilon\cdot\tplus^{-1}$ with basis $\tplus^{-1}$ on the one hand and $\Koh^1_\cris\bigr(\ulM,k\dpl z\dpr\bigr)=(k\dpl z\dpr,z,(z-\zeta)^{-1}\Fp_D)$ with basis $1$ on the other hand are given explicitly as follows. With respect to these bases $h_{\dR,\cris}$ is given by $\hat\sigma(\tminus^{-1})\in K\dbl z-\zeta\dbr\mal$ and $h_{\epsilon,\cris}$ is given by $\hat\sigma(\tplusminus^{-1})=\hat\sigma(\tminus)^{-1}\hat\sigma(\tplus)^{-1}=z^{-1}\tplusminus^{-1}$. In particular, $h_{\dR,\cris}$ modulo $z-\zeta$ is given by $\prod_{i\in\BN_{>0}}(1-\zeta^{\hat q^i-1})^{-1}\in\CO_K\mal$. So in this example $h_{\dR,\cris}$ is compatible with the integral structures.
\end{example}

%
%

\section{Admissibility and weak admissibility}\label{SectWA}
\setcounter{equation}{0}

It turns out that not all $z$-isocrystals with Hodge-Pink structure can arise via the functor $\BH$ from \eqref{EqMystFunct}. Namely a necessary condition is weak admissibility as in the following

\begin{definition}\label{DefWA}
Let $R$ be an arbitrary valuation ring as in Notation~\ref{Notation}, let $\ulD=(D,\tau_D,\Fq_D)$ be a $z$-isocrystal with Hodge-Pink structure over $R$ and set $r=\dim_{k\dpl z\dpr}D$.
\begin{enumerate}
\item Choose a $k\dpl z\dpr$-basis of $D$ and let $\det\tau_D$ be the determinant of the matrix representing $\tau_D$ with respect to this basis. The number $t_N(\ulD):=\ord_z(\det\tau_D)$ is independent of this basis and is called the \emph{Newton slope of $\ulD$}.
\item 
The integer $t_H(\ulD):=-\mu_1-\ldots-\mu_r$, where $\mu_1,\ldots,\mu_r$ are the Hodge-Pink weights of $\ulD$ from Definition~\ref{DefHPWts}, satisfies $\wedge^r\Fq_D=(z-\zeta)^{-t_H(\ulD)}\wedge^r\Fp_D$ and is called the \emph{Hodge slope of $\ulD$}.
\item 
$\ulD$ is called \emph{weakly admissible} (or \emph{semi-stable of slope $0$}) if 
\[
t_H(\ulD)=t_N(\ulD)\qquad \text{and}\qquad t_H(\ulD')\le t_N(\ulD')\quad\text{for every strict subobject $\ulD'\subset\ulD$.}
\]
\item 
$\ulD$ is called \emph{admissible} if there exists a rigidified local $\hat\sigma$-shtuka $\ulHM$ over $R$ with $\ulD=\BH(\ulHM)$.
\end{enumerate}
\end{definition}

\begin{remark}\label{RemCompWithHJ_B}
This definition parallels Fontaine's definition \cite{Fontaine79} of (weak) admissibility of filtered isocrystals. Namely, Fontaine calls a filtered isocrystal admissible if it comes from a crystalline Galois representation; compare Remark~\ref{RemCrystRealiz}.

In comparison with the Hodge-Pink structures at $\infty$ from \HJSectHPStruct{} the Hodge slope $t_H(\ulD)$ corresponds to $\deg_\Fq$, the Newton slope $t_N(\ulD)$ corresponds to $\deg^W$, and our notion of weak admissibility corresponds to Pink's notion of ``local semi-stability''; see \HJDefHPWts{}.
\end{remark}

\begin{example}\label{ExCarlitz6A}
For the Carlitz motive from Example~\ref{ExCarlitz5A} and its $z$-isocrystal with Hodge-Pink structure $\ulD=(k\dpl z\dpr,z,(z-\zeta)^{-1}\Fp_D)$ one has $t_N(\ulD)=1=t_H(\ulD)$. It is (weakly) admissible.
\end{example}

To obtain a criterion for (weak) admissibility we need to recall a little bit of the theory of $\sigma$-bundles from \cite{HartlPink1}, see also \HJSectSigmaBd. Here we consider a variant which also works if $K$ is not algebraically closed. Namely, we do not consider the punctured open unit disc $\{0<|z|<1\}$ but instead the punctured closed disc $\{0<|z|\le|\zeta|\}$ of radius $|\zeta|$, and the ring $K\ancon$ of rigid analytic functions with coefficients in $K$ on it from \eqref{Kancon}. There are two morphisms from $K\ancon$ to $K\ancon[\hat q]$, namely
\begin{eqnarray*}
& \iota\colon K\ancon\longto K\ancon[\hat q], & \TS\sum_i b_i z^i \longmapsto \sum_i b_i z^i \qquad\text{and}\\[2mm]
& \hat\sigma\colon K\ancon\longto K\ancon[\hat q], & \TS\sum_i b_i z^i \longmapsto \sum_i b_i^{\hat q} z^i\,.
\end{eqnarray*}
For a $K\ancon$-module $\CF$ we define the two $K\ancon[\hat q]$-modules $\hat\sigma^*\CF:=\CF\otimes_{K\ancon,\,\hat\sigma}K\ancon[\hat q]$ and $\iota^*\CF:=\CF\otimes_{K\ancon,\,\iota}K\ancon[\hat q]$.

\begin{definition}\label{DefSigmaBd}
A \emph{$\hat\sigma$-bundle} (\emph{on $\{0<|z|\le|\zeta|\}$}) \emph{of rank $r$ over $K$} is a pair $\ulCF=(\CF,\tau_\CF)$ consisting of a free $K\ancon$-module $\CF$ of rank $r$ together with an isomorphism $\tau_\CF\colon\hat\sigma^*\CF\isoto\iota^*\CF$ of $K\ancon[\hat q]$-modules.

A \emph{homomorphism} $f\colon(\CF,\tau_\CF)\to(\CG,\tau_\CG)$ between $\hat\sigma$-bundles is a homomorphism $f\colon\CF\to\CG$ of $K\ancon$-modules which satisfies $\tau_\CF\circ\hat\sigma^*f=\iota^*f\circ\tau_\CG$.

The \emph{$\hat\tau$-invariants} of $(\CF,\tau_\CF)$ are defined as $\ulCF^{\hat\tau}:=\{\,f\in\CF:\tau_\CF(\hat\sigma^\ast f)=f\,\}$. It is a vector space over $\{f\in K\ancon:\hat\sigma(f)=f\,\}=\BF_\epsilon\dpl z\dpr=Q_\epsilon$.
\end{definition}

\begin{example}\label{ExSigmaBd}
1. The trivial $\hat\sigma$-bundle over $K$ is $\ulCF_{_{\SC 0,1}}:=\bigl(K\ancon,\id_{K\ancon[\hat q]}\bigr)$. 

\medskip\noindent
2. More generally, for relatively prime integers $d,r$ with $r>0$ we let $\ulCF_{d,r}$ be the $\hat\sigma$-bundle over $K$ consisting of $\CF_{d,r}=K\ancon^{\oplus r}$ with 
\[
\tau_{_{\CF_{d,r}}}:= \left( \raisebox{6.2ex}{$
\xymatrix @C=0.3pc @R=0.3pc {
0 \ar@{.}[ddd]\ar@{.}[drdrdrdr] & 1\ar@{.}[drdrdr] & 0\ar@{.}[rr]\ar@{.}[drdr] & & 0 \ar@{.}[dd]\\
& & & & \\
& & & & 0 \\
0 & & & & 1\\
z^{-d} & 0 \ar@{.}[rrr]  & & & 0\\
}$}
\right).
\]

\medskip\noindent
3. We already saw in equation~\eqref{EqT} that $\tplusminus\in(\ulCF_{_{\SC 1,1}})^{\hat\tau}$.
\end{example}

\begin{lemma}\label{LemmaF11Surj}
If $K$ is algebraically closed the evaluation at $z=\zeta$ induces an exact sequence of $\BF_\epsilon\dpl z\dpr$-vector spaces
\[
\xymatrix @R=0pc {
0 \ar[r] & \BF_\epsilon\dpl z\dpr\cdot\tplusminus \ar[r] & (\ulCF_{_{\SC 1,1}})^{\hat\tau} \ar[r] & K \ar[r] & 0\,, \\
& & f(z) \ar@{|->}[r] & f(\zeta)
}
\]
where $K$ is an $\BF_\epsilon\dpl z\dpr$-vector space via the inclusion $\BF_\epsilon\dpl z\dpr\into K$, $z\mapsto\zeta$.
\end{lemma}

\begin{proof}
The sequence is obviously exact on the left. Furthermore, $\tplusminus$ maps to zero under the right morphism. To prove exactness in the middle let $f\in(\ulCF_{_{\SC 1,1}})^{\hat\tau}$ vanish at $z=\zeta$. Then it vanishes at $z=\zeta^{\hat q^i}$ for all $i\in\BN_0$ and $g:=f/\tplusminus\in K\ancon$ satisfies $g=\hat\sigma^*(g)\in\BF_\epsilon\dpl z\dpr$ as desired. The exactness on the right was established in Corollary~\ref{CorF11Surj}.
\end{proof}

The structure theory of $\hat\sigma$-bundles over an algebraically closed complete field (for example over the completion $\olK$ of an algebraic closure of $K$) was developed in \cite{HartlPink1} and \cite[\S\,1.4]{HartlPSp}.

\begin{theorem}\label{ThmHartlPink1} We assume that $K$ is algebraically closed (or we work over $\olK$).
\begin{enumerate}
\item \label{ThmHartlPink1a}
Any $\hat\sigma$-bundle $\ulCF$ over $K$ is isomorphic to $\bigoplus_i\ulCF_{d_i,r_i}$ for pairs of relatively prime integers $d_i,r_i$ with $r_i>0$, which are uniquely determined by $\ulCF$ up to permutation. They satisfy $\rk\ulCF=\sum_i r_i$ and we define the \emph{degree} of $\ulCF$ as $\deg\ulCF:=\sum_i d_i$.
\item  \label{ThmHartlPink1b}
There is a non-zero morphism $\ulCF_{d',r'}\to\ulCF_{d,r}$ if and only if $\frac{d'}{r'}\le\frac{d}{r}$.
\item  \label{ThmHartlPink1c}
Any $\hat\sigma$-sub-bundle $\ulCF'\subset(\ulCF_{d,r})^{\oplus n}$ satisfies $\deg\ulCF'\le\frac{d}{r}\cdot\rk\ulCF'$.
\end{enumerate}
\end{theorem}

\begin{proof}
Statements (a) and (b) are \cite[Theorem~1.4.2 and Proposition~1.4.5]{HartlPSp}. Statement (c) follows easily from (a) and (b). Namely, $\CF'\cong\bigoplus_i\ulCF_{d_i,r_i}$ by (a) with $\frac{d_i}{r_i}\le\frac{d}{r}$ by (b) yields (c).
\end{proof}

\bigskip

Now let $(\ulHM,\delta_{\hat M})=(\hat M,\tau_{\hat M},\delta_{\hat M})$ be a rigidified local shtuka over $R$ as in Definition~\ref{DefRigLocSht}, and let $\ulD=(D,\tau_D,\Fq_D)$ be the associated $z$-isocrystal with Hodge-Pink structure over $R$. We define two $\hat\sigma$-bundles on $\{0<|z|\le|\zeta|\}$
\begin{eqnarray*}
\ulCE(\ulHM) & := & (\CE(\ulHM),\tau_\CE) \es := \es \hat\sigma^*\ulHM\otimes_{R\dbl z\dbr}k\dpl z\dpr\otimes_{k\dpl z\dpr}K\ancon \es = \es (D,\tau_D)\otimes_{k\dpl z\dpr}K\ancon \es\text{and}\\[2mm]
\ulCF(\ulHM) & := & (\CF(\ulHM),\tau_\CF) \es := \es \ulHM\otimes_{R\dbl z\dbr}K\ancon\,.
\end{eqnarray*}
There is a canonical isomorphism $\hat\sigma^*\delta_{\hat M}\circ\tau_{\hat M}^{-1}\colon\ulCF(\ulHM)[\tminus^{-1}]\isoto\ulCE(\ulHM)[\tminus^{-1}]$. We use this isomorphism to view $\CF(\ulHM)$ as a $K\ancon$-submodule of $\CE(\ulHM)[\tminus^{-1}]$. If $\ulHM$ is effective, that is if $\tau_{\hat M}(\hat\sigma^*\hat M)\subset\hat M$, we have $\CE(\ulHM)\subset\hat\sigma^*\delta_{\hat M}\circ\tau_{\hat M}^{-1}(\CF(\ulHM))$. Assume that $\tau_{\hat M}(\hat\sigma^*\hat M)\subsetneq\hat M$. Then we visualize these $\hat\sigma$-bundles by the following diagram, in which the thick lines represent $K\ancon$-modules:

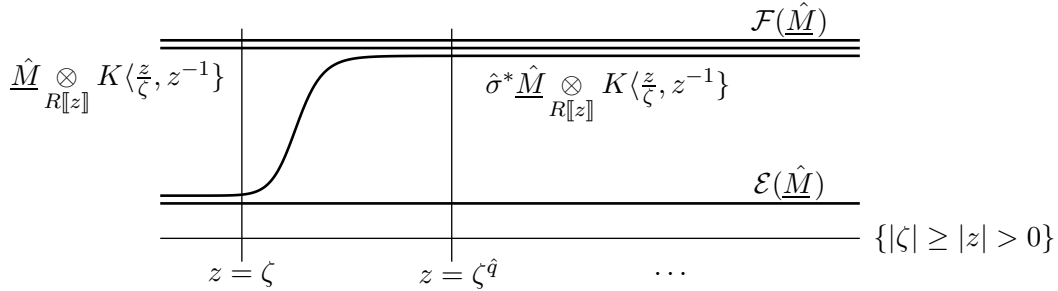
\begin{center}
\newpsobject{showgrid}{psgrid}{subgriddiv=1,griddots=10,gridlabels=0pt}
 \psset{unit=3.1cm} 
 \begin{pspicture}(1.6,-0.4)(4.6,0.85) 
   \rput(4.3,0.79){$\CF(\ulHM)$}
   \psplot[plotstyle=curve,plotpoints=200,linewidth=1pt,linecolor=black]{1.6}%
                {4.6}{0.7}
   \rput(4.3,0.09){$\CE(\ulHM)$}
   \psplot[plotstyle=curve,plotpoints=200,linewidth=1pt,linecolor=black]{1.6}%
                {4.6}{6 7 div 0.7 mul 0.6 sub}
   \rput(1.42,0.5){$\ulHM\!\underset{R\dbl z\dbr}{\otimes}K\ancon$}
   \psline[linewidth=1pt](1.6,0.666667)(4.6,0.666667)
   \rput(3.52,0.466667){$\hat\sigma^*\ulHM\!\underset{R\dbl z\dbr}{\otimes}K\ancon$}
   \psplot[plotstyle=curve,plotpoints=200,linewidth=1pt,linecolor=black]{1.6}{4.6}%
                {x 1.3 sub 16 exp 6 mul x 1.3 sub 16 exp 7 mul 1 add div 0.7 mul 0.1 3 div add}
   \rput(5.05,-0.15){$\{|\zeta|\ge|z|>0\}$}
   \psline[linewidth=0.5pt](1.6,-0.15)(4.6,-0.15)
   \rput(1.95,-0.3){$z=\zeta$}
   \psline[linewidth=0.5pt](1.95,-0.25)(1.95,0.75)
   \rput(2.88,-0.3){$z=\zeta^{\hat q}$}  
   \psline[linewidth=0.5pt](2.85,-0.25)(2.85,0.75)
   \rput(3.8,-0.3){$\dots$}  
 \end{pspicture}
\end{center}

\noindent
Modules drawn higher contain the ones drawn below. All $K\ancon$-modules coincide outside the discrete set $\bigcup_{i\in\BN_0}\{z=\zeta^{\hat q^i}\}\subset\{|\zeta|\ge|z|>0\}$. At those points in $\bigcup_{i\in\BN_0}\{z=\zeta^{\hat q^i}\}$ where two modules are drawn at almost the same height, they also coincide. Indeed, $\CF(\ulHM)$ equals $\hat M$. Moreover, it contains $\tau_{\hat M}(\hat\sigma^\ast\hat M)$ and differs from it only at $z=\zeta$ by our assumption. Also $\CE(\ulHM)$ coincides with $(\hat\sigma^\ast\hat M)\otimes_{R\dbl z\dbr}K\ancon$ at $z=\zeta$ because $\hat\sigma^*\delta_{\hat M}$ is an isomorphism of $K\ancon{}[\hat\sigma^*\tminus^{-1}]$-modules and $\hat\sigma^*\tminus$ is a unit in $K\dbl z-\zeta\dbr$. Finally, the strict inclusion $\tau_{\hat M}(\hat\sigma^*\hat M)\subsetneq\hat M$ at $z=\zeta$ is transported by $\tau_\CF$ and $\tau_\CE$ to strict inclusions $\CE(\ulHM)\subset\CF(\ulHM)$ at $\bigcup_{i\in\BN_0}\{z=\zeta^{\hat q^i}\}$. 

The stalks at $z=\zeta$, or more precisely the tensor products with $K\dbl z-\zeta\dbr$ satisfy
\begin{eqnarray}\label{EqCF}
\CE(\ulHM)\otimes_{K\ancon} K\dbl z-\zeta\dbr & = & \Fp_D \\[2mm]
\hat\sigma^*\delta_{\hat M}\circ\tau_{\hat M}^{-1}(\CF(\ulHM))\otimes_{K\ancon} K\dbl z-\zeta\dbr & = & \Fq_D \es \subset \es \CE(\ulHM)\otimes_{K\ancon} K\dpl z-\zeta\dpr \es = \es \Fp_D[\tfrac{1}{z-\zeta}]\,. \nonumber
\end{eqnarray}
We thus can view $\CF(\ulHM)$ as the modification of $\CE(\ulHM)$ at $\bigcup_{i\in\BN_0}\{z=\zeta^{\hat q^i}\}$ defined by the inclusion $\Fq_D\subset\Fp_D[\tfrac{1}{z-\zeta}]$, that is
\begin{eqnarray*}
\hat\sigma^*\delta_{\hat M}\circ\tau_{\hat M}^{-1}(\CF(\ulHM)) & = & \ulCE(\ulHM)[\tminus^{-1}]\cap\bigcap_{j\in\BN_0}\tau_\CE^j\bigl(\Fq_D\otimes_{K\dbl z-\zeta\dbr,\,\hat\sigma^j}K\dbl z-\zeta^{\hat q^j}\dbr\bigr)\\[2mm]
& = & \bigl\{\,f\in\CE(\ulHM)[\tminus^{-1}]\colon \tau_\CE^i(\hat\sigma^i(f))\in\Fq_D\text{ for all }i\le0\,\bigr\}\,.
\end{eqnarray*}
From this description it becomes clear that $\ulCE(\ulHM)$ and $\ulCF(\ulHM)$ only depend on the $z$-isocrystal with Hodge-Pink structure $\ulD=(D,\tau_D,\Fq_D)=\BH(\ulHM)$ associated with $\ulHM$. Namely, we can define
\begin{eqnarray}\label{EqCEFromD}
\ulCE(\ulD) & := & (\CE(\ulD),\tau_\CE) \es := \es (D,\tau_D)\otimes_{k\dpl z\dpr}K\ancon \qquad\text{and}\\[2mm]
\CF(\ulD) & = & (\CF(\ulD),\tau_\CF) \es := \es \ulCE(\ulD)[\tminus^{-1}]\cap\bigcap_{j\in\BN_0}\tau_\CE^j\bigl(\Fq_D\otimes_{K\dbl z-\zeta\dbr,\,\hat\sigma^j}K\dbl z-\zeta^{\hat q^j}\dbr\bigr)\,.\label{EqCFFromD}
\end{eqnarray}

\begin{definition}\label{DefCFFromD}
We call $\bigl(\ulCE(\ulD),\ulCF(\ulD)\bigr)$ from \eqref{EqCEFromD} and \eqref{EqCFFromD} the \emph{pair of $\hat\sigma$-bundles associated with $\ulD$}.
\end{definition}

The following proposition is analogous to \HJPropDeg{} and was proved in \cite[Lemma~2.4.5]{HartlPSp}, where we used the notation $\CP=\ulCE(\ulD)$ and $\CQ=\ulCF(\ulD)$.

\begin{proposition}\label{PropDeg}
The degree (defined in Theorem~\ref{ThmHartlPink1}\ref{ThmHartlPink1a}) satisfies
\[
t_N(\ulD)\;=\;-\deg\ulCE(\ulD) \qquad\text{and} \qquad t_H(\ulD)\;=\;\deg\ulCF(\ulD)-\deg\ulCE(\ulD)\,.
\]
\end{proposition}

The criterion for (weak) admissibility is now the following

\begin{theorem}\label{ThmCritWA}
Let $\ulD$ be a $z$-isocrystal with Hodge-Pink structure over $R$ and let $\bigl(\ulCE(\ulD),\ulCF(\ulD)\bigr)$ be the associated pair of $\hat\sigma$-bundles. Then
\begin{enumerate}
\item \label{ThmCritWA_A}
$\ulD$ is admissible if and only if over $\olK$ there is an isomorphism $\ulCF(\ulD)\cong(\ulCF_{_{\SC 0,1}})^{\oplus\rk\ulD}$.
\item \label{ThmCritWA_B}
$\ulD$ is weakly admissible if and only if $\deg\ulCF(\ulD)=0\ge\deg\ulCF(\ulD')$ for every strict subobject $\ulD'\subset\ulD$.
\end{enumerate}
\end{theorem}

\begin{proof}
\ref{ThmCritWA_B} directly follows from Proposition~\ref{PropDeg}.

\smallskip\noindent
\ref{ThmCritWA_A} was proved in \cite[Theorem~2.4.7]{HartlPSp}. We explain some parts of the proof.

The implication ``$\Longrightarrow$'' is easy to see. Choose a rigidified local shtuka $\ulHM:=(\hat M,\tau_{\hat M},\delta_{\hat M})$ over $R$ with $\ulD=\BH(\ulHM)$. We saw in Remark~\ref{RemTateMod} that there is an isomorphism $h\colon\check T_\epsilon\ulHM\otimes_{A_\epsilon}K^\sep\langle\tfrac{z}{\zeta}\rangle\isoto\ulHM\otimes_{R\dbl z\dbr}K^\sep\langle\tfrac{z}{\zeta}\rangle$, which is $\Gal(K^\sep/K)$- and $\hat\tau$-equivariant. Since $K^\sep\langle\tfrac{z}{\zeta}\rangle\subset\olK\ancon$ we obtain an isomorphism of $\hat\sigma$-bundles $(\ulCF_{_{\SC 0,1}})^{\oplus\rk\ulM}\isoto\check T_\epsilon\ulHM\otimes_{A_\epsilon}\ulCF_{_{\SC 0,1}}\isoto\ulCF(\ulHM)=\ulCF(\ulD)$ over $\olK$.

The converse implication ``$\Longleftarrow$'' is more complicated, as one has to construct the local shtuka out of $\ulCF(\ulD)\cong(\ulCF_{_{\SC 0,1}})^{\oplus\rk\ulD}$; see \cite[Theorem~2.4.7 and Proposition~2.4.9]{HartlPSp}.
\end{proof}

\begin{corollary}\label{CorA=>WA}
A $z$-isocrystal with Hodge-Pink structure over $R$ which is admissible is also weakly admissible.
\end{corollary}

\begin{proof}
This follows from the characterization of (weak) admissibility in Theorem~\ref{ThmCritWA} together with Theorem~\ref{ThmHartlPink1}\ref{ThmHartlPink1c} applied to the $\hat\sigma$-sub-bundle $\ulCF(\ulD')\subset\ulCF(\ulD)$.
\end{proof}

Whether the converse of Corollary~\ref{CorA=>WA} holds, depends on the field $K$. Before we state cases in which it holds, we discuss the following

\begin{example}\label{ExCarlitz6B}
For the Carlitz motive and its local shtuka $\ulHM=(R\dbl z\dbr,z-\zeta)$ from Example~\ref{ExCarlitz6A} we have $\ulCE(\ulHM)=\ulCF_{_{\SC -1,1}}$ and $\ulCF(\ulHM)=\bigl(K\ancon,z-\zeta\bigr)$ which is isomorphic over $\olK\ancon$ to $\ulCF_{_{\SC 0,1}}$ via an isomorphism given by multiplication with $\tplus$.
\end{example}

\begin{example}\label{ExDrinfeld}
If the $z$-isocrystal $(D,\tau_D)$ satisfies $\ulCE(D,\tau_D)=\ulCF_{_{\SC -1,r}}$, for example if $(D,\tau_D)$ is the crystalline realization of a Drinfeld module over $k$, then every weakly admissible Hodge-Pink lattice $\Fq_D$ on $(D,\tau_D)$ is already admissible. This can be shown in the same way as \cite[Theorem~9.3(a)]{HartlRZ}.
\end{example}

\begin{example}\label{ExWANotA}
We show that ``weakly admissible does not imply admissible'' if the field $K$ is algebraically closed (actually perfect suffices) and complete. Let $D=k\dpl z\dpr^{\oplus2},\,\tau_D= \left(\begin{smallmatrix} 0 & 1 \\ z^{-3} & 0 \end{smallmatrix}\right)$. We search a Hodge-Pink lattice $\Fq_D\subset\Fp_D$ with $\dim_K\Fp_D/\Fq_D=3$ for which $\ulD=(D,\tau_D,\Fq_D)$ is not admissible. This means the Hodge-Pink weights are $(2,1)$ or $(3,0)$. Note that any such $\ulD$ is weakly admissible, because $t_N(\ulD)=-3=t_H(\ulD)$ and there are no proper subobjects $\ulD'\subset\ulD$, as the $z$-isocrystal $(D,\tau_D)$ is simple.

We find $\Fq_D$ by imposing the condition that there is a $\hat\sigma$-sub-bundle $\ulCF_{_{\SC 1,1}}\subset\ulCF(\ulD)\subset\ulCE(\ulD)=\ulCF_{_{\SC 3,2}}$. Then $\ulCF(\ulD)\not\cong(\ulCF_{_{\SC 0,1}})^{\oplus2}$ because there is no non-zero homomorphism $\ulCF_{_{\SC 1,1}}\to\ulCF_{_{\SC 0,1}}$ by Theorem~\ref{ThmHartlPink1}\ref{ThmHartlPink1b}. Therefore Theorem~\ref{ThmCritWA} will imply that $\ulD$ is not admissible. One easily computes
\begin{equation}\label{EqHomCF}
\Hom\bigl(\ulCF_{_{\SC 1,1}},\ulCE(\ulD)\bigr) \es = \es \Bigl\{\,f=\sum_{i\in\BZ}z^i\cdot\left(\begin{smallmatrix}\TS u^{\hat q^{-2i}} \\[1mm] \TS u^{\hat q^{-2i-3}} \end{smallmatrix}\right)\colon u\in K,\,|u|<1\,\Bigr\}\,,
\end{equation}
because $f=\left(\begin{smallmatrix} f_1 \\ f_2 \end{smallmatrix}\right)$ must satisfy the equation $\tau_D\cdot\hat\sigma(f)=f\cdot z^{-1}$. This implies $\hat\sigma(f_2)=z^{-1}f_1$ and $z^{-3}\hat\sigma(f_1)=z^{-1}f_2$, whence $f_1=z^{-1}\hat\sigma^2(f_1)$. Writing $f_1=\sum_{i\in\BZ}b_iz^i$, we must have $b_i=(b_{i+1})^{\hat q^2}$. So for $b_0=u\in K$ we obtain $f_1=\sum_{i\in\BZ}u^{\hat q^{-2i}}z^i$. Now the convergence condition for $f\in K\ancon$ implies $|u|<1$.

Thus, if we take any $u\in K$ with $0<|u|<1$ and the corresponding $f\ne0$ from \eqref{EqHomCF} and $\Fq_D:=(z-\zeta)^3\Fp_D+K\dbl z-\zeta\dbr\cdot f$ the homomorphism $f\colon\ulCF_{_{\SC 1,1}}\to\ulCE(\ulD)$ factors through $\ulCF(\ulD)\subset\ulCE(\ulD)$, as can be seen from \eqref{EqCFFromD}. This implies that $\ulD$ is not admissible.
\end{example}

On the positive side there is the following

\begin{theorem}\label{ThmWA=>A}
If the field $K$ satisfies the following condition
\begin{enumerate}
\item \label{ThmWA=>A_a}
\cite[Th\'eor\`eme~7.3]{GL11}: $K$ is discretely valued, or
\item \label{ThmWA=>A_b}
\cite[Theorem~2.5.3]{HartlPSp}: the value group of $K$ is finitely generated, or
\item \label{ThmWA=>A_c}
\cite[Theorem~2.5.3]{HartlPSp}: the value group of $K$ does not contain an element $x\ne1$ such that $x^{\hat q^{-n}}\in|K^{^{\SSC\times}}|$ for all $n\in\BN_0$, or
\item \label{ThmWA=>A_d}
\cite[Theorem~2.5.3]{HartlPSp}: the completion $\wh{K\!\cdot\!k^\alg}$ of the compositum $K\!\cdot\!k^\alg$ inside $K^\alg$ does not contain an element $a$ with $0<|a|<1$ such that $a^{\hat q^{-n}}\in\wh{K\!\cdot\!k^\alg}$ for all $n\in\BN_0$,
\end{enumerate}
then every weakly admissible $z$-isocrystal with Hodge-Pink structure over $R$ is admissible.
\end{theorem}

\begin{proof}
Clearly the conditions satisfy \ref{ThmWA=>A_a}$\Longrightarrow$\ref{ThmWA=>A_b}$\Longrightarrow$\ref{ThmWA=>A_c}$\Longrightarrow$\ref{ThmWA=>A_d}, because the value groups of $K$ and $\wh{K\!\cdot\!k^\alg}$ coincide. Under condition \ref{ThmWA=>A_d} the theorem was proved in \cite[Theorem~2.5.3]{HartlPSp}. In terms of Example~\ref{ExWANotA} the idea is that in this case one cannot have $f\in\Fq_D$.
\end{proof}

\begin{remark}\label{RemAnalog}
This theorem is the analog of the Theorem of Colmez and Fontaine~\cite[Th\'eor\`eme~A]{CF}, which states that in the theory of $p$-adic Galois representations every weakly admissible filtered isocrystal $\ulD=(D,\tau_D,F^\bullet D_L)$ over a discretely valued extension $L$ of $\BQ_p$ with perfect residue field, comes from a representation of $\Gal(L^\alg/L)$ on a finite dimensional $\BQ_p$-vector space; compare Remark~\ref{RemCompIsom}. In the special case where the Hodge filtration satisfies $F^0D_L=D_L\supset F^1D_L\supset F^2D_L=(0)$ and all Newton slopes of the isocrystal $(D,\tau_D)$ lie in the interval $[0,1]$ it was proved by Kisin~\cite[Theorem 0.3]{Kisin06} and Breuil~\cite[Theorem 1.4]{Breuil00} that $\ulD$ is (weakly) admissible if and only if $\ulD$ comes from a $p$-divisible group. Also in this situation there are weakly admissible filtered isocrystals $\ulD$ over $\BC_p$ which do not come from a $p$-divisible group; see \cite[Example~6.7]{HartlRZ}.
\end{remark}

\begin{remark}\label{RemEqCharCrystRep}
Note that the Theorem of Colmez and Fontaine~\cite[Th\'eor\`eme~A]{CF} from the previous remark actually says that a continuous representation of $\Gal(L^\sep/L)$ in a finite dimensional $\BQ_p$-vector space is crystalline if and only if it isomorphic to $F^0(\ulD\otimes_{L_0}\wt\bB_\rig)^{\Frob\,=\id}$ for a weakly admissible filtered isocrystal $\ulD=(D,\tau_D,F^\bullet D_L)$ over $L$. In the function field case, when $K$ is discretely valued, we could therefore define the \emph{category of equal characteristic crystalline representations of $\Gal(K^\sep/K)$} as the essential image of the functor 
\begin{equation}\label{EqVcris}
\ulD=(D,\tau_D,\Fq_D)\;\longmapsto\; \bigl(\ulD\otimes_{k\dpl z\dpr}\CO_\olK\dbl z,z^{-1}\}[\tplusminus^{-1}]\bigr)^{\hat\tau=\id}\,\cap\,\Fq_D\otimes_{K\dbl z-\zeta\dbr}\olK\dbl z-\zeta\dbr\
\end{equation}
from weakly admissible $z$-isocrystals with Hodge-Pink structure $\ulD$ to continuous representations of $\Gal(K^\sep/K)$ in finite dimensional $Q_\epsilon$-vector spaces.
By Theorems~\ref{ThmWA=>A}, \ref{ThmCompEpsCris} and \ref{ThmAndersonKim} and Proposition~\ref{PropBHfullyFaithful} this functor is fully faithful and this definition coincides with our Definition~\ref{DefCrystRep} above.
\end{remark}

\begin{remark}\label{RemFFCurve}
The functor \eqref{EqVcris} can also be described via the associated $\hat\sigma$-bundles. Namely if $\ulD$ is admissible, that is arises from a (rigidified) local shtuka $\ulHM$ over $R$ with $\ulD=\BH(\ulHM)$, then $\ulCF(\ulD)=\ulCF(\ulHM)=\ulHM\otimes_{R\dbl z\dbr}K\ancon$, and hence 
\[
\bigl(\ulCF(\ulD)\otimes_{K\ancon}\olK\ancon\bigr)^{\hat\tau} \; = \; \bigl(\ulHM\otimes_{R\dbl z\dbr}\olK\ancon\bigr)^{\hat\tau} \; = \; \bigl(\check T_\epsilon\ulHM\otimes_{A_\epsilon}\ulCF_{_{\SC 0,1}}\bigr)^{\hat\tau} \; = \; \check V_\epsilon\ulHM\,;
\]
see the proof of Theorem~\ref{ThmCritWA}\ref{ThmCritWA_A}. In the analogous situation of $p$-adic Galois representations, mentioned in the previous remarks, there is a similar description due to Fargues and Fontaine~\cite{FarguesFontaine}. Namely, consider the punctured open unit disc $\{0<|z|<1\}$ over $\olK$. Then every $\hat\sigma$-bundle on $\{0<|z|\le|\zeta|\}$ extends canonically to a $\hat\sigma$-bundle over $\{0<|z|<1\}$; see \cite[Proposition~1.4.1(b)]{HartlPSp}. The quotient $\{0<|z|<1\}/\hat\sigma^\BZ$ of $\{0<|z|<1\}$ modulo the action of $\hat\sigma$ exists in the category of Huber's~\cite{Huber} adic spaces and the category of vector bundles on $\{0<|z|<1\}/\hat\sigma^\BZ$ is equivalent to the category of $\hat\sigma$-bundles over $\{0<|z|<1\}$. Under this equivalence the $Q_\epsilon$-vector space of global sections of the vector bundle on $\{0<|z|<1\}/\hat\sigma^\BZ$ equals the $\hat\tau$-invariants of the associated $\hat\sigma$-bundle. 

In the theory of $p$-adic Galois representations the analog of the quotient $\{0<|z|<1\}/\hat\sigma^\BZ$ over $\BQ_p$ is the adic curve $X_{F,E}^\ad$ constructed by Fargues~\cite{Fargues14}. Every weakly admissible filtered isocrystal $\ulD=(D,\tau_D,F^\bullet D_L)$ over $L$ gives rise to two vector bundles $\CE^\infty$ and $\CE_\infty$ on $X_{F,E}^\ad$ which agree outside the point corresponding to $\bB_\dR^+$. These are analogous to the vector bundles on $\{0<|z|<1\}/\hat\sigma^\BZ$ corresponding to our $\hat\sigma$-bundles $\ulCE(\ulD)$ and $\ulCF(\ulD)$, which agree outside the image of the point $z=\zeta$ on $\{0<|z|<1\}/\hat\sigma^\BZ$ corresponding to our analog $\olK\dbl z-\zeta\dbr$ of $\bB_\dR^+$. By \cite[Th\'eor\`emes~4.43 and 3.5]{Fargues14} the $p$-adic Galois representation associated with $\ulD$ equals the global sections of $\CE_\infty$ over $X_{F,E}^\ad$.
\end{remark}

%
%

\section{Torsion local shtukas and torsion Galois representations}\label{SectTorLocSh}
\setcounter{equation}{0}

As a preparation for the equi-characteristic deformation theory, which will be discussed in Section~\ref{SectDefoTh}, we need a ``torsion version'' of equi-characteristic Fontaine's theory -- or rather, a suitable function field analog of finite flat group schemes of $p$-power order. Let $R$ be an arbitrary valuation ring as in Notation~\ref{Notation}. From Lemma~\ref{LemmaAuslanderBuchsbaum} onwards we will assume that $R$ is discretely valued. The analogy between $p$-divisible groups and local shtukas suggests the following

\begin{definition}\label{DefTorLocSht}
A  \emph{torsion local shtuka} (over $R$) is a pair $\ulHM=(\hat M,\tau_{\hat M})$ consisting of a finitely presented $R\dbl z\dbr$-module $\hat M$ which is $z$-power torsion and free (necessarily of finite rank) as an $R$-module, and an isomorphism $\tau_{\hat M}\colon\hat\sigma^\ast\hat M[\frac{1}{z-\zeta}] \isoto\hat M[\frac{1}{z-\zeta}]$. If $\tau_{\hat M}(\hat\sigma^*\hat M)\subset\hat M$ then $\ulHM$ is called \emph{effective}, and if $(z-\zeta)^d\hat M\subset\tau_{\hat M}(\hat\sigma^*\hat M)\subset\hat M$ then we say that $\ulHM$ is  \emph{of height $\leqslant d$}. If $\tau_{\hat M}(\hat\sigma^*\hat M)=\hat M$ then $\ulHM$ is called \emph{\'etale}.

A \emph{morphism} of torsion local shtukas $f\colon(\hat M,\tau_{\hat M})\to(\hat M',\tau_{\hat M'})$ over $R$ is a morphism of $R\dbl z\dbr$-modules $f\colon\hat M\to\hat M'$ which satisfies $\tau_{\hat M'}\circ \hat\sigma^\ast f = f\circ \tau_{\hat M}$.
\end{definition}

\begin{remark}\label{RemDVR}
(a) Since $\zeta$ is $\hat M$-regular and $\hat M$ is $z$-power torsion, $\hat M$ injects into $\hat M[\frac{1}{z-\zeta}]$. 

\smallskip\noindent
(b) Since $\hat M$ is finite free over $R$ the pair $(\hat M, \tau_{\hat M})$ with its $R\dbl z\dbr$-module structure can be defined in terms of finitely many parameters over $R$. Therefore it is defined over $R'\dbl z\dbr$ for some finitely generated and hence discretely valued $A_\epsilon$-subalgebra $R'\subset R$. So we may assume that $R$ is discretely valued. 

\smallskip\noindent
(c) The definition of torsion local shtukas over  $S\in\Nilp_{A_\epsilon}$ is not completely straightforward, because  we have $\hat M[\frac{1}{z-\zeta}] = 0$ for any finitely generated $\CO_S\dbl z \dbr$-module $\hat M$ killed by some power of $z$. Although it is possible to define torsion local shtukas over $S\in\Nilp_{A_\epsilon}$ (cf.~Remark~\ref{RemTorAndersonMod}), we will only consider torsion local shtukas over $R$. 
\end{remark}

The following example (together with Lemma~ \ref{LemmaRaynaudBBM}) explains why we can regard  torsion local shtukas over $R$ as the function field analog of finite flat group schemes of $p$-power order.
\begin{example}\label{ExCokerIsog}
Let $f\colon\ulHM_1 \rightarrow \ulHM_0$ be an isogeny of local shtukas over $R$; i.e., a morphism which is also a quasi-isogeny. Let $\hat M:=\coker f$. Since $f$ is a morphism of local shtukas, $\tau_{\hat M_0}$ induces $\tau_{\hat M}:\hat\sigma^*\hat M[\frac{1}{z-\zeta}]\isoto\hat M[\frac{1}{z-\zeta}]$. We claim that $\ulHM:=(\hat M,\tau_{\hat M})$ is a torsion local shtuka over $R$, which is effective (respectively, of height $\leqslant d$) if  $\ulHM_0$ is effective (respectively, if $(z-\zeta)^d\hat M_0 \subset \tau_{\hat M_0}(\hat\sigma^*\hat M_0)\subset\hat M_0$).

Since $R\dbl z\dbr\subset R\dbl z\dbr[\tfrac{1}{z}]$ the map $f$ is injective and there is an $n$ such that $z^n\hat M_0\subset f(\hat M_1)$. In particular $\hat M$ is killed by $z^n$. Tensoring with the residue field $k$ over $R$ yields an exact sequence
\[
0\;\longto\;\Tor_1^{R}(k,\hat M)\;\longto\;\hat M_1\otimes_{R\dbl z\dbr}k\dbl z\dbr\;\xrightarrow{\; f\otimes\id_k\;}\;\hat M_0\otimes_{R\dbl z\dbr}k\dbl z\dbr\;\longto\;\hat M\otimes_R k\;\longto\;0\,.
\]
As $\hat M_1$ and $\hat M_0$ are free we have $\hat M_i\otimes_{R\dbl z\dbr}k\dbl z\dbr\cong k\dbl z\dbr^{\oplus\rk\ulHM_i}$ for $i=0,1$ with $\rk\ulHM_0=\rk\ulHM_1$. Since $\hat M\otimes_R k$ is killed by $z^n$ the map $f\otimes\id_k$ is injective by the elementary divisor theorem. So $\Tor_1^{R}(k,\hat M)=(0)$ and since $\hat M\,=\,\coker\bigl(f\colon\hat M_1/z^n\hat M_1 \to \hat M_0/z^n\hat M_0\bigr)$ is finitely presented over $R$ it is free of finite rank by Nakayama's Lemma; e.g.\ \cite[Exercise~6.2]{Eisenbud}.
 \end{example}

From now on we assume that $R$ is discretely valued. We will need the following

\begin{lemma}\label{LemmaAuslanderBuchsbaum}
Assume that $R$ is discretely valued, and let $\hat M$ be a finitely generated module over $R\dbl z \dbr/(z^n)$ for some $n$. Then the following are equivalent
\begin{enumerate}
\item
$\hat M$ is flat over $R$;
\item
The kernel of any surjective map $\hat M_0\onto \hat M$, where $\hat M_0$ is  a free $R\dbl z\dbr$-module of finite rank, is free over $R\dbl z \dbr$.
\end{enumerate}
\end{lemma}
\begin{proof}
Since $R\dbl z\dbr$ is regular, whence Cohen-Macaulay, the theorem of Auslander and Buchsbaum \cite[Theorem~19.9]{Eisenbud} tells us that
\[
\mathrm{proj.dim}_{R\dbl z\dbr}\hat M + \mathrm{depth}_{R\dbl z\dbr}\hat M \;=\; \mathrm{depth}(R\dbl z\dbr) \;=\; \mathrm{dim}(R\dbl z\dbr) \;=\; 2
\]
for any finitely generated $R\dbl z\dbr$-module $\hat M$. If $\hat M$ is flat over $R$, then $\zeta$ is $\hat M$-regular so the depth of $\hat M$ is at least $1$. Therefore, the projective dimension of $\hat M$ is at most $1$; in other words, $\hat M$ admits a minimal projective resolution consisting of two terms, say $\hat F_1 \to \hat F_0$. If $\hat M_0\onto \hat M$ is an epimorphism for a free $R\dbl z\dbr$-module $\hat M_0$ then $\hat M_0\cong\hat F_0\oplus\hat N_0$ for a free $R\dbl z\dbr$-module $\hat N_0$ and $\ker(\hat M_0\onto \hat M)\cong\hat F_1\oplus\hat N_0$ by \cite[Theorem~20.2 and Lemma~20.1]{Eisenbud}.

Conversely, assume that $\hat M$ is not flat over $R$ and killed by $z^n$. Then the $\Fm_R$-torsion of $\hat M$ is of finite positive length since it is non-zero and killed by some powers of $z$ and $\zeta$, so the depth of $\hat M$ as a $R\dbl z \dbr$-module is zero. So the projective dimension of $\hat M$ is $2$; in other words, any projective resolution of $\hat M$ consists of at least three terms. 
\end{proof}

The following lemma states that any torsion local shtuka over a discretely valued $R$ arises from the construction in Example~\ref{ExCokerIsog}. This is analogous to  Raynaud's theorem \cite[3.1.1]{BBM82}, which states that any finite flat group scheme can be embedded, Zariski-locally on the base, in some abelian scheme.
\begin{lemma}\label{LemmaRaynaudBBM}
 Let $\ulHM$ be a torsion local shtuka over $R$. Then there exist local shtukas $\ulHM_0$ and $\ulHM_1$, and an isogeny $f\colon\ulHM_1\rightarrow\ulHM_0$ such that we have a short exact sequence
 \[
 0 \;\longto\; \ulHM_1\xrightarrow{\es f\;}\ulHM_0 \;\longto\; \ulHM \;\longto\; 0\,,
 \]
where each arrow is $\hat\tau$-equivariant. If $\ulHM$ is effective of height $\leqslant d$ (i.e., it satisfies  $(z-\zeta)^d\hat M \subset \tau_{\hat M}(\hat\sigma^*\hat M) \subset \hat M$), then we may take $\ulHM_0$ and $\ulHM_1$ to satisfy $(z-\zeta)^d\hat M_i \subset \tau_{\hat M_i}(\hat\sigma^*\hat M_i) \subset \hat M_i$.
\end{lemma}

\begin{proof}
The proof is analogous to the proof of \cite[2.3.4]{Kisin06}. By replacing $\tau_{\hat M}$ with $(z-\zeta)^n\tau_{\hat M}$ for some $n$, we may assume that $\ulHM$ satisfies $(z-\zeta)^d\hat M \subset \tau_{\hat M}(\hat\sigma^*\hat M) \subset \hat M$. 
We now choose a finite free $R\dbl z\dbr/(z-\zeta)^d$-module $L$ which surjects onto $\hat M/\tau_{\hat M}(\hat\sigma^*\hat M)$, and a finite free $R\dbl z\dbr$-module $\hat M_0$ which fits into the following diagram
\[
\xymatrix{
0\ar[r]& \ker(\hat M_0\onto L)\ar[d]\ar[r] &\hat M_0 \ar@{->>}[d] \ar[r] & L\ar@{->>}[d] \ar[r] & 0\\
0\ar[r]&\hat\sigma^*\hat M \ar[r]^-{\tau_{\hat M}}& \hat M  \ar[r] & \hat M/\tau_{\hat M}(\hat\sigma^*\hat M) \ar[r]&0
}.
\]
Furthermore, by enlarging $\hat M_0$ we may arrange so that the left vertical map is also surjective; this can be done by replacing $\hat M_0$ with $\hat M_0\oplus\hat M'$, where $\hat M'$ is a finite free $R\dbl z\dbr$-module which surjects onto the kernel of $L\onto \hat M/\tau_{\hat M}(\hat\sigma^*\hat M)$ and maps to zero in $\hat M$.

Applying Lemma~\ref{LemmaAuslanderBuchsbaum} to $\hat M_0\onto L$, $\ker(\hat M_0\onto L)$ is free over $R\dbl z\dbr$. We decompose $\ker(\hat M_0\onto L) = \hat N\oplus \hat N'$ into two free modules so that $\hat N\otimes_{R\dbl z\dbr} k = (\sigma^*\hat M)\otimes_{R\dbl z\dbr} k$, and hence $\hat N\onto\sigma^*\hat M$ is surjective by Nakayama. Then we may lift  $\sigma^*\hat M_0 \onto \sigma^* \hat M$ to a map $\sigma^*\hat M_0 \to \hat N$ which is also surjective by Nakayama and therefore a direct summand. We can now lift the latter map to an isomorphism  $\sigma^*\hat M_0 \xrightarrow\sim \hat N \oplus \hat N' = \ker(\hat M_0\onto L)$, as both modules are free of same finite rank over $R\dbl z\dbr$.
Therefore, we obtain a map
\[\tau_{\hat M_0}\colon \hat\sigma^*\hat M_0 \isoto \ker(\hat M_0\onto L)\hookrightarrow \hat M_0\]
lifting $\tau_{\hat M}\colon\hat\sigma^*\hat M \isoto \tau_{\hat M}(\hat\sigma^*\hat M)$. Clearly, $\ulHM_0:=(\hat M_0,\tau_{\hat M_0})$ is an effective local shtuka over $R$ with $(z-\zeta)^d\hat M_0 \subset \tau_{\hat M_0}(\hat\sigma^*\hat M_0) \subset \hat M_0$. Applying Lemma~\ref{LemmaAuslanderBuchsbaum} to $\hat M_0\onto \hat M$, $\hat M_1:=\ker (\hat M_0\onto \hat M)$ is free over $R\dbl z \dbr$ so $\ulHM_1:=(\hat M_1,\tau_{\hat M_1}:= \tau_{\hat M_0}|_{\hat M_1[\frac{1}{z-\zeta}]})$ is an effective local shtuka with $(z-\zeta)^d\hat M_1 \subset \tau_{\hat M_1}(\hat\sigma^*\hat M_1) \subset \hat M_1$. Let $f\colon \ulHM_1\rightarrow\ulHM_0$ be the morphism induced by the natural inclusion, then $\ulHM$ is  the cokernel of $f$.
\end{proof}

\begin{remark}\label{RemTorAndersonMod}
We have seen that the definition of torsion local shtukas over $R$ does not straightforwardly apply to base schemes $S\in\Nilp_{A_\epsilon}$. On the other hand, Lemma~\ref{LemmaRaynaudBBM} and the anti-equivalence between effective local shtukas and $z$-divisible local Anderson modules (Remark~\ref{RemEquiv}) suggest the following notion as the function filed analog of finite flat group schemes of $p$-power order.

We define a \emph{torsion local Anderson module} over $S\in\Nilp_{A_\epsilon}$ to be a finite locally free $A_{\epsilon}$-module scheme $H$ over $S$ such that for some Zariski covering $\{S_\alpha\}$ of $S$ there exist a $z$-divisible local Anderson module $G_\alpha$ over $S_\alpha$ and a monomorphism $H_{S_\alpha}\hookrightarrow G_\alpha$ of fppf sheaves of $A_\epsilon$-modules for each $\alpha$. 

Note that the anti-equivalence of categories $\Dr_{\hat q}$ in Remark~\ref{RemEquiv} can be extended to bounded effective local shtukas over $R$ by limit as in Remark~\ref{RemRelationBetweenTheTwoS}. If $\hat M$ is a torsion local shtuka over $R$ obtained as the cokernel of an isogeny $f\colon\ulHM_1\rightarrow\ulHM_0$ of effective local shtukas, then we set 
\[\Dr_{\hat q}(\ulHM)\;:=\;\ker(\Dr_{\hat q}f\colon \Dr_{\hat q}(\ulHM_0) \rightarrow \Dr_{\hat q}(\ulHM_1)).\]
Then $\Dr_{\hat q}(\ulHM)$ is independent of the choice of $f$ up to isomorphism, and this induces an anti-equivalence of categories between effective torsion local shtukas and torsion local Anderson modules over $R$. Indeed, it is possible to define $\Dr_{\hat q}$ directly without choosing $f\colon\ulHM_1\rightarrow\ulHM_0$ as
\[
\Dr_{\hat q}(\ulHM)\;:=\;\Spec\;(\Sym_R\hat M)\big/\bigl(m^{\otimes\hat q}-\tau_{\hat M}(\hat\sigma_{\hat M}^*m)\colon m\in \hat M\bigr)\,;
\]
like in \eqref{EqDrinfeldGr} in Section~\ref{SectDivLocAM}. And if $H = \Dr_{\hat q}(\ulHM)$ for some effective torsion local shtuka $\ulHM$ over $R$, then we can recover $\ulHM$ as $\Hom_{R\text{\rm-groups},\BF_\epsilon\text{\rm-lin}}(H\,,\,\BG_{a,R})$ by \cite[Theorem~2]{Abrashkin} or \cite[Theorem~5.2]{HartlSingh}. 
\end{remark}

To a torsion local shtuka over $R$, we associate a torsion Galois representation as follows.

\begin{definition}\label{DefTorTateMod}
Let $\ulHM$ be a torsion local shtuka over $R$. We define the (\emph{dual}) \emph{Tate module} of $\ulHM$ to be
\[
\check T_\epsilon\ulHM\;:=\;(\hat M\otimes_{R\dbl z\dbr}K^\sep\dbl z\dbr)^{\hat\tau}\;:=\;\bigl\{m\in\hat M\otimes_{R\dbl z\dbr}K^\sep\dbl z\dbr\colon \tau_{\hat M}(\hat\sigma_{\!\hat M}^*m)=m\bigr\}.
\]
\end{definition}

\begin{proposition}\label{PropTorTateMod}
For a torsion local shtuka $\ulHM$ over $R$,
$\check T_\epsilon\ulHM$ is a torsion $A_\epsilon$-module of length equal to $\rk_R (\hat M)$, which carries a discrete action of $\Gal(K^\sep/K)$. Moreover the inclusion $\check T_\epsilon\ulHM\subset\hat M\otimes_{R\dbl z\dbr}K^\sep\dbl z\dbr$ defines a canonical isomorphism of $K^\sep\dbl z\dbr$-modules 
\begin{equation}\label{EqTorTateModIsom}
\check T_\epsilon\ulHM\otimes_{A_\epsilon}K^\sep\dbl z\dbr\isoto\ulHM\otimes_{R\dbl z\dbr}K^\sep\dbl z\dbr
\end{equation}
which is $\Gal(K^\sep/K)$- and $\hat\tau$-equivariant, where on the left module $\Gal(K^\sep/K)$ acts on both factors and $\hat\tau$ is $\id\otimes\hat\sigma$, and on the right module $\Gal(K^\sep/K)$ acts only on $K^\sep\dbl z\dbr$ and $\hat\tau$ is $(\tau_{\hat M}\circ\hat\sigma_{\!\hat M}^*)\otimes \hat\sigma$. In particular one can recover $\ulHM\otimes_{R\dbl z\dbr}K\dbl z\dbr=\bigl(\check T_\epsilon\ulHM\otimes_{A_\epsilon}K^\sep\dbl z\dbr\bigr)^{\Gal(K^\sep/K)}$ as the Galois invariants.

If we have a $\hat\tau$-equivariant short exact sequence $0\rightarrow \ulHM'\rightarrow \ulHM\rightarrow \ulHM'' \rightarrow 0$ where each term is either a local shtuka or a torsion local shtuka over $R$, then we have a $\Gal(K^\sep/K)$-equivariant short exact sequence
\[0\,\longto\, \check T_\epsilon\ulHM' \,\longto\,\check T_\epsilon\ulHM \,\longto\, \check T_\epsilon\ulHM''\,\longto\, 0.
\]
\end{proposition}
\begin{proof}
Let $\ulHM$ be a torsion local shtuka over $R$.
To see that $\check T_\epsilon \ulHM$ carries a natural discrete action of $\Gal(K^\sep/K)$, note that $\hat M$ is killed by some power of $z$, say by $z^n$, so the natural $\Gal(K^\sep/K)$-action on 
\[
\hat M \otimes_{R\dbl z\dbr} K^\sep\dbl z \dbr =\hat M \otimes_{R\dbl z\dbr} K^\sep\dbl z \dbr/(z^n) 
\]
is discrete and commutes with $\tau_{\hat M}$.

Let us consider a $\hat\tau$-equivariant short exact sequence $0\rightarrow \ulHM_1\xrightarrow f\ulHM_0\rightarrow \ulHM \rightarrow0$, where $\ulHM$ is a torsion local shtuka over $R$ and $\ulHM_0$ and $\ulHM_1$ are local shtukas over $R$. We obtain a $\Gal(K^\sep/K)$-equivariant sequence
\begin{equation}\label{EqPropTorTateMod}
0 \,\longto\, \check T_\epsilon\ulHM_1 \xrightarrow{\,\check T_\epsilon f} \check T_\epsilon\ulHM_0 \,\longto\, \check T_\epsilon\ulHM\,\longto\,0\,,
\end{equation}
which by definition is exact on the left and in the middle. To see that it is also exact on the right, we obtain from Proposition~\ref{PropTateMod}
\[
\coker(\check T_\epsilon f)\otimes_{A_\epsilon}K^\sep\dbl z\dbr \isoto \hat M\otimes_{R\dbl z\dbr}K^\sep\dbl z\dbr \hookleftarrow \check T_\epsilon\ulHM\otimes_{A_\epsilon}K^\sep\dbl z\dbr.
\]
Combining these, it follows that the natural inclusion induces a $\Gal(K^\sep/K)$- and $\hat\tau$-equivariant isomorphism
\[
\check T_\epsilon\ulHM \otimes_{A_\epsilon}K^\sep\dbl z\dbr \isoto \hat M\otimes_{R\dbl z\dbr}K^\sep\dbl z\dbr.      
\]
Taking $\hat\tau$-invariants yields $\check T_\epsilon\ulHM=\coker(\check T_\epsilon f)$ and so the sequence \eqref{EqPropTorTateMod} is also exact on the right. The claim on the $A_\epsilon$-length of $\check T_\epsilon\ulHM$ follows from this isomorphism.

It now remains to show that $\check T_\epsilon$ on the category of torsion local shtukas over $R$ is exact. Indeed, it is clearly left exact by definition, and the right exactness follows from the length consideration.
\end{proof}

Recall from Remark~\ref{RemTorAndersonMod} that to an effective torsion local shtuka $\ulHM$ over $R$ we associated a torsion local Anderson module $H:=\Dr_{\hat q}(\ulHM)$ over $R$. Hence we have a finite-length $A_\epsilon$-module $H(K^\sep)$ equipped with a discrete $\Gal(K^\sep/K)$-action. We will now discuss the comparison between $H(K^\sep)$ and $\check T_\epsilon\ulHM$ in a way that is analogous to and compatible with Proposition~\ref{PropCompTateMod}. For this, we need some preparations.

\begin{lemma}\label{LemmaTorTateModES}
Let $H$ be the kernel of an isogeny $G_0\rightarrow G_1$ of $z$-divisible local Anderson modules over $R$. Then there exists a natural Galois-equivariant short exact sequence
 \[
0\,\longto\, T_\epsilon G_0\,\longto\, T_\epsilon G_1 \,\longto\, H(K^\sep) \,\longto\,0\,.
 \]
 The surjective map on the right can be defined as follows. Given $\xi\colon Q_\epsilon/A_\epsilon\rightarrow G_1(K^\sep)$ (which is an element of $T_\epsilon G_1$), choose a lift $\tilde\xi:Q_\epsilon\rightarrow G_0(K^\sep)$. Then we have $\tilde\xi(1)\in H(K^\sep)$, and it only depends on $\xi$ (not on the choice of $\tilde\xi$). The surjective map on the right is given by $\xi\mapsto\tilde\xi(1)$.
\end{lemma}
\begin{proof}
For any $n$, consider the exact sequence $0\rightarrow A_\epsilon \rightarrow \epsilon^{-n}A_\epsilon \rightarrow \epsilon^{-n}A_\epsilon/A_\epsilon \rightarrow 0$, where the first two terms are projective over $A_\epsilon$. By applying to it the long exact sequence for $\Hom_{A_\epsilon}\bigl(\fdot,H(K^\sep)\bigr)$, we obtain the following natural Galois-equivariant isomorphism 
\begin{equation}\label{EqExtH}
\xymatrix @C=-6pc {
*!R(0.6)
\objectbox{\Ext^1_{A_\epsilon}\bigl(\epsilon^{-n}A_\epsilon/A_\epsilon,\,H(K^\sep)\bigr)\cong 
\coker\bigl[
\Hom_{A_\epsilon}\bigl(\epsilon^{-n}A_\epsilon,H(K^\sep)\bigr)} \ar[r] \ar[d]_{\xi\mapsto\xi(z^{-n})}^\cong &
\Hom_{A_\epsilon}\bigl(A_\epsilon,H(K^\sep)\bigr) \ar[d]^{\xi\mapsto\xi(1)}_\cong
\bigr]\\
**{!R(0.6) =<12pc,2pc>}
\objectbox{\cong\coker\bigl[\qquad\qquad H(K^\sep)} \ar[r]^{[z^n]} & **{!L(0.6)
 =<8pc,2pc>} \objectbox{H(K^\sep)\qquad\bigr]},
}
\end{equation}
where we let $\Gal(K^\sep/K)$ act trivially on $\epsilon^{-n}A_\epsilon/A_\epsilon$. (Indeed, this isomorphism is independent of the choice of $z$.) In particular, if we choose $n$ so that $z^n$ kills $H$, then we have a natural isomorphism $\Ext^1_{A_\epsilon}\bigl(\epsilon^{-n}A_\epsilon/A_\epsilon,\,H(K^\sep)\bigr) \xrightarrow\sim H(K^\sep)$, and the natural inclusion $\epsilon^{-n}A_\epsilon/A_\epsilon \hookrightarrow \epsilon^{-n-1}A_\epsilon/A_\epsilon$ induces an isomorphism $\Ext^1_{A_\epsilon}\bigl(\epsilon^{-n-1}A_\epsilon/A_\epsilon,\,H(K^\sep)\bigr)\xrightarrow\sim \Ext^1_{A_\epsilon}\bigl(\epsilon^{-n}A_\epsilon/A_\epsilon,\,H(K^\sep)\bigr)$, which induces the identity map on $H(K^\sep)$.

Recall that we have the following Galois-equivariant short exact sequence of $A_\epsilon$-modules 
\[
0\,\longto\, H(K^\sep)\,\longto\, G_0(K^\sep)\,\longto\, G_1(K^\sep)\,\longto\,0\,.
\]
If $z^n$ kills $H$, then we have the following natural Galois-equivariant exact sequence via the $\Hom$-$\Ext$ long exact sequence (independent of the choice of $z$):
 \begin{multline}\label{EqTorTateModES}
0\,\rightarrow\, \Hom_{A_\epsilon}\bigl(\epsilon^{-n}A_\epsilon/A_\epsilon, H(K^\sep)\bigr)\,\rightarrow\,  \Hom_{A_\epsilon}\bigl(\epsilon^{-n}A_\epsilon/A_\epsilon, G_0(K^\sep)\bigr)\\
\xrightarrow{\es g\;}\,  \Hom_{A_\epsilon}\bigl(\epsilon^{-n}A_\epsilon/A_\epsilon,G_1(K^\sep)\bigr)\,\rightarrow\, \Ext^1_{A_\epsilon}\bigl(\epsilon^{-n}A_\epsilon/A_\epsilon,\,H(K^\sep)\bigr)\,\rightarrow\,0\, . 
 \end{multline}
The right exactness follows from the fact that $\Ext^1_{A_\epsilon}\bigl(\epsilon^{-n}A_\epsilon/A_\epsilon,\,G_i(K^\sep)\bigr)=0$ by the analog of \eqref{EqExtH} for $G_i$ using that $[z^n]\colon G_i(K^\sep)\to G_i(K^\sep)$ is surjective as $G_i$ is $z$-divisible. Now by taking the projective limit of the sequence \eqref{EqTorTateModES} as we increase $n$ and observing that both the kernel and the cokernel of $g$ satisfy the Mittag-Leffler condition, we obtain
\begin{multline*}
0\,\rightarrow\, \Hom_{A_\epsilon}\bigl(Q_\epsilon/A_\epsilon, H(K^\sep)\bigr)\,\rightarrow\,  \Hom_{A_\epsilon}\bigl(Q_\epsilon/A_\epsilon, G_0(K^\sep)\bigr)\\
\xrightarrow{\es g\;}\,  \Hom_{A_\epsilon}\bigl(Q_\epsilon/A_\epsilon,G_1(K^\sep)\bigr)\,\rightarrow\, H(K^\sep)\,\rightarrow\,0\,.
\end{multline*}
Since $\Hom_{A_\epsilon}\bigl(Q_\epsilon/A_\epsilon, H(K^\sep)\bigr) = 0$, we obtain the short exact sequence as in the statement.

It remains to explicitly describe the surjective map $T_\epsilon G_1 \rightarrow H(K^\sep)$. Choosing $n$ so that $z^n$ kills $H$, the map, by construction, factors as follows:
\begin{multline*}
 T_\epsilon G_1 =  \Hom_{A_\epsilon}\bigl(Q_\epsilon/A_\epsilon,G_1(K^\sep)\bigr) \rightarrow \Hom_{A_\epsilon}\bigl(\epsilon^{-n}A_\epsilon/A_\epsilon,G_1(K^\sep)\bigr) \\ 
 \rightarrow  \Ext^1_{A_\epsilon}\bigl(\epsilon^{-n}A_\epsilon/A_\epsilon,\,H(K^\sep)\bigr)\xrightarrow\sim H(K^\sep),
\end{multline*}
where the first arrow is the restriction map and the second arrow is the connecting homomorphism in \eqref{EqTorTateModES}. Keeping in mind the description of $\Ext^1_{A_\epsilon}\bigl(\epsilon^{-n}A_\epsilon/A_\epsilon,H(K^\sep)\bigr)$ in \eqref{EqExtH}, it is straightforward that the description given in the statement matches with the connecting homomorphism of \eqref{EqTorTateModES} given by the snake's lemma.
\end{proof}
\begin{remark}
One can prove Lemma~\ref{LemmaTorTateModES} by directly checking that the map $T_\epsilon G_1 \rightarrow H(K^\sep)$ is well defined and gives the desired short exact sequence. We instead appealed to the isomorphism $\Ext^1_{A_\epsilon}\bigl(\epsilon^{-n}A_\epsilon/A_\epsilon,\,H(K^\sep)\bigr) \xrightarrow\sim H(K^\sep)$ for the sake of conceptual clarity.
\end{remark}

In order to relate \eqref{EqPropTorTateMod} with the  exact sequence in Lemma~\ref{LemmaTorTateModES}, we need a little digression on (some variant of) Pontryagin duality. Consider a short exact sequence of $A_\epsilon$-modules 
\[
0\,\longto\, T_0\,\longto\, T_1 \,\longto\, T\,\longto\,0\,,
\]
where $T_0$ and $T_1$ are finitely generated free $A_\epsilon$-modules, and $T$ is of finite length. Then by applying to it the long exact sequence for $\Hom_{A_\epsilon}\bigl(\fdot, \widehat\Omega^1_{A_\epsilon/\BF_\epsilon}\bigr)$, we get
\begin{equation}\label{EqHomExt}
0\,\longto\, \Hom_{A_\epsilon}\bigl(T_1, \widehat\Omega^1_{A_\epsilon/\BF_\epsilon}\bigr)\,\longto\, \Hom_{A_\epsilon}\bigl(T_0, \widehat\Omega^1_{A_\epsilon/\BF_\epsilon}\bigr) \,\longto\, \Ext^1_{A_\epsilon}\bigl(T, \widehat\Omega^1_{A_\epsilon/\BF_\epsilon}\bigr)\,\longto\,0\,,
\end{equation}
since $\Hom_{A_\epsilon}\bigl(T, \widehat\Omega^1_{A_\epsilon/\BF_\epsilon}\bigr) = 0$ and $\Ext^1_{A_\epsilon}\bigl(T_i, \widehat\Omega^1_{A_\epsilon/\BF_\epsilon}\bigr) = 0$ for $i=0,1$ (as $T_i$ are projective). 

Now, consider the following exact sequence
\[
0\,\longto\,  \widehat\Omega^1_{A_\epsilon/\BF_\epsilon} \,\longto\,  \widehat\Omega^1_{A_\epsilon/\BF_\epsilon} \otimes_{A_\epsilon}Q_\epsilon  \,\longto\, \widehat\Omega^1_{A_\epsilon/\BF_\epsilon} \otimes_{A_\epsilon}Q_\epsilon/A_\epsilon\,\longto\,0\,.
\]
Since $\widehat\Omega^1_{A_\epsilon/\BF_\epsilon}$ is a rank-$1$ free module over $A_\epsilon$, the last two terms are injective  over $A_\epsilon$. Using that $T$ is torsion and $\widehat\Omega^1_{A_\epsilon/\BF_\epsilon} \otimes_{A_\epsilon}Q_\epsilon$ is torsion free, we get the following natural isomorphism
\[
\Hom_{A_\epsilon}\bigl(T,\,\widehat\Omega^1_{A_\epsilon/\BF_\epsilon}\otimes_{A_\epsilon}Q_\epsilon/A_\epsilon\bigr) \isoto \Ext^1_{A_\epsilon}\bigl(T, \widehat\Omega^1_{A_\epsilon/\BF_\epsilon}\bigr).
\]
Combining this isomorphism with the exact sequence \eqref{EqHomExt}, we obtain the following exact sequence:
\begin{equation}\label{EqDualityES}
0\,\longto\, \Hom_{A_\epsilon}\bigl(T_1, \widehat\Omega^1_{A_\epsilon/\BF_\epsilon}\bigr)\,\longto\, \Hom_{A_\epsilon}\bigl(T_0, \widehat\Omega^1_{A_\epsilon/\BF_\epsilon}\bigr) \,\longto\, \Hom_{A_\epsilon}\bigl(T,\,\widehat\Omega^1_{A_\epsilon/\BF_\epsilon}\otimes_{A_\epsilon}Q_\epsilon/A_\epsilon\bigr)\,\longto\,0\,.
\end{equation}
We can make explicit the surjective map in \eqref{EqDualityES} in a similar way as Lemma~\ref{LemmaTorTateModES}. Namely, we view $\xi\colon T_0 \rightarrow \widehat\Omega^1_{A_\epsilon/\BF_\epsilon}$ as a map $T_1 \rightarrow \widehat\Omega^1_{A_\epsilon/\BF_\epsilon}\otimes_{A_\epsilon}Q_\epsilon$ (using the isomorphism $T_0\otimes_{A_\epsilon}Q_\epsilon\xrightarrow\sim T_1\otimes_{A_\epsilon}Q_\epsilon$), and take $\bar\xi\colon T=T_1/T_0 \rightarrow \widehat\Omega^1_{A_\epsilon/\BF_\epsilon}\otimes_{A_\epsilon}Q_\epsilon/A_\epsilon$ to be its reduction. Then the surjective map in \eqref{EqDualityES} is defined by $\xi\mapsto\bar\xi$.

Applying \eqref{EqDualityES} to \eqref{EqPropTorTateMod} and the exact sequence in Lemma~\ref{LemmaTorTateModES}, we obtain the following commutative diagrams with exact columns:
\begin{equation}\label{EqPropCompTorTateMod}
\xymatrix@C=3.5pt{
0 \ar[d]&& 0 \ar[d]& 
0 \ar[d]&& 0 \ar[d]\\
T_\epsilon G_0  \ar[rr]^-{\cong} \ar[d]&& \Hom_{A_\epsilon}\bigl(\check T_\epsilon\ulHM_0 ,\widehat\Omega^1_{A_\epsilon/\BF_\epsilon}\bigr) \ar[d]&
\check T_\epsilon\ulHM_1  \ar[rr]^-{\cong} \ar[d]&&  \Hom_{A_\epsilon}\bigl(T_\epsilon G_1 ,\widehat\Omega^1_{A_\epsilon/\BF_\epsilon}\bigr) \ar[d]\\
T_\epsilon G_1\ar[rr]^-{\cong} \ar[d] && \Hom_{A_\epsilon}\bigl(\check T_\epsilon\ulHM_1,\widehat\Omega^1_{A_\epsilon/\BF_\epsilon}\bigr) \ar[d]& 
\check T_\epsilon\ulHM_0\ar[rr]^-{\cong} \ar[d] && \Hom_{A_\epsilon}\bigl(T_\epsilon G_0,\widehat\Omega^1_{A_\epsilon/\BF_\epsilon}\bigr) \ar[d] \\
H(K^\sep)  \ar[rr]^-{\cong} _-{(*')}\ar[d] &&\Hom_{A_\epsilon}\bigl(\check T_\epsilon\ulHM ,\widehat\Omega^1_{A_\epsilon/\BF_\epsilon}\otimes_{A_\epsilon}Q_\epsilon/A_\epsilon\bigr) \,; \ar[d] &
\check T_\epsilon\ulHM  \ar[rr]^-{\cong}_-{(*)} \ar[d] &&\Hom_{A_\epsilon}\bigl(H(K^\sep) ,\widehat\Omega^1_{A_\epsilon/\BF_\epsilon}\otimes_{A_\epsilon}Q_\epsilon/A_\epsilon\bigr) \,; \ar[d] \\
0&&0&0&&0}
\end{equation}
where the horizontal isomorphisms on the first two rows are defined in Proposition~\ref{PropCompTateMod}.

We now give an intrinsic description of the isomorphism $(*)$ and $(*')$ in the above diagram, not depending on the choice of the isogeny $\ulHM_1\rightarrow\ulHM_0$.
Assume that $H$  is killed by $z^n$. Then we can identify the Galois-equivariant $A/\epsilon^n$-isomorphism
\[
\Hom_{\BF_\epsilon}(A/\epsilon^n, H(K^\sep)) \isoto H(K^\sep)),
\]
defined by $\bar\xi\mapsto \bar\xi(1)$. This isomorphism is compatible with increasing $n$ with respect to the natural projections $A/\epsilon^{n+n'} \onto A/\epsilon^n$. 

Let us consider the following $A/\epsilon^n$-linear isomorphism
\[
 \widehat\Omega^1_{A_\epsilon/\BF_\epsilon}\otimes_{A_\epsilon}\epsilon^{-n}A_\epsilon/A_\epsilon \isoto \Hom_{\BF_\epsilon}(A/\epsilon^n,\BF_\epsilon),\quad 
 \omega \mapsto (a\mapsto \Res_\epsilon(\tilde a\tilde\omega)),
\]
where $\tilde\omega\in\epsilon^{-n}\widehat\Omega^1_{A_\epsilon/\BF_\epsilon}$ and $\tilde a\in A_\epsilon$ are some lifts of $\omega$ and $a$, respectively. Note that the residue $\Res_\epsilon (\tilde a\tilde\omega)$ modulo $\epsilon^n$ only depends on $a$ and $\omega$. If we choose a uniformizing parameter $z\in A_\epsilon$, then this map can be identified with the reduction of the isomorphism $\widehat\Omega^1_{A_\epsilon/\BF_\epsilon} \isoto \Hom_{\BF_\epsilon}(Q_\epsilon/A_\epsilon,\BF_\epsilon)$ (introduced above Proposition~\ref{PropCompTateMod}) via the isomorphism $A/\epsilon^n\isoto \epsilon^{-n}A_\epsilon/A_\epsilon$ sending $1$ to $\frac{1}{z^n}$, and we can write down the inverse map as $\lambda \mapsto dz\otimes(\sum_{i=1}^{n}\lambda(z^{i-1})z^{-i})$.

Let us now state the following comparison result analogous to Proposition~\ref{PropCompTateMod}:

\begin{proposition}\label{PropCompTorTateMod}
We use the notation as above, and view $\ulHM$ as $\Hom_{R\text{\rm-groups},\BF_\epsilon\text{\rm-lin}}(H\,,\,\BG_{a,R})$; \emph{cf.} Remark~\ref{RemTorAndersonMod}.
Then the following Galois-equivariant pairing of $A_\epsilon$-modules
\begin{align*}
H(K^\sep)\times \check T_\epsilon\ulHM&\longto \widehat\Omega^1_{A_\epsilon/\BF_\epsilon}\otimes_{A_\epsilon}Q_\epsilon/A_\epsilon\,,\\
(f,m)&\mapsto m\circ f \in \Hom_{\BF_\epsilon}(A/\epsilon^n,\BF_\epsilon) \cong \widehat\Omega^1_{A_\epsilon/\BF_\epsilon}\otimes_{A_\epsilon}\epsilon^{-n}A_\epsilon/A_\epsilon \subset \widehat\Omega^1_{A_\epsilon/\BF_\epsilon}\otimes_{A_\epsilon}Q_\epsilon/A_\epsilon\,;
\end{align*}
is perfect; in other words, it induces the following Galois-equivariant $A_\epsilon$-linear isomorphisms:
\[
 H(K^\sep)\isoto \Hom_{A_\epsilon}\bigl(\check T_\epsilon\ulHM,\widehat\Omega^1_{A_\epsilon/\BF_\epsilon}\otimes_{A_\epsilon}Q_\epsilon/A_\epsilon\bigr) \quad\text{and}\quad \check T_\epsilon\ulHM \isoto \Hom_{A_\epsilon}\bigl(H(K^\sep),\widehat\Omega^1_{A_\epsilon/\BF_\epsilon}\otimes_{A_\epsilon}Q_\epsilon/A_\epsilon\bigr).
\] 
Furthermore, if we choose an isogeny $\ulHM_1\rightarrow\ulHM_0$ whose cokernel is $\ulHM$, then the above isomorphism coincides with the isomorphisms $(*)$ and $(*')$ in the diagram \eqref{EqPropCompTorTateMod}.

If we choose a uniformizing parameter $z\in A_\epsilon$ and identify $\widehat\Omega^1_{A_\epsilon/\BF_\epsilon}\cong A_\epsilon dz$, then the above pairing non-canonically identifies $H(K^\sep)$ with the Pontryagin dual of $\check T_\epsilon\ulHM$ as a torsion $\Gal(K^\sep/K)$-representation.
\end{proposition}
\begin{proof}
By Lemma~\ref{LemmaRaynaudBBM} we can find an isogeny of effective local shtukas $\ulHM_1\rightarrow\ulHM_0$ whose cokernel is the given effective torsion local shtuka $\ulHM$. Therefore, it suffices to show that the pairing in the statement induces the isomorphisms $(*)$ and $(*')$ in the diagram \eqref{EqPropCompTorTateMod}. Now, since  the surjective vertical arrows in the diagram \eqref{EqPropCompTorTateMod} are  made completely explicit (\emph{cf.} Lemma~\ref{LemmaTorTateModES} and the discussion below \eqref{EqDualityES}), we can directly check the commutativity of \eqref{EqPropCompTorTateMod} with horizontal maps given by the pairings in Propositions~\ref{PropCompTateMod} and \ref{PropCompTorTateMod}.
\end{proof}
\begin{remark}
As suggested in the proof of Proposition~\ref{PropCompTorTateMod}, one can easily show that Propositions~\ref{PropCompTateMod} and \ref{PropCompTorTateMod} are equivalent. Indeed, it is possible to deduce Proposition~\ref{PropCompTateMod} by directly proving Proposition~\ref{PropCompTorTateMod} and apply it to $H=G[z^n]$ for a $z$-divisible local Anderson module $G$ over $R$.
\end{remark}

\begin{definition}\label{DefTorShModel}
Let $\Rep_{A_\epsilon}^\tors\Gal(K^\sep/K)$ denote the category of finite-length torsion $A_\epsilon$-modules equipped with a discrete action of $\Gal(K^\sep/K)$.
Then for $T\in \Rep_{A_\epsilon}^\tors\Gal(K^\sep/K)$, a \emph{torsion local shtuka model} of $T$ is a torsion local shtuka $\ulHM$ over $R$ equipped with a $\Gal(K^\sep/K)$-equivariant isomorphism $T\cong \check T_\epsilon\ulHM$. 
\end{definition}

Contrary to the case of local shtukas over $R$, the functor $\check T_\epsilon$ from the category of torsion local shtukas over $R$ to $\Rep_{A_\epsilon}^\tors\Gal(K^\sep/K)$ is not fully faithful. In particular, it is possible to have more than one torsion local shtuka model for $ T\in\Rep_{A_\epsilon}^\tors\Gal(K^\sep/K)$. For example, for any torsion local shtuka $\ulHM$ over $R$, the Galois action on $\check T_\epsilon\ulHM $ trivializes after replacing $R$ with some finite extension, so as $R$ gets more ramified we obtain more torsion local shtuka models of the trivial torsion Galois module.

\begin{example}\label{ExTorLSModel}
In this example, we find  more torsion local shtuka models of $A_\epsilon/(z^n)$ with trivial $\Gal(K^\sep/K)$-action. We consider the following torsion local shtuka for any $d\in\BZ$:
\[
\ulHM^d_n:=(R\dbl z \dbr/(z^n), (z-\zeta)^d).
\]
Using the notation from Example~\ref{ExCarlitz4A}, it follows that $\check T_\epsilon\ulHM_n^d\subset K^\sep\dbl z\dbr/(z^n)$ is an $A_\epsilon/(z^n)$-lattice generated by $(\tplus)^{-d}\bmod{z^n}$  with the Galois action given by the character $(\chi_\epsilon)^{-d}\bmod{z^n}$. Moreover, if $e\in\BN$ is such that $\hat q^e\ge n$ and $d=\hat q^e(\hat q-1)d'$ then we have $(\tplus)^{-d}\bmod{z^n}=\ell_0^{-d}=(-\zeta)^{-\hat q^ed'}\in K\dbl z\dbr/(z^n)$; in other words, $(\chi_\epsilon)^d\equiv 1\bmod{z^n}$. Therefore, in this case $\ulHM^d_n$ defines a torsion local shtuka model of the trivial Galois module $A_\epsilon/(z^n)$ for any $d\in\BZ$. Note that all these are pairwise non-isomorphic because $(\tplus)^{-d}\bmod{z^n}$ is not a unit in $R\dbl z\dbr/(z^n)$.
\end{example}

\begin{proposition}\label{PropRamakrishnaCrit}
The full subcategory of $\Rep_{A_\epsilon}^\tors\Gal(K^\sep/K)$, consisting of torsion Galois representations which admit an effective torsion local shtuka model with height $\leqslant d$, is stable under subquotients, direct products, and twisted duality $T\mapsto T^\vee(d):=\Hom_{A_\epsilon}(T,Q_\epsilon/A_\epsilon)\otimes(\chi_\epsilon)^d$, where $\chi_\epsilon$ is defined in Example~\ref{ExCarlitz4A}.
\end{proposition}

The proposition, especially the assertion about subquotients, is not obvious since $\check T_\epsilon$ may not be fully faithful. Recall that the analogous assertion for torsion representations of the Galois group of a $p$-adic field which admit finite flat group scheme models over the valuation ring can be obtained by working with the scheme-theoretic closure of the generic fiber. The following lemma provides an analog of the scheme-theoretic closure, which is needed for the proof of Proposition~\ref{PropRamakrishnaCrit}.

\begin{lemma}\label{LemmaSchCl}
Let $\ulHM$ be a torsion local shtuka over $R$, and consider a $\Gal(K^\sep/K)$-stable quotient $\check T_\epsilon\ulHM\onto T'$. We define $\hat M'$ to be the image of the following map:
\[
\hat M \hookrightarrow \hat M\otimes_{R\dbl z\dbr}K^\sep\dbl z\dbr \cong \check T_\epsilon\ulHM\otimes_{A_\epsilon}K^\sep\dbl z\dbr \onto T'\otimes_{A_\epsilon}K^\sep\dbl z\dbr.
\]
Then, $\tau_{\hat M}$ induces an isomorphism $\tau_{\hat M'}:\hat\sigma^*\hat M'[\frac{1}{z-\zeta}]\isoto\hat M'[\frac{1}{z-\zeta}]$, and $\ulHM':=(\hat M',\tau_{\hat M'})$ is a torsion local shtuka over $R$. If $\ulHM$ is effective (respectively, effective with height $\leqslant d$), then so is $\ulHM'$. Furthermore, $\ulHM'$ is the unique quotient of $\ulHM$ such that $ T' = \check T_\epsilon\ulHM'$ as a quotient of $\check T_\epsilon\ulHM$.
\end{lemma}

\begin{proof}
Note that $\hat M'$ is flat over the discrete valuation ring $R$ and finitely presented over the noetherian ring $R\dbl z\dbr$ because it is a submodule of $T'\otimes_{A_\epsilon}K^\sep\dbl z\dbr$ and a quotient of $\hat M$. The isomorphism $\tau_{\hat M'}$ is obtained from the diagram
\[
\xymatrix {
\hat\sigma^*\hat M[\tfrac{1}{z-\zeta}] \ar@{->>}[r] \ar[d]_{\TS\tau_{\hat M}}^{\TS\cong} & \hat\sigma^*\hat M'[\tfrac{1}{z-\zeta}] \ar@{^{ (}->}[r] \ar@{-->}[d]_{\TS\tau_{\hat M'}}^{\TS\cong} & T'\otimes_{A_\epsilon}K^\sep\dbl z\dbr \ar@{=}[d]\\
\hat M[\tfrac{1}{z-\zeta}] \ar@{->>}[r] & \hat M'[\tfrac{1}{z-\zeta}] \ar@{^{ (}->}[r] & T'\otimes_{A_\epsilon}K^\sep\dbl z\dbr 
}
\]
using that $\hat\sigma\colon R\to R$ is flat. The diagram also shows that $\ulHM'$ is effective (respectively, effective with height $\leqslant d$) if $\ulHM$ is. Finally, the uniqueness of $\ulHM'$ follows from Proposition~\ref{PropTorTateMod}.
\end{proof}

\begin{remark}
In the setting of Lemma~\ref{LemmaSchCl}, assume furthermore that $\ulHM$ is effective. Then for a Galois-stable quotient $T'$ of $\check T_\epsilon\ulHM$, the associated $\ulHM'$ is effective. So there is a unique $A_\epsilon$-submodule scheme $H'_K$ of the generic fiber of $H:=\Dr_{\hat q}\ulHM'$ such that $H'_K(K^\sep)$ is the Pontryagin dual of $T'$. Then $\Dr_{\hat q}\ulHM'$ is the scheme-theoretic closure of $H'_K$ in $H$, which can be seen from the uniqueness of the scheme-theoretic closure.
\end{remark}

\begin{proof}[Proof of Proposition~\ref{PropRamakrishnaCrit}]
The claim on direct products is clear as $\check T_\epsilon$ commutes with direct products.

Let $\ulHM$ be an effective torsion local shtuka over $R$ with height $\leqslant d$, and set $T:=\check T_\epsilon \ulHM$.
Lemma~\ref{LemmaSchCl} shows that any Galois-stable quotient of $T$ admits an effective torsion local shtuka model with height $\leqslant d$. Let $T'\subset T$ be a Galois-stable $A_\epsilon$-submodule. Then by Lemma~\ref{LemmaSchCl}, there exists a quotient $\ulHM''$ of $\ulHM$ corresponding to $T/T'$. Then Proposition~\ref{PropTorTateMod} implies that $\hat M':=\ker(\hat M \onto\hat M'')$ can naturally be viewed as a torsion local shtuka model of $T'$, which is effective and of height $\leqslant d$.

We construct a torsion local shtuka model of $T^\vee(d)$ which is effective and of height $\leqslant d$. We first define $\ulHM^\vee:=(\hat M^\vee, \tau_{\hat M^\vee})$ as follows:
\begin{eqnarray}\label{EqPonDual}
\hat M^\vee & := & \Hom_{R\dbl z\dbr}(\hat M,R\dbl z \dbr[\tfrac{1}{z}]/R\dbl z\dbr)\,,\\[2mm]
\tau_{\hat M^\vee} & := & (\tau_{\hat M}^{-1})^\vee\colon \;\hat\sigma^*\hat M^\vee[\tfrac{1}{z-\zeta}] \;\cong\; (\hat\sigma^*\hat M)[\tfrac{1}{z-\zeta}]^\vee \;\isoto\; \hat M^\vee[\tfrac{1}{z-\zeta}]\,.\nonumber
\end{eqnarray}
Note that if $z^n\hat M=(0)$ then 
\[
\xymatrix @R=0pc {
\Hom_R(M,R)\,\cong\,\Hom_R(M,R)\otimes_{R\dbl z\dbr/(z^n)}\,\tfrac{1}{z^n}R\dbl z\dbr/R\dbl z\dbr \ar[r]^{\qquad\qquad\qquad\qquad\sim} & \hat M^\vee\qquad\qquad\quad\,\;\\
\es\qquad\qquad\qquad\qquad\qquad\qquad h\otimes a \;\es\qquad\qquad\qquad\qquad\ar@{|->}[r] & \bigl(m\mapsto a\cdot h(m)\bigr) \,.
}
\]
So as an $R$-module it is flat and finitely generated over $R$. By \cite[Proposition~2.10]{Eisenbud} we have
\begin{eqnarray*}
\hat M^\vee\otimes_{R\dbl z\dbr}K^\sep\dbl z\dbr & \cong & \Hom_{K^\sep\dbl z\dbr}\bigl(\hat M\otimes_{R\dbl z\dbr}K^\sep\,,\,K^\sep\dbl z \dbr[\tfrac{1}{z}]/K^\sep\dbl z\dbr\bigr) \\[2mm]
& \cong & \Hom_{A_\epsilon}(\check T_\epsilon\hat M,Q_\epsilon/A_\epsilon)\otimes_{A_\epsilon}K^\sep\dbl z\dbr\,,
\end{eqnarray*}
and hence $\check T_\epsilon\hat M^\vee\cong\Hom_{A_\epsilon}(\check T_\epsilon\hat M,Q_\epsilon/A_\epsilon)$. Note that $\ulHM^\vee$ is not necessarily an effective torsion local shtuka, but we can slightly modify to get
\begin{equation}\label{EqCartierDual}
\ulHM^{\vee,d}:=(\hat M^\vee, (z-\zeta)^d\tau_{\hat M^\vee}),
\end{equation} 
which is effective and of height $\leqslant d$ if the same holds for $\ulHM$. To see $T^\vee(d)\cong\check T_\epsilon\ulHM^{\vee,d}$, note that $(\tplus)^d\check T_\epsilon\ulHM^\vee = \check T_\epsilon\ulHM^{\vee,d}$ inside $\hat M^\vee\otimes_{R\dbl z\dbr}K^\sep\dbl z\dbr$, using the notation from Example~\ref{ExCarlitz4A}.
\end{proof}

Just as for finite flat group scheme models of torsion Galois representations of a $p$-adic field, there exists a natural notion of partial ordering on (equivalence classes of) torsion local shtuka models.
\begin{definition}\label{DefPO}
We fix $T\in \Rep^\tors_{A_\epsilon}\Gal(K^\sep/K)$, and consider torsion local shtuka models $\ulHM$ and $\ulHM'$ of $T$. We write $\ulHM\precsim\ulHM'$ if there exists a (necessarily unique) map $\ulHM\rightarrow\ulHM'$ which respects the identification $\check T_\epsilon \ulHM \cong T\cong\check T_\epsilon\ulHM'$. We say that $\ulHM$ and $\ulHM'$ are \emph{equivalent} if $\ulHM\precsim\ulHM'$ and $\ulHM\succsim\ulHM'$; or equivalently, if there exists a necessarily unique isomorphism $\ulHM\cong\ulHM'$ which respects the identification $\check T_\epsilon \ulHM \cong T\cong\check T_\epsilon\ulHM'$. 
\end{definition}

\begin{lemma}\label{LemmaMaxMin}
For $T\in \Rep^\tors_{A_\epsilon}\Gal(K^\sep/K)$, we have the following:
\begin{enumerate}
\item\label{LemmaMaxMin:PO}
$\precsim$ defines a partial ordering on the set of equivalence classes of torsion local shtuka models of $T$.
\item\label{LemmaMaxMin:Exist}
In the set of equivalence classes of effective torsion local shtuka models of $T$ with height $\leqslant d$, there exist unique maximal and minimal elements with respect to the partial ordering $\precsim$. Furthermore, the formation of maximal and minimal effective torsion local shtuka models with height $\leqslant d$ is functorial in the sense that any Galois-equivariant morphism of torsion Galois representations comes from a unique map of their maximal (respectively, minimal) effective torsion local shtuka models with height $\leqslant d$. 
\item\label{LemmaMaxMin:Finite}
The set of equivalence classes of effective torsion local shtuka models of $T$ with height $\leqslant d$ is finite.
\end{enumerate}
\end{lemma}
\begin{proof}
For any torsion local shtuka models $\ulHM$ and $\ulHM'$ of $T$, the isomorphism $\check T_\epsilon\ulHM\cong T\cong \check T_\epsilon\ulHM'$ enables us to identify the equivalence classes of $\ulHM$ and $\ulHM'$ as $R\dbl z\dbr$-submodules of $T\otimes_{R\dbl z\dbr}K^\sep\dbl z\dbr$ which are stable under $1\otimes\hat\sigma$. 
One can easily check that $\ulHM+\ulHM'$ and $\ulHM\cap\ulHM'$ can naturally be viewed as torsion local shtuka models of $T$, where $\hat\tau$ is obtained by the restriction of $1\otimes\hat\sigma$. Namely, $\ulHM+\ulHM'$ is the torsion local shtuka from Lemma~\ref{LemmaSchCl} associated with the quotient $T_\epsilon\ulHM=T_\epsilon\ulHM'$ of $T_\epsilon(\ulHM\oplus \ulHM')$, and $\ulHM\cap\ulHM'$ is the kernel of $\ulHM\oplus \ulHM'\onto\ulHM+ \ulHM'$. This shows that $\precsim$ is a partial ordering, i.e., part~\ref{LemmaMaxMin:PO}.

We next show that the set of equivalence classes of \emph{effective} torsion local shtuka models of $T$ admits a unique maximal element (if the set is non-empty).  For an effective torsion local shtuka model $\ulHM$ of $T$ we consider
\[
\Dr_{\hat q}(\ulHM)\;:=\;\Spec\;(\Sym_R\hat M)\big/\bigl(m^{\otimes\hat q}-\tau_{\hat M}(\hat\sigma_{\hat M}^*m)\colon m\in \hat M\bigr)\,;
\]
as in Remark~\ref{RemTorAndersonMod}. If $\ulHM'$ is another effective torsion local  shtuka model of $T$ with $\ulHM\precsim\ulHM'$, then there exists a finite morphism $\Dr_{\hat q}(\ulHM')\rightarrow\Dr_{\hat q}(\ulHM)$ which is an isomorphism over $\Spec K$. On the other hand, the normalization $\wt X$ of $\Dr_{\hat q}(\ulHM)$ in its generic fiber is \emph{finite} over $\Dr_{\hat q}(\ulHM)$ as the generic fiber is \'etale over $K$ by the Jacobi criterion or because $\ulHM$ is \'etale over $K$. Indeed, the trace pairing
\[
(\CO_{\Dr_{\hat q}(\ulHM)}\otimes_RK )\times (\CO_{\Dr_{\hat q}(\ulHM)}\otimes_RK) \rightarrow K
\]
is perfect by \'etaleness \cite[IV$_4$, Proposition~18.2.3(c)]{EGA}, so the dual of $\CO_{\Dr_{\hat q}(\ulHM)}$ is finite over $\CO_{\Dr_{\hat q}(\ulHM)}$ and contains the normalisation of $\CO_{\Dr_{\hat q}(\ulHM)}$. Thus the normalization is finite. This shows that the set of equivalence class of effective torsion local shtuka models of $T$ is bounded above with respect to $\precsim$, because every such is contained in the effective torsion local shtuka corresponding to $\wt X$. On the other hand, if $\ulHM$ and $\ulHM'$ are effective torsion local shtuka models of $T$, then so is the torsion local shtuka model $\ulHM+\ulHM'$. This shows the uniqueness of the maximal element.

Let us show there exists a unique minimal effective torsion local shtuka model with height $\leqslant d$ up to equivalence. Let $\ulHM'$ be a torsion local shtuka model of $T^\vee(d)$ maximal among those effective and of height $\leqslant d$. Since $\ulHM'\mapsto(\ulHM')^{\vee,d}$ reverses the partial ordering by \eqref{EqPonDual}, it follows that $(\ulHM')^{\vee,d}$ is minimal among effective torsion local shtuka models of $T$ with height $\leqslant d$.

For the functoriality claim on maximal and minimal objects, we  may assume that $\ulHM$ and $\ulHM'$ are maximal effective torsion local shtuka models of $T=\check T_\epsilon\ulHM$ and $T'=\check T_\epsilon\ulHM'$, respectively. (The case of minimal objects is reduced to this since the functor \eqref{EqCartierDual} switches the maximal and the minimal objects.) Let $f\colon T\rightarrow T'$ be a $\Gal(K^\sep/K)$-equivariant map. If $f$ is surjective, then the image of $\hat M$ under $\hat M  \hookrightarrow T\otimes_{A_\epsilon}K^\sep\dbl z\dbr \xrightarrow{f\otimes1}T'\otimes_{A_\epsilon}K^\sep\dbl z\dbr$ is an effective local torsion shtuka model of $T'$ with height $\leqslant d$, so  it is contained in $\hat M'$ by the maximality of $\ulHM'$. If $f$ is not surjective, then consider a surjective map $T\oplus T' \xrightarrow{f+\id}T'$. Let $\ulHM''$ denote a  torsion local shtuka model of $T\oplus T'$ maximal among those effective and of height $\leqslant d$, then $f+\id$ induces a map $\ulHM''\rightarrow \ulHM'$. By 
maximality of $\ulHM''$, there is a map $\ulHM\oplus\ulHM'\rightarrow\ulHM''$, so we obtain
\[
\ulHM \xrightarrow{(\id,0)}\ulHM\oplus\ulHM'\,\longto\,\ulHM'' \,\longto\, \ulHM',
\] 
which induces $f\colon T\rightarrow T'$. This proves part~\ref{LemmaMaxMin:Exist}.

Let $\ulHM^+$ and $\ulHM^-$ respectively denote the maximal and minimal effective torsion local shtuka models of $T$ with height $\leqslant d$. Since $R$ is discretely valued, the quotient $\hat M^+/\hat M^-$ is of finite length over $R\dbl z\dbr$. Since any equivalence class of effective torsion local shtuka models of $T$ with height $\leqslant d$ gives rise to a unique $R\dbl z\dbr$-submodule of $\hat M^+/\hat M^-$, we obtain the desired finiteness claim (i.e., part~\ref{LemmaMaxMin:Finite}).
\end{proof}

\begin{lemma}\label{LemmaLimit}
Let $T$ be a finitely generated free $A_\epsilon$-module equipped with a continuous action of $\Gal(K^\sep/K)$. Then the following are equivalent:
\begin{enumerate}
 \item\label{LemmaLimitLatt} There exists an effective local shtuka $\ulHM$ with $(z-\zeta)^d\hat M\subset\tau_{\hat M}(\hat\sigma^*\hat M)\subset\hat M$ such that $T\cong\check T_\epsilon\ulHM$.
 \item\label{LemmaLimitTor} For each positive integer $n$, there exists an effective torsion local shtuka $\ulHM_n$ with height $\leqslant d$ (i.e., satisfying $(z-\zeta)^d\hat M_n\subset\tau_{\hat M_n}(\hat\sigma^*\hat M_n)\subset\hat M_n$) such that $T/(z^n)\cong \check T(\ulHM_n)$.
\end{enumerate}
\end{lemma}

\begin{proof}
It is clear that \ref{LemmaLimitLatt} implies \ref{LemmaLimitTor}, so let us assume \ref{LemmaLimitTor}.
By choosing each $\ulHM_n$ to be maximal among effective torsion local shtuka models with height $\leqslant d$, we may assume that for each $n$ there exists a (not necessarily surjective) morphism $\ulHM_{n+1} \rightarrow \ulHM_n$ that induces the natural projection $T/(z^{n+1}) \onto T/(z^n)$ by Lemma~\ref{LemmaMaxMin}\ref{LemmaMaxMin:Exist}. Set $\hat M':=\varprojlim_n\hat M_n$, equipped with $\tau_{\hat M'}:\sigma^*\hat M'[\frac{1}{z-\zeta}] \isoto \hat M'[\frac{1}{z-\zeta}]$ obtained as the limit of $\{\tau_{\hat M_n}\}$. Note that the projective system $(\hat M_n)$ satisfies the Mittag-Leffler condition, because by Lemma~\ref{LemmaMaxMin}\ref{LemmaMaxMin:Finite} for each $n$ the set of images $\hat M_{n'}\to\hat M_n$ is finite. Therefore $\hat M'/z\hat M'$ is a quotient of $\hat M_n$ for $n\gg1$, and since $\hat M'$ is $z$-adically separated and complete, it is finitely generated over $R\dbl z\dbr$.

We next consider the following injective map
\[
\hat M' = \varprojlim_n\hat M_n \rightarrow  \varprojlim_n (\hat M_n\otimes_{R\dbl z\dbr} K^\sep\dbl z\dbr) \cong \varprojlim_n (T/z^nT\otimes_{R\dbl z\dbr} K^\sep\dbl z\dbr)\cong T\otimes_{A_\epsilon}K^\sep\dbl z\dbr,
\]
where the non-trivial isomorphism is from Proposition~\ref{PropTorTateMod}. (The injectivity follows from the left exactness of the projective limit.) In particular, $\hat M'$ is torsion-free. By $K^\sep\dbl z\dbr$-linearly extending it, we obtain a map
\begin{equation}\label{EqLimit}
\hat M'\otimes_{R\dbl z\dbr}K^\sep\dbl z\dbr \rightarrow T\otimes_{A_\epsilon}K^\sep\dbl z\dbr,
\end{equation}
which is  injective by torsion-freeness of $\hat M'$. We claim that this map is an isomorphism. Indeed, the surjectivity follows since the image of $\hat M'$ in $\hat M_n\otimes_{R\dbl z\dbr}K^\sep\dbl z\dbr \cong T/z^nT\otimes_{A_\epsilon}K^\sep\dbl z\dbr$ coincides with the image of $\hat M_{n'}$ for $n'\gg n$ by Lemma~\ref{LemmaMaxMin}\ref{LemmaMaxMin:Finite}.

One can now apply Lemma~\ref{LemmaSaturation} to $(\hat M', \tau_{\hat M'})$ to obtain a local shtuka  $\ulHM'':=(\hat M'',\tau_{\hat M''})$ with  $\hat M'[\frac{1}{z-\zeta}] = \hat M''[\frac{1}{z-\zeta}]$. So we have $(z-\zeta)^d\hat M''\subset\tau_{\hat M''}(\hat\sigma^*\hat M'')\subset\hat M''$. Furthermore, the isomorphism \eqref{EqLimit} implies that  $\check V_\epsilon\ulHM'' \cong T[\frac{1}{z}]$. Then by Proposition~\ref{PropLattices}, there exists a local shtuka $\ulHM$ isogenous to $\ulHM''$ with $T = \check T_\epsilon\ulHM$. Finally, the property $(z-\zeta)^d\hat M\subset\tau_{\hat M}(\hat\sigma^*\hat M)\subset\hat M$ could be checked up to isogeny, which concludes the proof.
\end{proof}
\begin{remark}
Using the same notation in the proof of Lemma~\ref{LemmaLimit}, one can also show that $\ulHM = \ulHM''$ by using the isomorphism \eqref{EqLimit} and suitably adapting the exact sequence \eqref{EqLatticesProof}.
\end{remark}

We finish the section by discussing torsion local shtukas with coefficients. Let $B$ be an $A_\epsilon$-algebra with $\#B<\infty$, and set $R\dbl z\dbr_B:=R\dbl z\dbr\otimes_{A_\epsilon}B$. We define $\hat\sigma:R\dbl z\dbr_B\rightarrow R\dbl z\dbr_B$ by $B$-linearly extending $\hat\sigma:R\dbl z\dbr \rightarrow R\dbl z\dbr$.
\begin{definition}
A \emph{torsion local $B$-shtuka}  over $R$ is a pair $\ulHM_B=(\hat M_B,\tau_{\hat M_B})$ consisting of a finitely generated free $R\dbl z\dbr_B$-module $\hat M_B$, and an $R\dbl z\dbr_B$-linear isomorphism $\tau_{\hat M_B}\colon\hat\sigma^\ast\hat M_B[\frac{1}{z-\zeta}] \isoto\hat M_B[\frac{1}{z-\zeta}]$. We take the obvious notion of morphisms.

If $B'$ is a $B$-algebra with $\#B'<\infty$ and $\ulHM_B$ is a torsion local $B$-shtuka, then we define a torsion local $B'$-shtuka $\ulHM_B\otimes_B B':=(\hat M_B\otimes_B B', \tau_{\hat M_B}\otimes \id_{B'})$. We call $\ulHM_B\otimes_B B'$ the \emph{scalar extension} of $\ulHM_B$

We may view a torsion local $B$-shtuka $\ulHM_B$ as a torsion local shtuka by forgetting the $B$-action. We say that $\ulHM_B$ is \emph{effective} (respectively, \emph{effective with height $\leqslant d$}; respectively, \emph{\'etale}) if it is so as a torsion local shtuka. These properties are stable under scalar extensions.
\end{definition}

\begin{lemma}
\begin{enumerate}
\item
For any torsion local $B$-shtuka $\ulHM_B$, $\check T_\epsilon\ulHM_B$ is a finitely generated \emph{free} $B$-module of rank equal to $\rk_{R\dbl z\dbr_B}\hat M_B$, where the $B$-action on $\check T_\epsilon\ulHM_B$ is induced from the $B$-action on $\ulHM_B$.
\item
For any $B$-algebra $B'$ with $\#B'<\infty$, we have a natural $\Gal(K^\sep/K)$-equivariant isomorphism $\check T_\epsilon(\ulHM_B\otimes_B B') \cong (\check T_\epsilon\ulHM_B)\otimes_B B'$.
\end{enumerate}
\end{lemma}
\begin{proof}
Both claims are straightforward from the isomorphism (\ref{EqTorTateModIsom}).
\end{proof}

Let $\ulHM=(\hat M,\tau_{\hat M})$ be a torsion local shtuka equipped with a $B$-action that commutes with $\tau_{\hat M}$, and assume that $\check T_\epsilon\ulHM$ is \emph{free} over $B$. 
Note that it does not necessarily follow that $\ulHM$ is a torsion local $B$-shtuka; indeed, if $B$ is non-reduced then one can create an example where $\hat M$ is \emph{not} free over $R\dbl z\dbr\otimes_{A_\epsilon}B$ by the same approach as in \cite[Remark~1.6.3]{KimNormField}.  On the other hand, we have the following lemma:

\begin{lemma}\label{LemmaBFSht}
Let $\BF$ be a finite extension of $\BF_\epsilon$. Then any local shtuka $\ulHM$ equipped with an $\BF$-action is a torsion local $\BF$-shtuka; i.e., the underlying module $\hat M$ is free over $R\dbl z\dbr_\BF$.
\end{lemma}

\begin{proof} 
Since the assertion can be checked after increasing the residue field of $R$, we may assume that $R$ contains $\BF$. Since $z=0$ in $\BF$ we have $R\dbl z\dbr_\BF \cong R\otimes_{\BF_\epsilon}\BF \cong \prod_{i=0}^{[\BF:\BF_\epsilon]-1} R^{(i)}$ where $R^{(i)}$ is $R$ viewed as an $\BF$-algebra via $\hat\sigma^i:\BF\rightarrow R$. Likewise, we obtain the isotypic decomposition $\hat M =  \prod_i \hat M^{(i)}$ for any local shtuka $\ulHM=(\hat M, \tau_{\hat M})$ with $\BF$-action. Since $\hat M$ is $\zeta$-torsion free, each factor $\hat M^{(i)}$ is free over $R^{(i)}$. It remains to show that the $R^{(i)}$-rank of $\hat M^{(i)}$ is constant, which follows since  $\tau_{\hat M}$ restricts to an isomorphism $\hat\sigma^*\hat M^{(i-1)}[\frac{1}{z-\zeta}]\xrightarrow\sim \hat M^{(i)}[\frac{1}{z-\zeta}]$ for any $1\leqslant i \leqslant [\BF:\BF_\epsilon]-1$.
\end{proof}

%
%

\section{Deformation theory of Galois representations} \label{SectDefoTh}
\setcounter{equation}{0}

Mazur \cite{MazurMSRI} has introduced the notion of Galois deformation rings (both for number fields and $p$-adic local fields), as an attempt to see the counterpart on the Galois representation side of $p$-adic deformations of modular forms; cf.~Hida theory. Since Wiles's proof of Fermat's last theorem, which opened up its application to modularity of Galois representations, Galois deformation theory has become an indispensable technical tool in number theory. 

Among the important actors in modularity lifting theorems are $p$-adic local deformation rings with a certain condition in terms of Fontaine's theory, such as flat deformation rings, crystalline deformation rings and potentially semi-stable deformation rings. In the second author's thesis (cf.~\cite[\S4]{KimNormField}), it was shown that there exists an equi-characteristic analog of flat deformation rings (or crystalline deformation rings) by working with torsion local shtukas and Hodge-Pink theory instead of finite flat group schemes and Fontaine's theory. The existence of such equi-characteristic deformation rings is quite surprising because without imposing any deformation condition, the deformation functor would have infinite-dimensional tangent space; cf.~Remark~\ref{RemTangentSp}. Indeed, the main result we obtain is the finiteness of the tangent space when we impose a suitable deformation condition in terms of torsion local shtuka models. 

In this section, we indicate the main idea for the existence of the equi-characteristic deformation rings, and list some of their properties, such as smoothness and dimension, which are analogous to Kisin's study of flat deformation rings; cf.~\cite{KisinModuliFF}.

Throughout this section, we assume that $K$ is a \emph{finite} extension of $Q_\epsilon$.
We fix a finite extension $\BF$ of $\BF_\epsilon$ and  a continuous homomorphism 
\[\bar\rho\colon\Gal(K^\sep/K)\rightarrow\GL_r(\BF).\]
Let $\Art_{\BF\dbl z\dbr}$ denote the category of Artinian local $\BF\dbl z\dbr$-algebras with residue field $\BF$.
\begin{definition}\label{DefDeformation}
For $B\in\Art_{\BF\dbl z\dbr}$, a \emph{framed deformation} of $\bar\rho$ over $B$ is a continuous homomorphism $\rho_B\colon\Gal(K^\sep/K)\rightarrow\GL_r(B)$ which reduces to $\bar\rho$ modulo the maximal ideal $\Fm_B$ of $B$. For any morphism $B\rightarrow B'$ in $\Art_{\BF\dbl z\dbr}$, we define the \emph{scalar extension} $\rho_{B'}$ of a framed deformation $\rho_B$ to be the composite
\[
\rho_{B'}\colon \Gal(K^\sep/K)\xrightarrow{\rho_B}\GL_r(B) \rightarrow\GL_r(B').
\]
Let $\Def^\Box_{\bar\rho}$ denote the set-valued functor on $\Art_{\BF\dbl z\dbr}$, where $\Def^\Box_{\bar\rho}(B)$ is the set of framed deformations of $\bar\rho$ over $B$, and for any morphism $B\rightarrow B'$ in $\Art_{\BF\dbl z\dbr}$ the map $\Def_{\bar\rho}^\Box(B)\rightarrow\Def_{\bar\rho}^\Box(B')$ is defined by the scalar extension $\rho_B\mapsto\rho_{B'}$. 

A \emph{deformation} of $\bar\rho$ over $B$ is an equivalence class of framed deformations of $\bar\rho$ over $B$, where two framed deformations $\rho_B$ and $\rho'_B$ are equivalent if and only if there exists $g\in\ker(\GL_r(B)\onto\GL_r(\BF))$ which conjugates $\rho_B$ to $\rho_B'$. Let $\Def_{\bar\rho}$ denote a set-valued functor on $\Art_{\BF\dbl z\dbr}$, where $\Def_{\bar\rho}(B)$ is the set of  deformations of $\bar\rho$ over $B$.
\end{definition}

\begin{remark}\label{RemTangentSp}
If $K$ is replaced by a finite extension of $\mathbb{Q}_p$, then it is a well-known result of Mazur that $\Def^\Box_{\bar\rho}$ is pro-representable by a complete local noetherian ring (and $\Def_{\bar\rho}$ is pro-representable under some restriction on $\bar\rho$). For equi-characteristic $K$, on the other hand,  these functors cannot be pro-represented by a complete local noetherian ring because the reduced tangent spaces  $\Def_{\bar\rho}(\BF[u]/(u^2))$ and $\Def^\Box_{\bar\rho}(\BF[u]/(u^2))$ are infinite-dimensional. Indeed, we have $\Def_{\bar\rho}(\BF[u]/(u^2)) \cong \Koh^1(K, \bar\rho\otimes\bar\rho^\vee)$ by \cite[p.~391]{MazurMSRI}, and the Galois cohomology group is infinite-dimensional over $\BF$, because if $\ol 1\colon\Gal(K^\sep/K)\to\GL_1(\BF)$ is the trivial representation $\Koh^1(K,\ol 1)\cong \Hom(\Gal(K^\sep/K),\BF)$ is infinite-dimensional (by Artin-Schreier theory; see \cite[\S\,VI.1]{NSW08}) and it maps to $\Koh^1(K, \bar\rho\otimes\bar\rho^\vee)$ with finite dimensional kernel, by the following long exact sequence
\[
\Koh^0(K, \bar\rho\otimes\bar\rho^\vee/\ol1) \,\longto\, \Koh^1(K,\ol 1) \,\longto\, \Koh^1(K, \bar\rho\otimes\bar\rho^\vee)\,.
\]
\end{remark}

\begin{definition}
Let $d$ be a non-negative integer. We define a subfunctor $\Def^{\Box,\leqslant d}_{\bar\rho}$ of  $\Def^\Box_{\bar\rho}$ so that for  $\rho_B\in\Def^\Box_{\bar\rho}(B)$, we have $\rho_B\in\Def^{\Box,\leqslant d}_{\bar\rho}(B)$ if and only if $\rho_B$ admits an effective torsion local shtuka model with height  $\leqslant d$ as a torsion $A_\epsilon[\Gal(K^\sep/K)]$-module; in other words, there exists a torsion local shtuka $\ulHM$ satisfying $(z-\zeta)^d\hat M \subset \tau_{\hat M}(\hat\sigma^*\hat M) \subset \hat M$, equipped with an $A_\epsilon$-linear Galois-equivariant isomorphism $\rho_B\cong\check T_\epsilon\ulHM$. Since the condition defining the subfunctor $\Def^{\Box,\leqslant d}_{\bar\rho}\subset\Def^\Box_{\bar\rho}$ only depends on the equivalence classes of framed deformations (i.e., deformations), we may similarly define a subfunctor $\Def^{\leqslant d}_{\bar\rho}$ of  $\Def_{\bar\rho}$.
\end{definition}

\begin{theorem}[{cf.~\cite[Theorem~4.2.1]{KimNormField}}]\label{ThmDeformation}
Let $K$ be a finite extension of $Q_\epsilon$.
 Then the framed deformation functor $\Def^{\Box,\leqslant d}_{\bar\rho}$ is pro-representable by a complete local noetherian $\BF\dbl z \dbr$-algebra $R^{\Box,\leqslant d}_{\bar\rho}$, and there exists a complete local noetherian $\BF\dbl z \dbr$-algebra  $R^{\leqslant d}_{\bar\rho}$ which is a pro-representable hull of the deformation functor $\Def^{\leqslant d}_{\bar\rho}$  in the sense of Schlessinger \cite[Definition~2.7]{Schlessinger}. If we have $\End_{\Gal(K^\sep/K)}(\bar\rho) = \BF$, then $\Def^{\leqslant d}_{\bar\rho}$ is pro-representable by $R^{\leqslant d}_{\bar\rho}$.
\end{theorem}

\begin{proof}
 The statement is an equi-characteristic analog of \cite[Theorem~1.3]{KimNormField}, and the proof can be easily adapted by working with torsion local shtukas instead of torsion $\varphi$-modules over $\FS$. Let us sketch the main ideas.
 
 Let $F$ be a set-valued functor on  $\Art_{\BF\dbl z\dbr}$ such that $F(\BF)$ is a singleton. 
 By Schlessinger's theorem \cite[Theorem~2.11]{Schlessinger}, the following conditions are equivalent to pro-representability of $F$ by a complete local noetherian $\BF\dbl z\dbr$-algebra:
 \begin{description}
  \item[(H1)] For any surjective map $B'\onto B$ with length-$1$ kernel and a map $B''\rightarrow B$ in $\Art_{\BF\dbl z\dbr}$, the natural map $h:F(B'\times_BB'')\rightarrow F(B')\times_{F(B)}F(B'')$ is surjective.
  \item[(H2)] In the setting of \textbf{(H1)}, $h$ is bijective if $B=\BF$ and $B' = \BF[u]/(u^2)$. (This implies that $F(\BF[u]/(u^2))$ is an $\BF$-vector space; cf.~\cite[Lemma~2.10]{Schlessinger}.)
  \item[(H3)] Under \textbf{(H2)}, the $\BF$-vector space  $F(\BF[u]/(u^2))$ is finite-dimensional.
  \item[(H4)] For any surjective map $B'\onto B$ with length-$1$ kernel, the natural map $h:F(B'\times_BB')\rightarrow F(B')\times_{F(B)}F(B')$ is bijective.
 \end{description}
Furthermore, the existence of a pro-representable hull of $F$ is equivalent to the conditions  \textbf{(H1)}--\textbf{(H3)}.

The proof of the pro-representability result in \cite[\S1.2]{MazurMSRI} actually shows that $\Def^\Box_{\bar\rho}$ and $\Def_{\bar\rho}$ satisfy \textbf{(H1)} and \textbf{(H2)}, and \textbf{(H4)} is satisfied by $\Def^\Box_{\bar\rho}$ for any $\bar\rho$ and by $\Def_{\bar\rho}$ if $\End_{\Gal(K^\sep/K)}(\bar\rho)=\BF$. Since the condition defining the sub-functors $\Def^{\Box,\leqslant d}_{\bar\rho}$ and $\Def^{\leqslant d}_{\bar\rho}$ is stable under finite direct products and subquotients (by Proposition~\ref{PropRamakrishnaCrit}), one can check without difficulty that  $\Def^{\Box,\leqslant d}_{\bar\rho}$ (respectively,  $\Def^{\leqslant d}_{\bar\rho}$) satisfies \textbf{(H1)}, \textbf{(H2)}, or \textbf{(H4)} if and only if the same holds for $\Def^\Box_{\bar\rho}$ (respectively, for $\Def_{\bar\rho}$).

It remains to verify \textbf{(H3)} for $\Def^{\Box,\leqslant d}_{\bar\rho}$ and $\Def^{\leqslant d}_{\bar\rho}$, which is the key step. Indeed, it suffices to show that $\Def^{\leqslant d}_{\bar\rho}(\BF[u]/(u^2))$ is a finite set, since $\Def^{\leqslant d}_{\bar\rho}(\BF[u]/(u^2))$ is the quotient of $\Def^{\Box,\leqslant d}_{\bar\rho}(\BF[u]/(u^2))$ by the natural action of $\widehat\GL_r(\BF[u]/(u^2)):=\ker\bigl[\GL_r(\BF[u]/(u^2))\rightarrow\GL_r(\BF)\bigr]$; \emph{cf.} the definition of deformations as equivalence classes of framed deformations in Definition~\ref{DefDeformation}. The finiteness of $\Def^{\leqslant d}_{\bar\rho}(\BF[u]/(u^2))$ can be obtained by repeating the argument in \cite[\S1.6]{KimNormField}. The main ideas behind the argument can be described roughly as follows:
\begin{enumerate}
 \item 
 Any $\rho\in\Def^{\leqslant d}_{\bar\rho}(\BF[u]/(u^2))$ admits an effective torsion local shtuka model $\ulHM$ with height $\leqslant d$ with an action of  $\BF[u]/(u^2)$, which induces the correct scalar action on $\rho$; indeed, viewing $\rho$ as a torsion $A_\epsilon[\Gal(K^\sep/K)]$-module, the \emph{maximal} torsion local shtuka model $\ulHM^+$ of $\rho$ with height $\leqslant d$ (cf.~Lemma~\ref{LemmaMaxMin}) inherits the action of $\BF[u]/(u^2)$ by Lemma~\ref{LemmaMaxMin}\ref{LemmaMaxMin:Exist}.
 \item 
 Considering only the $\BF$-action on $\ulHM$, it follows that $\ulHM$ is a torsion local $\BF$-shtuka by Lemma~\ref{LemmaBFSht}. Then by the $\BF[u]/(u^2)$ action, $\ulHM$ fits in the following short exact sequence of torsion local $\BF$-shtukas:
 \[
  0\rightarrow \ulHM_1 \rightarrow \ulHM \rightarrow \ulHM_0 \rightarrow 0,
 \]
where $\ulHM_0$ and $\ulHM_1$ are effective torsion local $\BF$-shtuka models of $\bar\rho$ with height $\leqslant d$, and there exists an $\BF$-linear map $u:\ulHM_0\rightarrow \ulHM_1$ of torsion local shtuka models of $\bar\rho$ (induced by the action of $u$).
 \item 
 By Lemma~\ref{LemmaMaxMin}, there are only finitely many choices of $\ulHM_0$ and $\ulHM_1$. (Recall that $K$ is discretely valued.) Therefore, it suffices to show that for any fixed $\ulHM_0$ and $\ulHM_1$ as above there exist only finitely many extensions of $\ulHM_0$ by $\ulHM_1$ as effective torsion local $\BF$-shtukas with height $\leqslant d$. This finiteness assertion can be read off from \cite[\S1.6.9\emph{ff}]{KimNormField}. Note that the argument uses the finiteness of the residue field $k$ of $K$ as the extensions naturally form a $k$-vector space.
\end{enumerate}
This concludes the proof.
\end{proof}

\begin{remark}
One can define a notion of ``$\epsilon$-torsion $A$-motives'' with good reduction at $\epsilon$ (over certain open subschemes $U$ of a finite cover of the curve $C$ from \ref{Notation}) in an analogous way to torsion local shtukas over $R$. It is possible to formulate a global deformation problem (with the height condition at $\epsilon$ analogous to $\Def^{\leqslant d}_{\bar\rho}$) and prove a suitable global analog of Theorem~\ref{ThmDeformation}.
\end{remark}

For what follows, let us assume that $\End_{\Gal(K^\sep/K)}(\bar\rho) = \BF$ so that $R^{\leqslant d}_{\bar\rho}$ pro-represents $D^{\leqslant d}_{\bar\rho}$. 
Since $R^{\leqslant d}_{\bar\rho}$ is a complete local noetherian $\BF\dbl z\dbr$-algebra, it is possible to associate to it a rigid analytic variety $(\Spf R^{\leqslant d}_{\bar\rho})^\rig$ over $\BF\dpl z\dpr$, where  $(\,.\,)^\rig$ denotes the rigid analytic generic fiber. Keeping the analogy between $p$-adic crystalline representations and equi-characteristic crystalline representations, $(\Spf R^{\leqslant d}_{\bar\rho})^\rig$ can be viewed as an equi-characteristic analog of a $p$-adic crystalline deformation space. Now we are going to prove some properties of equi-characteristic crystalline deformation spaces (such as smoothness and the dimension formula), which will strengthen the analogy with $p$-adic crystalline deformation spaces. If $D^{\leqslant d}_{\bar\rho}$ is not pro-representable, then we can work with the framed deformation ring $R^{\Box,\leqslant d}_{\bar\rho}$ instead of $R^{\leqslant d}_{\bar\rho}$.

Let $\rho^{\leqslant d}:\Gal(K^\sep/K)\rightarrow \GL_r(R^{\leqslant d}_{\bar\rho})$ denote (a representative of) the universal deformation (obtained by taking the limit of the  deformations over Artinian quotients of $R^{\leqslant d}_{\bar\rho}$). Then given a finite-dimensional $\BF\dpl z\dpr$-algebra $B$ and a continuous $\BF\dbl z\dbr$-homomorphism $f_B:R^{\leqslant d}_{\bar\rho}\rightarrow B$, we obtain
\[f_B^*(\rho^{\leqslant d}):\Gal(K^\sep/K)\xrightarrow{\rho^{\leqslant d}} \GL_r(R^{\leqslant d}_{\bar\rho})\xrightarrow{f_B}\GL_r(B).\]

Let $E$ be a finite extension of $\BF\dpl z\dpr$ with valuation ring $\CO_E$, and fix a continuous $\BF\dbl z\dbr$-morphism $f_E:R^{\leqslant d}_{\bar\rho}\rightarrow E$. We set $\rho_E:=f_E^*(\rho^{\leqslant d})$. Since $f_E$ factors through $\CO_E$ by compactness, $\rho_E$ comes equipped with a distinguished Galois-stable $\CO_E$-lattice.

\begin{lemma}\label{LemmaGenFibDefor}
 Let $B$ be an Artinian local $\BF\dpl z\dpr$-algebra with residue field $E$, and let $f_B:R^{\leqslant d}_{\bar\rho}\rightarrow B$ be a continuous $\BF\dbl z\dbr$-homomorphism. Then there exists an effective local shtuka $\ulHM$ with $(z-\zeta)^d\hat M\subset\tau_{\hat M}(\hat\sigma^*\hat M)\subset\hat M$, with a $Q_\epsilon[\Gal(K^\sep/K)]$-isomorphism $f_B^*(\rho^{\leqslant d})\cong \check V_\epsilon\ulHM$.
 
 Conversely, let $\rho_B:\Gal(K^\sep/K)\rightarrow \GL_r(B)$ be a lift of $\rho_E$ such that for some effective local shtuka $\ulHM$ with $(z-\zeta)^d\hat M\subset\tau_{\hat M}(\hat\sigma^*\hat M)\subset\hat M$, there exists a $Q_\epsilon[\Gal(K^\sep/K)]$-isomorphism $\rho_B\cong \check V_\epsilon\ulHM$. Then there exists a unique continuous $\BF\dbl z\dbr$-homomorphism $f_B:R^{\leqslant d}_{\bar\rho}\rightarrow B$  lifting $f_E$, such that there exists an isomorphism $\rho_B\cong f_B^*(\rho^{\leqslant d})$ which reduces to the identity map on $\rho_E$.
\end{lemma}

\begin{proof}
 For the first claim, observe that by the usual compactness argument one can find an $\BF\dbl z\dbr$-subalgebra $B_0\subset B$, finitely generated as an $\BF\dbl z\dbr$-module, such that $B_0[\frac{1}{z}] = B$ and $f_B$ factors through $B_0$. By construction, $f_B^*(\rho^{\leqslant d})$ has a distinguished Galois-stable $B_0$-lattice, whose Artinian quotients admit effective torsion local shtuka models with height $\leqslant d$. We now obtain the first claim by applying Lemma~\ref{LemmaLimit}.
 
 For the converse, let $B^\circ\subset B$ denote the preimage of $\CO_E$. Since $\rho_E$ factors through $\GL_r(\CO_E)$, it follows that $\rho_B$ factors through $\GL_r(B^\circ)$. Since $B^\circ$ is a directed union of $\BF\dbl z\dbr$-sub-algebras of $B$ which are finitely generated as $\BF\dbl z\dbr$-modules, the usual compactness argument shows that there exists an $\BF\dbl z\dbr$-subalgebra $B_0\subset B$, finitely generated as a $\BF\dbl z\dbr$-module, such that $B_0[\frac{1}{z}] = B$ and $\rho_B$ is factored by $\rho_{B_0}:\Gal(K^\sep/K)\rightarrow \GL_r(B_0)$.
 
 Now viewing $\rho_B$ as a Galois representation over $Q_\epsilon$ isomorphic to $\check V_\epsilon\ulHM$, its $B_0$-lattice $\rho_{B_0}$ yields an $A_\epsilon$-lattice $T'\subset\check V_\epsilon\ulHM$. By Proposition~\ref{PropLattices} and Lemma~\ref{LemmaIsogBounded} there exists an effective local shtuka $\ulHM'$ isogenous to $\ulHM$ with $(z-\zeta)^d\hat M'\subset\tau_{\hat M'}(\hat\sigma^*\hat M')\subset\hat M'$ with an isomorphism $\rho_{B_0}\cong\check T_\epsilon\ulHM'$. By applying the universal property of $R^{\leqslant d}_{\bar\rho}$ to the mod~$z^n$ reduction of $\rho_{B_0}$ for each $n$, we obtain a continuous $\BF\dbl z\dbr$-morphism $f_{B_0}:R^{\leqslant d}_{\bar\rho}\rightarrow B_0$ giving rise to $\rho_{B_0}$. By composing $f_{B_0}$ with the natural inclusion $B_0\hookrightarrow B$, we obtain the desired $f_B$.
\end{proof}

Let us make a digression to the $z$-isocrystal with Hodge-Pink structure associated to an equi-characteristic crystalline representation ``with coefficients in $B$''. Let $B$ be a finite-dimensional $Q_\epsilon$-algebra, and let  $\rho_B\in\Rep_B\Gal(K^\sep/K)$ be a representation which is equi-characteristic crystalline if viewed as a Galois representation over $Q_\epsilon$, that is $\rho_B\cong\check V_\epsilon\ulHM$ as $Q_\epsilon[\Gal(K^\sep/K)]$-modules for a local shtuka $\ulHM$ over $R$. By full faithfulness of $\check V_\epsilon$ (Theorem~\ref{ThmAndersonKim}), there exists a local shtuka $\ulHM:=(\hat M,\tau_{\hat M})$, unique up to quasi-isogeny, such that  $B$ acts on $\ulHM$ by quasi-morphism and we have a $B[\Gal(K^\sep/K)]$-isomorphism $\rho_B\cong \check V_\epsilon\ulHM$.

\begin{lemma}\label{LemmaLSCoeff}
In the above setting, $\hat M[\frac{1}{z}]$ is free over $R\dbl z\dbr_B:=R\dbl z\dbr\otimes_{A_\epsilon}B$ with rank equal to $\rk_B\rho_B$.
\end{lemma}
\begin{proof}
One can repeat the proof of \cite[Proposition~1.6.1]{KisinPST} via Fitting ideals, with $\FS$ and $E(u)$ replaced with $R\dbl z\dbr$ and $z-\zeta$.
\end{proof}

\begin{corollary}\label{CorLSCoeff}
Let $\rho_B$ and $\ulHM$ be as above. We define $\BH(\rho_B):=\BH(\ulHM):=(D,\tau_D,\Fq_D)$, where we give the unique rigidification to $\ulHM$; cf.~Lemma~\ref{LemmaGL}. Then the admissible $z$-isocrystal with Hodge-Pink structure $\BH(\rho_B)$ enjoys the following additional properties:
\begin{enumerate}
\item
$D$ is free over $k\dpl z\dpr_B:=k\dpl z\dpr\otimes_{Q_\epsilon}B$ and $\tau_D$ is $B$-linear.
\item
$\Fq_D\subset D\otimes_{k\dpl z\dpr}K\dpl z-\zeta\dpr$ is a finitely generated free submodule over $K\dbl z-\zeta\dbr_B:=K\dbl z-\zeta\dbr\otimes_{Q_\epsilon}B$. 
\end{enumerate}
\end{corollary}
\begin{proof}
This corollary follows directly from Lemma~\ref{LemmaLSCoeff} and the construction of $D$ and $\Fq_D$; cf.~\eqref{EqMystFunct}.
\end{proof}

 Assume that $\End_{\Gal(K^\sep/K)}(\bar\rho) = \BF$, and fix $f_E:R^{\leqslant d}_{\bar\rho}\rightarrow E$  for some finite extension $E$ of $\BF\dpl z \dpr$. Set  $\rho_E:=f_E^*(\rho^{\leqslant d})$ and $\BH(\rho_E):=(D_E,\tau_{D_E},\Fq_{D_E})$, and write 
 \[
 \BH(\rho_E)\otimes\BH(\rho_E)^\vee= \left(\Ad(D_E),\tau_{\Ad(D_E)},\Fq_{\Ad(D_E)}\right),
 \]
 with $\Fp_{\Ad(D_E)}$ denoting the tautological lattice; cf.~Definition~\ref{DefHPStruct}. Identifying $\Ad(D_E) \otimes_{K\dbl z\dbr_E}K\dpl z-\zeta\dpr_E$ with $\End_{K\dpl z-\zeta\dpr_E}\bigl(D_E\otimes_{K\dpl z\dpr_E}K\dpl z-\zeta\dpr_E\bigr)$ we can describe  $\Fp_{\Ad(D_E)}$ and $\Fq_{\Ad(D_E)}$ as follows:
 \[
  \Fp_{\Ad(D_E)} = \End_{K\dbl z-\zeta\dbr_E}\bigl(\Fp_{D_E})\quad\text{and}\quad
  \Fq_{\Ad(D_E)} = \End_{K\dbl z-\zeta\dbr_E}\bigl(\Fq_{D_E}\bigr);  
 \]
 in other words, as a $K\dbl z-\zeta\dbr_E$-lattice of $\End_{K\dpl z-\zeta\dpr_E}\bigl(D_E\otimes_{K\dpl z\dpr_E}K\dpl z-\zeta\dpr_E\bigr)$, $\Fp_{\Ad(D_E)}$ (respectively, $\Fq_{\Ad(D_E)}$) consists of endomorphisms preserving $\Fp_{D_E}$ (respectively, $\Fq_{D_E}$).
 
\begin{theorem}\label{ThmGenSm}
In the above setting, the completion of $R^{\leqslant d}_{\bar\rho}\otimes_{\BF\dbl z\dbr}E$ with respect to the kernel of $f_E\otimes 1$ is isomorphic to a formal power series ring over $E$ of dimension 
\[
1+\dim_E \frac{\Fp_{\Ad(D_E)}}{\Fp_{\Ad(D_E)}\cap\Fq_{\Ad(D_E)}}
\]
(This is a finite number since $K$ is a finite extension of $Q_\epsilon$.)
In particular, $(\Spf R^{\leqslant d}_{\bar\rho})^\rig$ is a smooth rigid analytic variety over $\BF\dpl z \dpr$.
 \end{theorem}

 \begin{remark}
If we do not assume that $\End_{\Gal(K^\sep/K)}(\bar\rho) = \BF$, one can still show that $(\Spf R^{\Box,\leqslant d}_{\bar\rho})^\rig$ is a smooth rigid analytic variety over $\BF\dpl z \dpr$, and its dimension at an $E$-point $f_E:R^{\Box,\leqslant d}_{\bar\rho}\rightarrow E$ is 
\[r^2 +\dim_E \frac{\Fp_{\Ad(D_E)}}{\Fp_{\Ad(D_E)}\cap\Fq_{\Ad(D_E)}},\] 
with the obvious notation, where $r=\dim_\BF\bar\rho = \dim_E\rho_E$. This formula is compatible with Theorem~\ref{ThmGenSm} if $\End_{\Gal(K^\sep/K)}(\bar\rho) = \BF$, since the natural morphism $\Spf R^{\Box,\leqslant d}_{\bar\rho} \rightarrow \Spf R^{\leqslant d}_{\bar\rho}$ is formally smooth of relative dimension $r^2-1$. (In fact, formal smoothness is clear, and to obtain the relative dimension it suffices to look at the tangent spaces. Since we have $\End_{\Gal(K^\sep/K)}(\bar\rho) = \BF$, the stabilizer of the natural $\wh\GL_r(\BF[u]/(u^2))$-action on $\Def^{\Box,\leqslant d}_{\bar\rho}(\BF[u]/(u^2))$ is the subgroup of scalar matrices $\wh{\mathbb G}_m(\BF[u]/(u^2))$, which shows that $\Def^{\Box,\leqslant d}_{\bar\rho}(\BF[u]/(u^2))$ is a $\wh\PGL_r(\BF[u]/(u^2))$-torsor over $\Def^{\leqslant d}_{\bar\rho}(\BF[u]/(u^2))$. Since  $\wh\PGL_r(\BF[u]/(u^2))\cong\mathfrak{pgl}_r(\BF)$, where the latter is an $(r^2-1)$-dimensional $\BF$-vector space, we obtain the desired numerology.)

This result should be thought of as an equi-characteristic analog of the formal smoothness and the dimension formula for the (rigid analytic) generic fibers of crystalline deformation rings; cf.~the case with $N=0$ of \cite[\S3]{KisinPST}. Note also that in the dimension formula, we have terms involving the Hodge-Pink structure associated to $\rho_E$ instead of the Hodge-Pink filtration, which supports that the Hodge-Pink structure is the right equi-characteristic analog of the Hodge filtration associated to a $p$-adic crystalline representation. (Compare with \cite[Theorems~3.3.4, 3.3.8]{KisinPST}.)
 \end{remark}

\begin{example}\label{ExDimGemFib}
 Assume that $\BF=\BF_\epsilon$ so that $R\dbl z\dbr_\BF = R$. Consider a local shtuka $\ulHM$ where $\hat M = R\dbl z\dbr^2$ and $\tau_{\hat M}$ is given by the matrix $\left(\begin{smallmatrix}1&1\\0&(z-\zeta)^d\end{smallmatrix}\right)$. Then the isomorphism $\delta_{\hat M}$ from Lemma~\ref{LemmaGL} equals $\left(\begin{smallmatrix}1\es&g\\0\es&\tminus^{-d}\end{smallmatrix}\right)$ for some $g\in R\dbl z,z^{-1}\}[\tminus^{-1}]$ and so $\Fq_D=\Fp_D+(z-\zeta)^{-d}\left(\begin{smallmatrix} \hat\sigma(g)-1 \\ \hat\sigma(\tminus)^{-d} \end{smallmatrix}\right)K\dbl z-\zeta\dbr$. One can check that $\bar\rho:=\check T_\epsilon\ulHM\otimes_{A_\epsilon}\BF_\epsilon$ is a non-split extension of  $(\chi_\epsilon)^{-d}\bmod z$ by the trivial character. Therefore, we have  $\End_{\Gal(K^\sep/K)}(\bar\rho) = \BF$ if and only if $\hat q-1$ does not divide $d$, which we assume; cf.~Example~\ref{ExTorLSModel}.
 
Clearly, $\check T_\epsilon\ulHM$ defines an $A_\epsilon$-point  $f:R^{\leqslant d}_{\bar\rho}\rightarrow A_\epsilon$. By Theorem~\ref{ThmGenSm}, this implies that the dimension of $(\Spf R^{\leqslant d}_{\bar\rho})^\rig$ at $f$ is given by
 \[
  1 + \dim_{Q_\epsilon} \left(K\dbl z-\zeta\dbr/(z-\zeta)^d\right) = 1 + d\cdot [K:Q_\epsilon],
 \]
because $\End_{K\dbl z-\zeta\dbr}(\Fp_{D})=K\dbl z-\zeta\dbr^{2\times2}$ and $\End_{K\dbl z-\zeta\dbr}(\Fp_{D})\cap\End_{K\dbl z-\zeta\dbr}(\Fq_{D})=$
\[
\es\bigl\{\,\left(\begin{smallmatrix} a & b \\ c & d \end{smallmatrix}\right)\in K\dbl z-\zeta\dbr^{2\times2}\colon b\equiv c(\hat\sigma(g)-1)^2\hat\sigma(\tminus)^{2d}+(d-a)(\hat\sigma(g)-1)\hat\sigma(\tminus)^d\mod(z-\zeta)^d\,\bigr\}.
\]
\end{example}

\bigskip

Let us begin the proof of Theorem~\ref{ThmGenSm}.
Let $B$ be an Artinian local $E$-algebra with residue field $E$, and let $I\subset B$ be an ideal annihilated by the maximal ideal of $B$. We fix a lift $f_{B/I}:R^{\leqslant d}_{\bar\rho}\rightarrow  B/I$ of $f_E$, and set $\rho_{B/I}:=f_{B/I}^*(\rho^{\leqslant d})$.

In order to study a lift of $f_{B/I}$ to a $B$-point, we introduce the following complex concentrated in degrees $0$ and $1$:
\begin{equation}
 \CC^\bullet(\rho_E):= \left[\Ad(D_E) \xrightarrow{(\tau_{D_E}-1,incl)} \Ad(D_E) \oplus \frac{\Fp_{\Ad(D_E)}}{\Fp_{\Ad(D_E)}\cap\Fq_{\Ad(D_E)}}\right].
\end{equation}
Let $\CH^0(\rho_E)$ and $\CH^1(\rho_E)$ respectively denote the $0$th and the $1$st cohomology of $\CC^\bullet(\rho_E)$. Clearly, $\CH^0(\rho_E)$ is naturally isomorphic to the $E$-vector space of $E$-linear endomorphisms of $\BH(\rho_E)$, which is isomorphic to $\End_{\Gal(K^\sep/K)}(\rho_E)$ by full faithfulness results (Theorem~\ref{ThmAndersonKim}, Proposition~\ref{PropBHfullyFaithful}). 
\begin{lemma}\label{LemmaGenForSm}
The set of maps $f_B:R^{\leqslant d}_{\bar\rho}\rightarrow B$ lifting $f_{B/I}$ is a principal homogeneous space under $\CH^1(\rho_E)\otimes_E I$; in particular, there exists a lift $f_B$ of $f_{B/I}$.
\end{lemma}
\begin{proof}
By filtering $I$ if necessary, it suffices to show the lemma when $I$ is $1$-dimensional as an $E$-vector space, which we assume from now on. Let $u\in I$ denote a principal generator.

Let us first show the existence of a lift $f_B$.
 Write $\BH(\rho_{B/I})=(\olD,\tau_{\olD},\ol\Fq)$. Since $\olD$ is free over $k\dpl z\dpr_{B/I}$ by Corollary~\ref{CorLSCoeff}, we pick a free $k\dpl z\dpr_B$-module $D$ that lifts $\olD$ and lift $\tau_{\olD}$ arbitrarily to $\tau_D$. Note that $(D,\tau_D)$ is  a $z$-isocrystal. Since $\ol\Fq$ is free over $K\dbl z-\zeta\dbr_{B/I}$ by Corollary~\ref{CorLSCoeff}, we choose a free $K\dbl z-\zeta\dbr_B$-module $\Fq$ lifting $\ol\Fq$ and an injective map $\Fq\hookrightarrow (z-\zeta)^{-d}\Fp$ lifting the natural inclusion $\ol\Fq\hookrightarrow (z-\zeta)^{-d}\ol\Fp$. (Here, $\ol\Fp$ and $\Fp$ respectively denote the tautological lattices for $\olD$ and $D$.) We set $\ulD:=(D, \tau_D, \Fq)$.
 
 Let us first show that $\ulD$ is \emph{admissible}. Indeed, choosing an $E$-basis $u\in I$, we have an isomorphism $\BH(\rho_E)\isoto \ulD\otimes_B I$ (with the obvious notation), so we have the following short exact sequence of $z$-isocrystals with Hodge-Pink structure:
 \[
  0\rightarrow \BH(\rho_E) \rightarrow \ulD \rightarrow \BH(\rho_{B/I}) \rightarrow 0.
 \]
Now, one can show that any extension of admissible $z$-isocrystals with Hodge-Pink structure is again admissible, using the characterization of admissibility in terms of the associated $\hat\sigma$-bundle $\CF(\ulD)$ (Theorem~\ref{ThmCritWA}\ref{ThmCritWA_A}) and the classification of $\hat\sigma$-bundles (Theorem~\ref{ThmHartlPink1}\ref{ThmHartlPink1a}).

Let $\ulHM$ denote a local shtuka with $\BH(\ulHM)\cong\ulD$, which exists by admissibility. By construction, all the Hodge-Pink weights of $\ulHM$ are between $0$ and $d$, so $\ulHM$ is effective with $(z-\zeta)^d\hat M\subset\tau_{\hat M}(\hat\sigma^*\hat M)\subset\hat M$. By full faithfulness (cf.~Lemma~\ref{LemmaGL} and Proposition~\ref{PropBHfullyFaithful}), $B$ acts on $\ulHM$ by quasi-morphism, so $\check V_\epsilon(\ulHM)$ has a natural $B$-action commuting with the Galois action. On the other hand,  $\check V_\epsilon(\ulHM)$ is a free $B$-module lifting $\rho_{B/I}$ by comparing the $E$-dimension and the minimal number of generators obtained by the Nakayama lemma. By Lemma~\ref{LemmaGenFibDefor} it follows that $\check V_\epsilon(\ulHM)$ defines a $B$-point of $R^{\leqslant d}_{\bar\rho}$ lifting $f_{B/I}$.

Let us now show that the set of $f_B$ lifting $f_{B/I}$ is a principal homogeneous space under $\CH^1(\rho_E)\cong \CH^1(\rho_E)\otimes_E I$. Let $\ulD$ be as above. Given $\gamma\in\Ad(D_E)\cong\End_{k\dpl z\dpr_E}(D_E)$ and $\delta\in\Fp_{\Ad(D_E)}\cong \End_{K\dbl z-\zeta\dbr_E}(\Fp_{D_E})$, we obtain another $z$-isocrystal with Hodge-Pink structure as follows:
 \[
 \ulD_{(\gamma,\delta)}:=(D,(1+u\gamma)\tau_D, (1+u\delta)\Fq).
 \]
 By repeating the previous argument for $\ulD_{(\gamma,\delta)}$ instead of $\ulD$, it follows that $\ulD_{(\gamma,\delta)}$ defines a $B$-lift of $f_{B/I}$, thus we obtain an $E$-linear \emph{transitive} action of $\Ad(D_E)\oplus\Fp_{\Ad(D_E)}$ on the set of $B$-lifts of $f_{B/I}$. (Transitivity follows from Corollary~\ref{CorLSCoeff}.)
 One can easily check that $\ulD_{(\gamma,\delta)}$ and $\ulD_{(\gamma',\delta')}$ define the same $B$-lift of $f_{B/I}$ if and only if $(\gamma,\delta)$ and $(\gamma',\delta')$ define the same class in $\CH^1(\rho_E)$. This shows the desired claim.
\end{proof}

\begin{proof}[Proof of Theorem~\ref{ThmGenSm}] 
Lemma~\ref{LemmaGenForSm} shows that the completion of $R^{\leqslant d}_{\bar\rho}\otimes_{\BF\dbl z\dbr}E$ at the maximal ideal corresponding to $f_E$ is a formal power series ring of dimension equal to $\dim_E\CH^1(\rho_E)$. It remains to compute $\dim_E\CH^1(\rho_E)$. Note that
\[
 \dim_E\CH^0(\rho_E) - \dim_E\CH^1(\rho_E)= \dim_E\Ad(D_E) - \left(\dim_E\Ad(D_E) + \dim_E \frac{\Fp_{\Ad(D_E)}}{\Fp_{\Ad(D_E)}\cap\Fq_{\Ad(D_E)}}\right) ,
\]
and $\CH^0(\rho_E)\cong \End_{\Gal(K^\sep/K)}(\rho_E)$ is $1$-dimensional since $\End_{\Gal(K^\sep/K)}(\rho_\BF)\cong\BF$. 
This shows that $\dim_E\CH^1(\rho_E)$ has the expected dimension. 

By construction of rigid analytic generic fibers (\emph{cf.} \cite[\S7.1]{deJongDieudonne}), $(\Spf R^{\leqslant d}_{\bar\rho})^\rig$ is the direct union of affinoid rigid analytic spaces  $\Spm B_n$, where $B_n$ is the completion of $R^{\leqslant d}_{\bar\rho}\otimes_{\BF\dbl z\dbr}\BF\dpl z\dpr$ with respect to the norm such that $R^{\leqslant d}_{\bar\rho}[\frac{\Fm^n}{z}]$ is the set of norm $\leqslant 1$ elements in $R^{\leqslant d}_{\bar\rho}\otimes_{\BF\dbl z\dbr}\BF\dpl z\dpr$. Since $\{\Spm B_n\}$ is an admissible covering of   $(\Spf R^{\leqslant d}_{\bar\rho})^\rig$, it remains to show that each $B_n$ is a topologically smooth affinoid algebra over $\BF\dpl z\dpr$. Since $B_n$ is Jacobson \cite[\S5.2.6, Theorem~3]{BGR} and  all ideals of $B_n$ are closed \cite[\S5.2.7, Corollary~2]{BGR}, it suffices to show that the completion of $B_n$ at any closed maximal ideal (i.e., classical points of $\Spm B_n$) is geometrically regular. (Note that the invertibility of Jacobian can be captured by the completions at a dense set of prime ideals.) 

By \cite[Lemma~7.1.9]{deJongDieudonne}, any (closed) maximal ideal of $B_n$ restricts to a maximal ideal of  $R^{\leqslant d}_{\bar\rho}\otimes_{\BF\dbl z\dbr}\BF\dpl z\dpr$, and the completion of $B_n$ at a maximal ideal is isomorphic to the completion of $R^{\leqslant d}_{\bar\rho}\otimes_{\BF\dbl z\dbr}\BF\dpl z\dpr$ at the corresponding maximal ideal. Since Lemma~\ref{LemmaGenForSm} shows that the local ring of $R^{\leqslant d}_{\bar\rho}\otimes_{\BF\dbl z\dbr}\BF\dpl z\dpr$ at any maximal ideal is geometrically regular, we obtain the desired smoothness claim.
\end{proof}

We end the section with discussing the equi-characteristic analog of moduli of finite flat group schemes \cite{KisinModuliFF}. We first need the  following definition:
\begin{definition}
 Let $B$ be an $A_\epsilon$-algebra with $\#B<\infty$ (not a finite $Q_\epsilon$-algebra), and let $T_B$ be a finitely generated free $B$-module equipped with a discrete action of $\Gal(K^\sep/K)$. Then a \emph{torsion local $B$-shtuka model} of $T_B$ is a torsion local $B$-shtuka $\ulHM_B$ equipped with a $B[\Gal(K^\sep/K)]$-isomorphism $T_B\cong\check T_\epsilon \ulHM_B$.
 
 We take the obvious notion of equivalence for torsion local $B$-shtuka models of $T_B$. We can understand an equivalence class of torsion local $B$-shtuka models of $T_B$ as a certain $R\dbl z\dbr_B$-lattice of $T_B\otimes_{A_\epsilon}K^\sep\dbl z\dbr$ stable under $1\otimes\hat\sigma$ and invariant under the Galois action, using the isomorphism \eqref{EqTateModIsom}.
\end{definition}

\begin{theorem}[{\cite[Proposition~4.3.1]{KimNormField}}]
We fix $\bar\rho:\Gal(K^\sep/K)\rightarrow \GL_d(\BF)$, and for simplicity we assume that $\End_{\Gal(K^\sep/K)}(\bar\rho)\cong\BF$.
 Then there exists a projective scheme $\mathcal{GR}^{\leqslant d}_{\bar\rho}$ over $\Spec R^{\leqslant d}_{\bar\rho}$ with the following property. For any local $R^{\leqslant d}_{\bar\rho}$-algebra $B$ with $\#B<\infty$,  $\Hom_{R^{\leqslant d}_{\bar\rho}}(\Spec B,\mathcal{GR}^{\leqslant d}_{\bar\rho})$ is functorially bijective with the set of equivalence classes of torsion local $B$-shtuka models of $\rho^{\leqslant d}\otimes_{R^{\leqslant d}_{\bar\rho}}B$ which are effective and of height $\leqslant d$.
\end{theorem}
\begin{proof}
 Indeed, $\mathcal{GR}^{\leqslant d}_{\bar\rho}$ can be represented by a closed subscheme of a certain affine Grassmannian for $\GL_r$, following the same idea as Kisin's construction of moduli of finite flat group schemes; cf.~\cite[Proposition~2.1.10]{KisinModuliFF}.
\end{proof}

\begin{remark}
 In the $p$-adic setting, Kisin's construction of moduli of finite flat group schemes over flat deformation rings  was ``globalized'' by Pappas and Rapoport; cf.~\cite{PRCoeffSp}. Namely, the mod $p$ local Galois representation $\bar\rho$ is allowed to vary in the moduli space constructed by Pappas and Rapoport, and we recover moduli of finite flat group schemes by ``fixing $\bar\rho$''.  Such ``coefficient spaces'' (which can be thought of as fixing a base of $p$-divisible groups and varying coefficients) bear somewhat striking similarities with moduli spaces of $p$-divisible groups (with fixed coefficients $\BZ_p$ and varying its base), commonly known as Rapoport-Zink spaces. For example, there exists a natural ``period-morphisms'' \cite[\S5]{PRCoeffSp} on the rigid analytic generic fiber of a Pappas-Rapoport coefficient space whose image was computed by Hellmann and the first author; see \cite[Theorem 7.8]{Hellmann13} and \cite[Corollary~6.12]{HartlHellmann}.
 
 By repeating the construction in \cite{PRCoeffSp} or \cite{HartlHellmann} by working with $A_\epsilon$ instead of $\BZ_p$, one can obtain an equi-characteristic analog of Pappas-Rapoport coefficient spaces, which can roughly be thought of as the moduli space of local shtukas with fixed base $R$ and varying coefficients. Such equi-characteristic coefficient spaces could be interesting objects to study; for example, one can ask about a description of the image of the rigid analytic period morphism, analogous to \cite{Hellmann13,HartlHellmann}.
\end{remark}

%
%

\vfill

\begin{minipage}[t]{0.5\linewidth}
\noindent
Urs Hartl\\
Universit\"at M\"unster\\
Mathematisches Institut \\
Einsteinstr.~62\\
D -- 48149 M\"unster
\\ Germany
\\[1mm]
\href{http://www.math.uni-muenster.de/u/urs.hartl/}{\small www.math.uni-muenster.de/u/urs.hartl/}
\end{minipage}
\begin{minipage}[t]{0.45\linewidth}
\noindent
Wansu Kim\\
KAIST \\
Department of Mathematical Sciences\\
291 Daehak-ro, Yuseong-gu\\
Daejeon, 34141\\
Republic of Korea
\\[1mm]
{\small \href{https://sites.google.com/site/wansukimmaths/}{sites.google.com/site/wansukimmaths/}}
\end{minipage}

\end{document}